\documentclass[12pt]{amsart}
\usepackage[margin=1in]{geometry}  
\usepackage{graphicx}              
\usepackage{amsmath}               
\usepackage{amsfonts}              
\usepackage{amsthm}                
\usepackage{caption} 
\usepackage{subcaption}
\usepackage{tikz}
\usepackage{tabu}
\usepackage[shortlabels]{enumitem}
\usepackage{bm}
\usepackage{amssymb} 
\usepackage{tikz}
\usetikzlibrary{cd}
\DeclareMathOperator{\Aut}{Aut}
\usepackage{ulem}

\newtheorem{thm}{Theorem}[section]
\newtheorem{lem}[thm]{Lemma}

\newtheorem{prop}[thm]{Proposition}
\newtheorem{cor}[thm]{Corollary}

\theoremstyle{definition}
\newtheorem{definition}[thm]{Definition}
\newtheorem{remark}[thm]{Remark}
\newtheorem{question}[thm]{Question}
\newtheorem{example}[thm]{Example}

\newcommand{\ZZ}{\mathbb{Z}}      

\newcommand{\QQ}{\mathbb{Q}}

\begin{document}

\title{Classification of torus bundles that bound rational homology circles}
\author{Jonathan Simone}

\begin{abstract}
In this article, we completely classify orientable torus bundles over the circle that bound smooth 4-manifolds with the rational homology of the circle. Along the way, we classify certain integral surgeries along chain links that bound rational homology 4-balls and explore a connection to 3-braid closures whose double branched covers bound rational homology 4-balls. 
\end{abstract}

\maketitle 

\section{Introduction}\label{intro}
In \cite{simone3}, two infinite families of $T^2$-bundles over $S^1$ are shown to bound (smooth) rational homology circles ($\QQ S^1\times B^3s$). As an application, the $\QQ S^1\times B^3s$ were used to construct infinite families of rational homology 3-spheres ($\QQ S^3s$) that bound rational homology 4-balls ($\QQ B^4s$). The main purpose of this article is to show that the two families of torus bundles used in \cite{simone3} are the only torus bundles that bound smooth $\QQ S^1\times B^3s$.
 
After endowing $T^2\times[0,1]=\mathbb{R}^2/\mathbb{Z}^2\times[0,1]$ with the coordinates $(\textbf{x},t)=(x,y,t)$, any orientable torus bundle over $S^1$ is of the form $T^2\times[0,1]/(\textbf{x},1)\sim(\pm A\textbf{x},0)$, where $A\in SL(2,\ZZ)$. The matrix $A$ is called the \textit{monodromy} of the torus bundle and is defined up to conjugation. Throughout, we will express the monodromy in terms of the generators $T=\begin{bmatrix} 1&1\\0&1\end{bmatrix}$ and $S=\begin{bmatrix} 0&1\\-1&0\end{bmatrix}$. A torus bundle is called \textit{elliptic} if $|\text{tr}A|<2$, \textit{parabolic} if $|\text{tr}A|=2$, and \textit{hyperbolic} if $|\text{tr}A|>2$. Moreover, a torus bundle is called \textit{positive} if tr$A>0$ and \textit{negative} if tr$A<0$. Torus bundles naturally arise as the boundaries of plumbings of $D^2$-bundles over $S^2$ (see Section 6 in \cite{neumann} for details). Using these plumbing descriptions, it is easy to draw surgery diagrams for torus bundles. Figure \ref{torusbundles} gives a complete list of torus bundles over $S^1$, along with their monodromies (up to conjugation) and surgery diagrams. To simplify notation, $\textbf{T}_{\pm A(\textbf{a})}$ will always denote the hyperbolic torus bundle with monodromy $\pm A(\textbf{a})=\pm T^{-a_1}S\cdots T^{-a_n}S$, where $\textbf{a}=(a_1,\ldots,a_n)$, $a_1\ge3$, and $a_i\ge 2$ for all $i$.

\begin{figure}[h!]
\centering
 {\tabulinesep=1.5mm
 \begin{tabu}{|c@{\hskip.5in}c|c@{\hskip.5in}c|}
 \multicolumn{4}{c}{\textbf{Elliptic Torus Bundles}}\\\hline
\underline{Monodromy} & \underline{Surgery Diagram} & \underline{Monodromy} & \underline{Surgery Diagram} \\
\shortstack[c]{$ S$\\\hspace{1in}\\\hspace{1in}\\\hspace{1in}\\\hspace{1in}\\\hspace{1in}\\\hspace{1in}\\\hspace{1in}}  & \includegraphics[scale=.5]{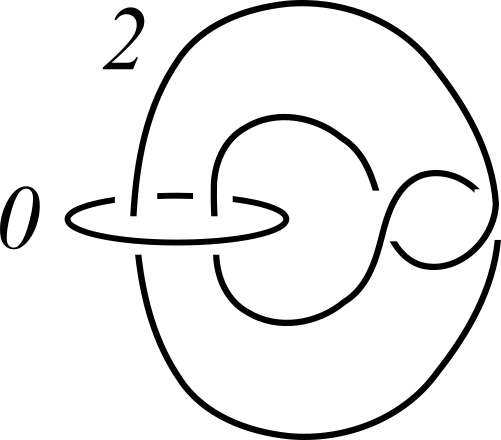} & \shortstack[c]{$-S$\\\hspace{1in}\\\hspace{1in}\\\hspace{1in}\\\hspace{1in}\\\hspace{1in}\\\hspace{1in}\\\hspace{1in}}  & \includegraphics[scale=.5]{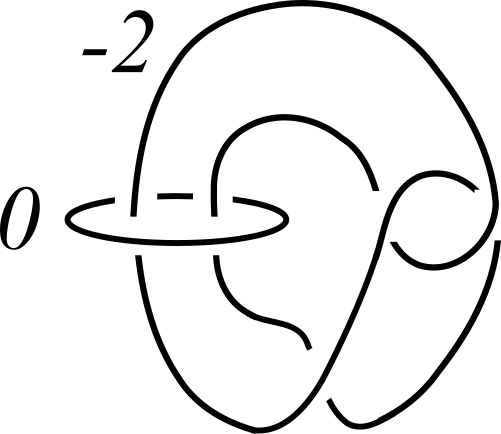} \\\hline
\shortstack[c]{$ T^{-1}S$\\\hspace{1in}\\\hspace{1in}\\\hspace{1in}\\\hspace{1in}\\\hspace{1in}\\\hspace{1in}\\\hspace{1in}} & \includegraphics[scale=.5]{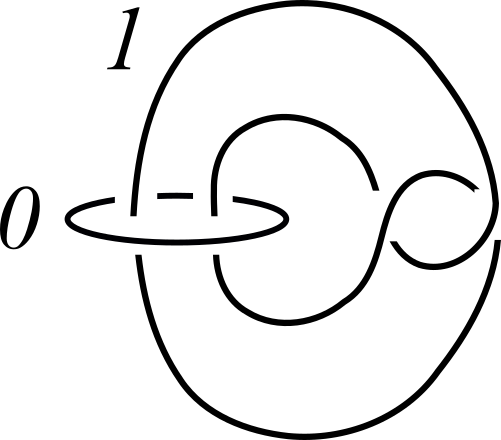} & \shortstack[c]{$- T^{-1}S$\\\hspace{1in}\\\hspace{1in}\\\hspace{1in}\\\hspace{1in}\\\hspace{1in}\\\hspace{1in}\\\hspace{1in}} & \includegraphics[scale=.5]{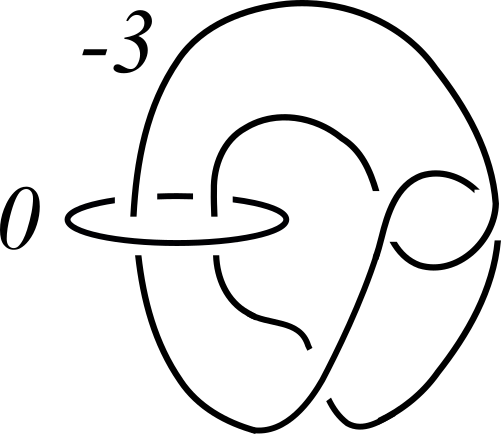}  \\\hline
 \shortstack[c]{$ (T^{-1}S)^2$\\\hspace{1in}\\\hspace{1in}\\\hspace{1in}\\\hspace{1in}\\\hspace{1in}\\\hspace{1in}\\\hspace{1in}} & \includegraphics[scale=.5]{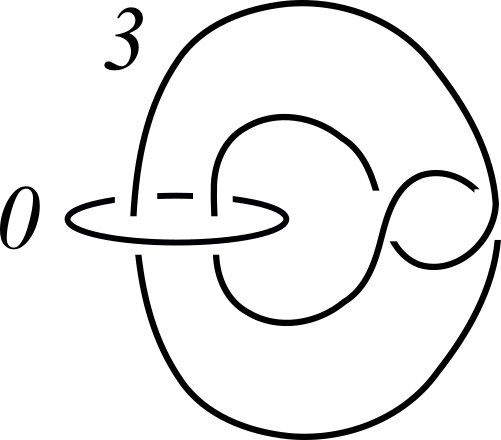} &  \shortstack[c]{$- (T^{-1}S)^2$\\\hspace{1in}\\\hspace{1in}\\\hspace{1in}\\\hspace{1in}\\\hspace{1in}\\\hspace{1in}\\\hspace{1in}} & \includegraphics[scale=.5]{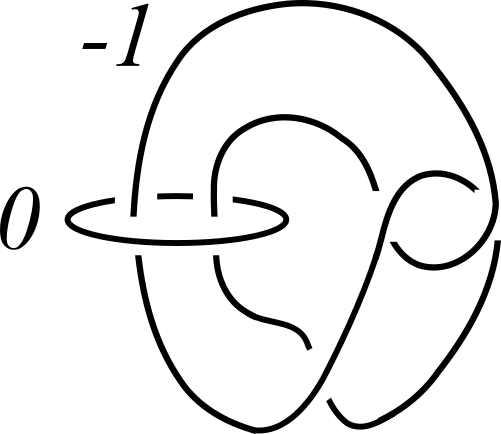}\\\hline
\end{tabu}}

\vspace{.2in}

{\tabulinesep=1.2mm
\begin{tabu}{|c@{\hskip.5in}c|c@{\hskip.5in}c|}
\multicolumn{4}{c}{\textbf{Parabolic Torus Bundles}}\\\hline
\underline{Monodromy} & \underline{Surgery Diagram} & \underline{Monodromy} & \underline{Surgery Diagram}\\
\shortstack[c]{$T^n$\\ $(n\in\ZZ)$\\\hspace{1in}\\\hspace{1in}\\\hspace{1in}\\\hspace{1in}\\\hspace{1in}} &  \includegraphics[scale=.5]{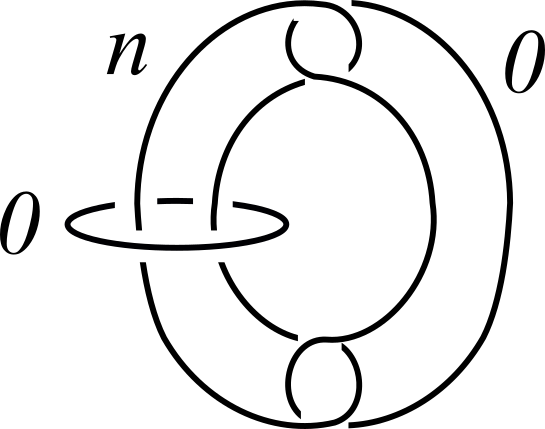} & \shortstack[c]{$-T^n$\\ $(n\in\ZZ)$\\\hspace{1in}\\\hspace{1in}\\\hspace{1in}\\\hspace{1in}\\\hspace{1in}} &  \includegraphics[scale=.5]{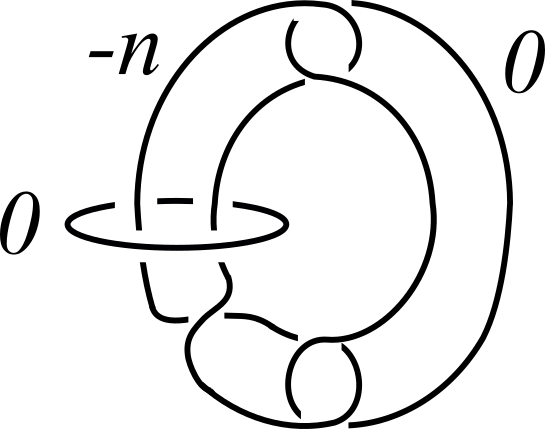}\\\hline
\end{tabu}}

\vspace{.2in}

{\tabulinesep=1.2mm
 \begin{tabu}{|c@{\hskip.5in}c|}
\multicolumn{2}{c}{\textbf{Hyperbolic Torus Bundles $\textbf{T}_{\pm A(a_1,\ldots,a_n)}$}}\\\hline
\underline{Monodromy} & {\underline{Surgery Diagram}}\\
\shortstack[c]{$T^{-a_1}S\cdots T^{-a_n}S$ \\ $(a_1\ge 3, a_i\ge2\,\forall\, i)$\\\hspace{1in}\\\hspace{1in}\\\hspace{1in}\\\hspace{1in}\\\hspace{1in}\\\hspace{1in}\\\hspace{1in}\\\hspace{1in}\\\hspace{1in}\\\hspace{1in}\\\hspace{1in}\\\hspace{1in}\\\hspace{1in} } & \includegraphics[scale=.5]{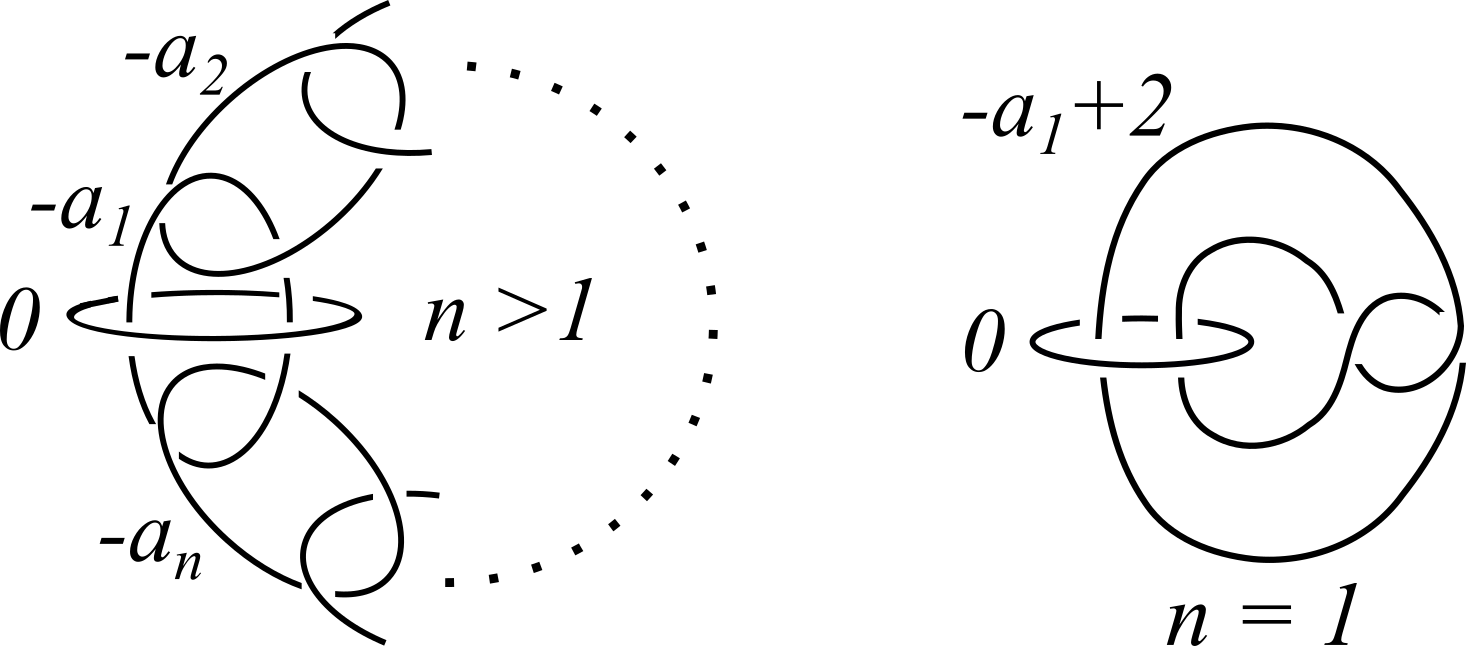}\\\hline
\shortstack[c]{$-T^{-a_1}S\cdots T^{-a_n}S$ \\ $(a_1\ge 3, a_i\ge2\,\forall\, i)$\\\hspace{1in}\\\hspace{1in}\\\hspace{1in}\\\hspace{1in}\\\hspace{1in}\\\hspace{1in}\\\hspace{1in}\\\hspace{1in}\\\hspace{1in}\\\hspace{1in}\\\hspace{1in}\\\hspace{1in}\\\hspace{1in} } & \includegraphics[scale=.5]{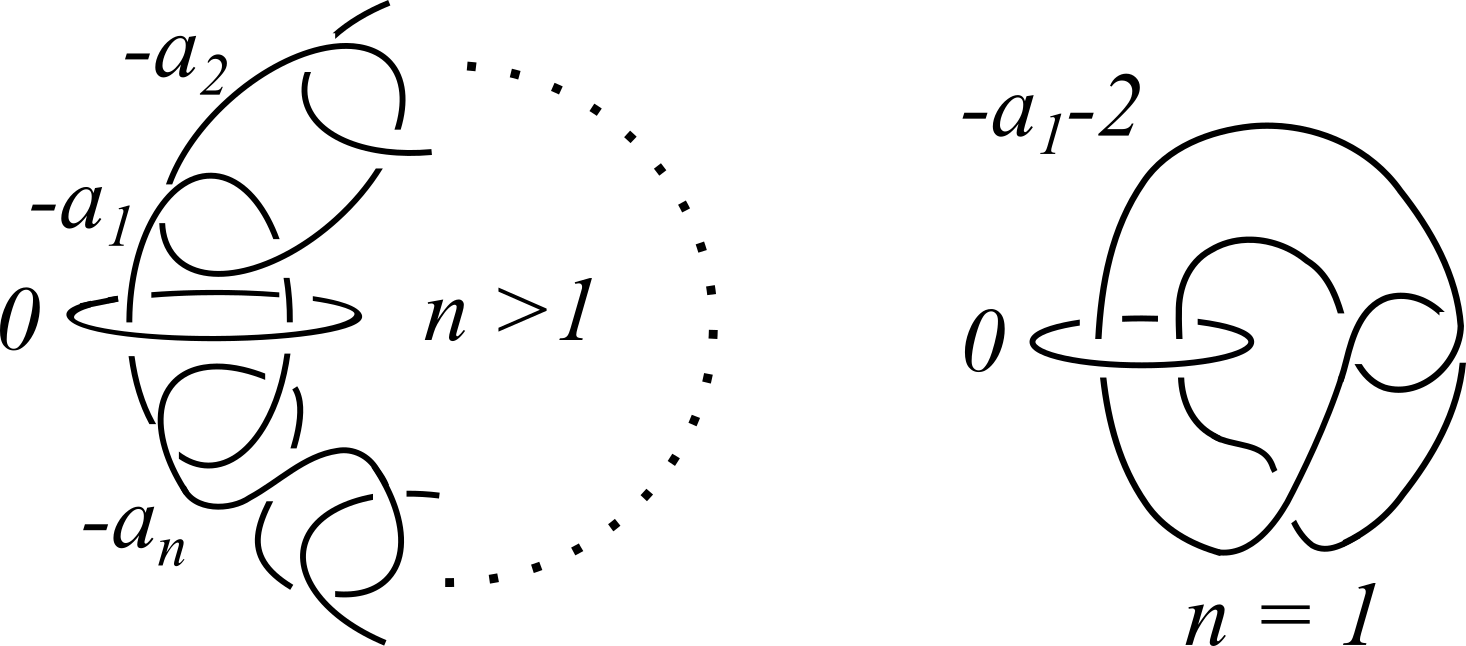}\\\hline
\end{tabu}}
\caption{Mondromy and surgery diagrams of $T^2$-bundles over $S^1$}\label{torusbundles}
\end{figure}

\begin{thm} A torus bundle over $S^1$ bounds a $\QQ S^1\times B^3$ if and only if
	\begin{itemize}
		\item it is negative parabolic or
		\item it is positive hyperbolic of the form $\textbf{T}_{A(\textbf{a})}$, where\\ $\textbf{a}=(3+x_1,2^{[x_2]},\ldots,3+x_{2m+1},2^{[x_1]}, 3+x_2,2^{[x_3]},\ldots,3+x_{2m},2^{[x_{2m+1}]})$, \\
		$m\ge0$, and $x_i\ge0$ for all $i$.
	\end{itemize}
\label{mainthm}\end{thm}

Elliptic torus bundles and parabolic torus bundles that bound $\QQ S^1\times B^3s$ are rather simple to classify. Classifying hyperbolic torus bundles, which make up the ``generic" class of torus bundles, is much more involved and includes the bulk of the technical work. In \cite{simone3}, it is shown that $\textbf{T}_{A(\textbf{a})}$ indeed bounds a $\QQ S^1\times B^3$ when $\textbf{a}=(3+x_1,2^{[x_2]},\ldots,3+x_{2m+1},2^{[x_1]}, 3+x_2,2^{[x_3]},\ldots,3+x_{2m},2^{[x_{2m+1}]})$. To obstruct all other hyperbolic torus bundles from bounding $\QQ S^1\times B^3s$, we first consider a related class of $\QQ S^3s$.

Let $L_n^t$ denote the $n$-component link shown in Figure \ref{chainlink}, where $t$ denotes the number of half-twists. We call $L_n^t$ the \textit{n-component, t-half-twisted chain link}. If $t=0$, we call the chain link \textit{untwisted}. Consider the surgery diagram for the hyperbolic torus bundle $\textbf{T}_{\pm A(\textbf{a})}$ given in Figure \ref{torusbundles}. Now, perform $m$-surgery along a meridian of the 0-framed unknot as in the left side of each of the four diagrams in Figure \ref{chainlinks}. Next, slide the unknot with framing $-a_1$ (or $-a_1\pm2$) twice over the blue $m$-framed unknot so that it no longer passes through the $0$-framed unknot. Then cancel the $0$-framed and $m$-framed unknots. When $n\ge2$, the resulting 3-manifolds are obtained by $(-a_1,\ldots,-a_n)$-surgery along the chain link $L_n^t$, where $t=2m$ or $2m-1$. We denote these 3-manifolds by $Y_{\textbf{a}}^t=S^3_{(-a_1,\ldots,-a_n)}(L_n^t)$, where $\textbf{a}=(a_1,\ldots,a_n)$, $a_1\ge 3$, and $a_i\ge 2$ for all $i$. Note that by cyclically reordering or reversing the surgery coefficients, we obtain the same 3-manifold.  When $n=1$, the resulting 3-manifolds are obtained by $-(a_1\pm 2)$-surgery along $L_1^t$, where $t=2m+(1\pm1)$; we denote them by  $Y_{\textbf{a}}^t=Y_{(a_1)}^t$. Note that $Y_{(a_1)}^t=S^3_{-a_1+2}(L_1^t)$ when $t$ is even, and $Y_{(a_1)}^t=S^3_{-a_1-2}(L_1^t)$ when $t$ is odd. Finally note that $Y^t_{\textbf{a}}$ is a $\QQ S^3$ for all $\textbf{a}$ and $t$; this follows from the fact that $|H_1(Y^t_{\textbf{a}})|=|\text{Tor}(H_1(\textbf{T}_{\pm A(\textbf{a})}))|$ is finite (see Lemma \ref{order} in the Appendix).

 \begin{lem}[\cite{simone3}] Let $Y$ be a $\mathbb{Q}S^1\times S^2$ that bounds a $\mathbb{Q}S^1\times B^3$, and let $K$ be a knot in $Y$ such that $[K]$ has infinite order in $H_1(Y;\mathbb{Z})$. Then any integer surgery on $Y$ along $K$ yields a $\QQ S^3$ that bounds a $\QQ B^4$. \label{rationalballobstruction}\end{lem}

By Lemma \ref{rationalballobstruction}, if $\textbf{T}_{A(\textbf{a})}$ bounds a $\QQ S^1\times B^3$, then $Y_{\textbf{a}}^{t}$ bounds a $\QQ B^4$ for all even $t$, and if $\textbf{T}_{-A(\textbf{a})}$ bounds a $\QQ S^1\times B^3$, then $Y_{\textbf{a}}^{t}$ bounds a $\QQ B^4$ for all odd $t$. Thus if $Y_{\textbf{a}}^t$ does not bound a $\QQ B^4$ for some even (or odd) $t$, then $\textbf{T}_{A(\textbf{a})}$ (or $\textbf{T}_{-A(\textbf{a})}$) does not bound a $\QQ S^1\times B^3$. Using this fact, we will obstruct most hyperbolic torus bundles from bounding $\QQ S^1\times B^3s$ by identifying the strings $\textbf{a}$ for which $Y_{\textbf{a}}^0$ and $Y_{\textbf{a}}^{-1}$ do not bound $\QQ B^4s$. Before writing down the result, we first recall and introduce some useful terminology.

Let $(b_1,\ldots,b_k)$ be a string of integers such that $b_i\ge2$ for all $i$. If $b_j\ge 3$ for some $j$, then we can write this string in the form $(2^{[m_1]},3+n_1,\ldots,2^{[m_j]},2+n_j)$, where $m_i,n_i\ge 0$ for all $i$ and $(\dots,2^{[t]},\ldots)$ denotes $(\dots,\overbrace{2,\ldots,2}^t,\ldots)$. The string $(c_1,\ldots,c_l)=(2+m_1, 2^{[n_1]},3+m_2,\ldots,3+m_j,2^{[n_j]})$ is called the \textit{linear-dual string} of $(b_1,\ldots,b_k)$. If $b_i=2$ for all $1\le i\le k$, then we define its linear-dual string to be $(k+1)$. Linear-dual strings have a topological interpretation. If $Y$ is obtained by $(-b_1,\ldots,-b_k)-$surgery along a linear chain of unknots, then the reversed-orientation manifold $\overline{Y}$ can be obtained by $(-c_1,\ldots,-c_l)-$surgery along a linear chain of unknots (see Theorem 7.3 in Neumann \cite{neumann}). Finally, we define the linear-dual string of $(1)$ to be the empty string.

\begin{figure}[t!]
	\centering
	\includegraphics[scale=.6]{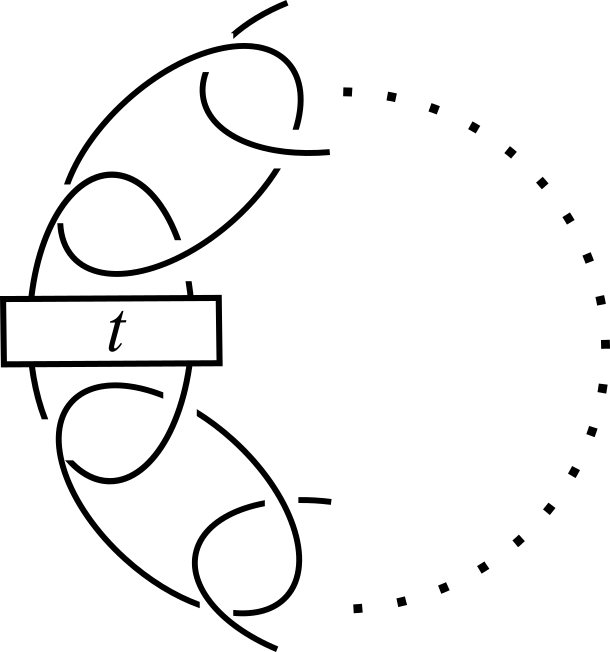}
	\caption{The $n$-component, $t$-half-twisted chain link, $L_n^t$. The box labeled ``$t$" denotes $t$ half-twists.}
	\label{chainlink}
\end{figure}

\begin{figure}
	\centering
	\includegraphics[scale=.55]{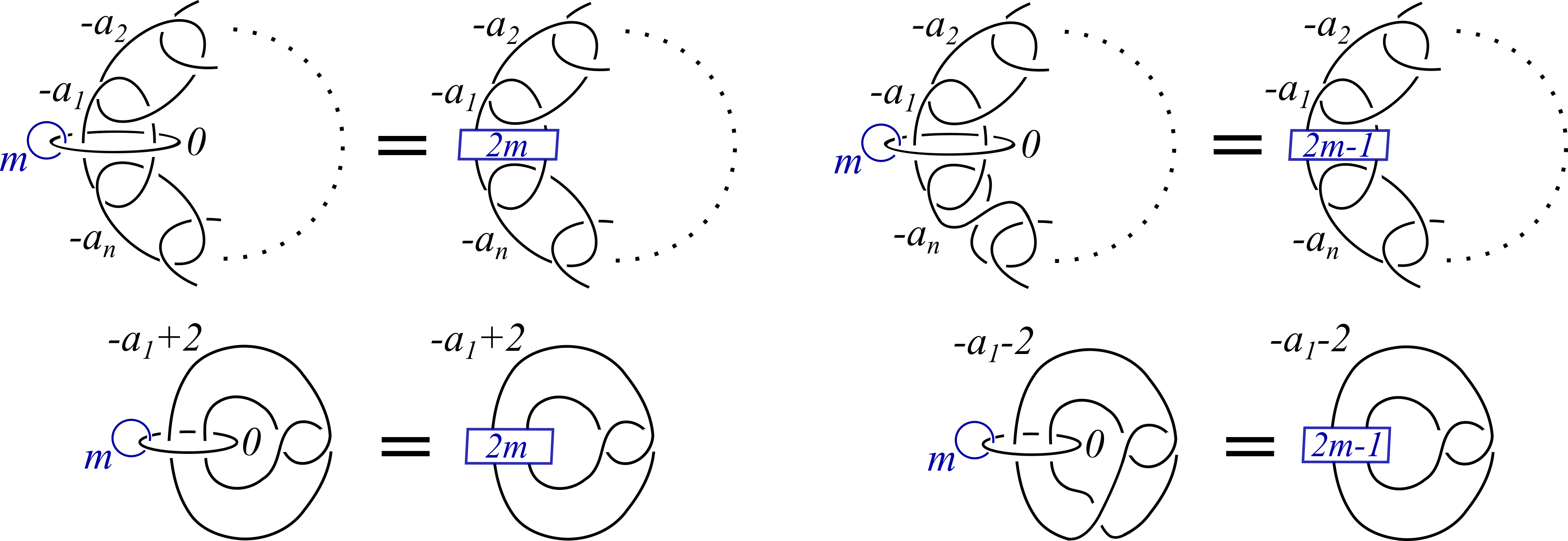}
	\caption{Surgering the hyperbolic torus bundle $\textbf{T}_{\pm A(\textbf{a})}$, where $\textbf{a}=(a_1,\ldots,a_n)$ to obtain the rational homology sphere $Y_{\textbf{a}}^t$. The blue boxes labeled $2m$ and $2m-1$ indicate the number of half-twists.}\label{chainlinks}
\end{figure}

Suppose $\textbf{a}=(a_1,\ldots,a_n)$ is of the form $(2^{[m_1]},3+n_1,\ldots,2^{[m_j]},3+n_j)$, where $m_i,n_i\ge 0$ for all $i$, we define its \textit{cyclic-dual} to be the string $\textbf{d}=(d_1,\ldots,d_m)=(3+m_1,2^{[n_1]},\ldots,3+m_j,2^{[n_j]})$. In particular, a string of the form $(x)$ with $x\ge 3$ has cyclic-dual $(2^{[x-3]},3)$. Notice that this definition only slightly differs from the definition of the linear-dual string. As a topological interpretation of cyclic-dual strings, one has that the reversed-orientation of $\textbf{T}_{\pm A(\textbf{a})}$ is given by  $\overline{\textbf{T}}_{\pm A(\textbf{a})}=\textbf{T}_{\pm A(\textbf{d})}$ (see Theorem 7.3 in Neumann \cite{neumann}). Finally, $(a_n,\ldots,a_1)$ is called the $\textit{reverse}$ of $(a_1,\ldots,a_n)$.

\begin{example} Consider the strings in Theorem \ref{thm1}: $$\textbf{a}=(3+x_1,2^{[x_2]},\ldots,3+x_{2m+1},2^{[x_1]}, 3+x_2,2^{[x_3]},\ldots,3+x_{2m},2^{[x_{2m+1}]}).$$
 It is easy to see that the cyclic-dual of $\textbf{a}$ is simply $\textbf{a}$. Moreover, $\textbf{a}$ is of the above form if and only if it can be expressed in the form $\textbf{a}=(b_1+1,b_2,\ldots,b_{k-1},b_k+1,c_1,\ldots,c_l)$ if $k\ge 2$, where $(b_1,\ldots,b_k)$ and $(c_1,\ldots,c_l)$ are linear-dual strings, or $\textbf{a}=(b_1+2,2^{[b_1-1]})$ if $k=1$.
\label{2crelabel}\end{example}

To remove the necessity of multiple cases, from now on, if $\textbf{a}$ contains a substring of the form $(b_1+1,b_2,\ldots,b_{k-1},b_k+1)$ and $k=1$, then we will understand this substring to simply be $(b_1+2)$, as in Example \ref{2crelabel}.

\begin{definition} Two strings are considered to be equivalent if one is a cyclic reordering and/or reverse of the other. Each string in each of the following sets is defined up to this equivalence. Moreover, in the following sets, strings of the form $(b_1,\ldots,b_k)$ and $(c_1,\ldots,c_l)$ are assumed to be linear-dual.
	\begin{itemize} 
		\item $ \mathcal{S}_{1a}=\{(b_1,\ldots,b_k,2,c_l,\ldots,c_1,2)\,|\, k+l\ge3\}$
	\item $ \mathcal{S}_{1b}=\{(b_1,\ldots,b_k,2,c_l,\ldots,c_1,5)\,|\, k+l\ge2\}$
	\item $ \mathcal{S}_{1c}=\{(b_1,\ldots,b_k,3,c_l,\ldots,c_1,3)\,|\, k+l\ge2\}$
	\item $ \mathcal{S}_{1d}=\{(2,b_1+1,b_2,\ldots,b_{k-1},b_k+1,2,2,c_l+1,c_{l-1},\ldots,c_2,c_1+1,2)\,|\, k+l\ge3\}$
	\item $ \mathcal{S}_{1e}=\{(2, 3+x, 2, 3, 3, 2^{[x-1]},3,3)\,|\, x\ge0 \textup{ and } (3,2^{[-1]},3):=(4)\}$
	\item $ \mathcal{S}_{2a}=\{(b_1+3,b_2,\ldots,b_k,2,c_l,\ldots,c_1)\}$
	\item $ \mathcal{S}_{2b}=\{(3+x,b_1,\ldots,b_{k-1},b_k+1,2^{[x]},c_l+1,c_{l-1},\ldots,c_1)\,|\, x\ge0\text{ and } k+l\ge2\}$
	\item $\mathcal{S}_{2c}=\{(b_1+1,b_2,\ldots,b_{k-1},b_k+1,c_1,\ldots,c_l)\,|\,k+l\ge2\}$
	\item $ \mathcal{S}_{2d}=\{(2,2+x,2,3,2^{[x-1]},3,4)\,|\, x\ge0 \textup{ and } (3,2^{[-1]},3):=(4)\}$
	\item $ \mathcal{S}_{2e}=\{(2,b_1+1,b_2,\ldots,b_k,2,c_l,\ldots,c_2,c_1+1,2),(2,2,2,3)\,|\, k+l\ge3\}$
	\item $\mathcal{O}=\{(6,2,2,2,6,2,2,2), (4,2,4,2,4,2,4,2), (3,3,3,3,3,3)\}$
		\end{itemize}
	Moreover, $\mathcal{S}_1=\mathcal{S}_{1a}\cup\mathcal{S}_{1b}\cup\mathcal{S}_{1c}\cup\mathcal{S}_{1d}\cup\mathcal{S}_{1e}$, $\mathcal{S}_2=\mathcal{S}_{2a}\cup\mathcal{S}_{2b}\cup\mathcal{S}_{2c}\cup\mathcal{S}_{2d}\cup\mathcal{S}_{2e}$, and $\mathcal{S}=\mathcal{S}_1\cup\mathcal{S}_2$.
	\label{definition}
\end{definition}

\begin{definition} Let $\textbf{a}=(a_1,\ldots,a_n)$, where $a_i\ge 2$ for all $i$. Define $I(\textbf{a})$ to be the integer $I(\textbf{a})=\sum_{i=1}^n(a_i-3)$.\label{def:I}\end{definition}

\begin{remark} If $\textbf{b}$ and $\textbf{c}$ are linear-dual strings, then it is easy to see that $I(\textbf{b})+I(\textbf{c})=-2$. Using this observation, it easy to check that if $\textbf{a}\in\mathcal{S}_1$, then $-4\le I(\textbf{a})\le-1$ and if $\textbf{a}\in\mathcal{S}_2$, then $-3\le I(\textbf{a})\le0$. In the same vein, if $\textbf{a}$ and $\textbf{d}$ are cyclic-dual strings, then $I(\textbf{a})+I(\textbf{d})=0$. Consequently, if $\textbf{a},\textbf{d}\in\mathcal{S}$, then $I(\textbf{a})=I(\textbf{d})=0$. Moreover, $\textbf{a}\in\mathcal{S}$ and $I(\textbf{a})=0$ if and only if $\textbf{a}\in\mathcal{S}_{2a}\cup\mathcal{S}_{2b}\cup\mathcal{S}_{2c}$.\label{ibounds}\end{remark}

\begin{thm} Let $\textbf{a}=(a_1,\ldots,a_n)$, where $n\ge1$, $a_i\ge 2$ for all $i$, and $a_j\ge 3$ for some $j$, and let $\textbf{d}$ be the cyclic-dual of $\textbf{a}$.
\begin{enumerate} 
	\item Suppose $\textbf{d}\notin\mathcal{S}_{1a}\cup\mathcal{O}$. Then $Y_{\textbf{a}}^{-1}$ bounds a $\QQ B^4$ if and only if $\textbf{a}\in\mathcal{S}_1$ or $\textbf{d}\in\mathcal{S}_{1b}\cup\mathcal{S}_{1c}\cup\mathcal{S}_{1d}\cup\mathcal{S}_{1e}$.\label{(1)}
	\item Suppose $\textbf{a}\notin\mathcal{S}_{1a}\cup\mathcal{O}$. Then $Y_{\textbf{a}}^{1}$ bounds a $\QQ B^4$ if and only if $\textbf{d}\in\mathcal{S}_1$ or $\textbf{a}\in\mathcal{S}_{1b}\cup\mathcal{S}_{1c}\cup\mathcal{S}_{1d}\cup\mathcal{S}_{1e}$.\label{(2)}
	\item $Y_{\textbf{a}}^0$ bounds a $\QQ B^4$ if and only if $\textbf{a}\in \mathcal{S}_2$ or $\textbf{d}\in \mathcal{S}_2$.\label{(3)}
\end{enumerate}
\label{thm1}\end{thm}

\begin{remark} The hypothesis ``$a_j\ge3$ for some $j$" in Theorem \ref{thm1} ensures that $\textbf{T}_{\pm A(\textbf{a})}$ is a hyperbolic torus bundle. If we remove this condition from the theorem, then we would have an additional case: $a_i=2$ for all $i$. In this case, $Y_{\textbf{a}}^{-1}$ bounds a $\QQ B^4$ and $Y_{\textbf{a}}^0$ does not bound a $\QQ B^4$. This follows from Lemma \ref{rationalballobstruction} and Theorem \ref{mainthm}, and the fact that the corresponding torus bundles are the parabolic torus bundles with respective monodromies $-T^n$ and $T^n$ (c.f. \cite{simone3}).\label{parremark} \end{remark}

\begin{remark} We will see in Section \ref{spheres} (Lemma \ref{noboundlem}) that for certain strings $\textbf{d}$ that are the cyclic-duals of $(b_1,\ldots,b_k,2,c_l,\ldots,c_1,2)$, $Y_{\textbf{d}}^{-1}$ does not bound a $\QQ B^4$ (c.f. Theorem \ref{thm1}(\ref{(1)})). However, we are unable to prove this fact for all such strings. Moreover, for each $\textbf{a}\in\mathcal{O}$, we are unable to obstruct $Y^{\pm1}_{\textbf{a}}$ from bounding a $\QQ B^4$ or show that it indeed bounds a $\QQ B^4$. These strings are outliers that are unobstructed by the analysis presented in this paper.\end{remark}

Combined with Lemma \ref{rationalballobstruction}, Theorem \ref{thm1} obstructs most hyperbolic torus bundles from bounding $\QQ S^1\times B^3s$. In Section \ref{tbundles}, we will obstruct the rest by considering certain cyclic covers of $\QQ S^1\times B^3s$. The proof of Theorem \ref{thm1} relies on Donaldson's Diagonalization Theorem \cite{donaldson} and lattice analysis. From this analysis, it follows that if $\textbf{a}\notin\mathcal{S}_1\cup\mathcal{O}$, then $Y^t_{\textbf{a}}$ does not bound a $\QQ B^4$ for all odd $t$ and if $\textbf{a}\notin\mathcal{S}_2$, then $Y^t_{\textbf{a}}$ does not bound a $\QQ B^4$ for all even $t$. Moreover, by Lemma \ref{rationalballobstruction} and Theorem \ref{mainthm}, if $\textbf{a}\in\mathcal{S}_{2c}$, then $Y^t_{\textbf{a}}$ bounds a $\QQ B^4$ for all even $t$. This leads to the following question.

\begin{question} For what values of $t$ and for which strings $\textbf{a}\in\mathcal{S}\setminus\mathcal{S}_{2c}$ does $Y^t_{\textbf{a}}$ bound a $\QQ B^4$?\label{questionballs}\end{question}


\subsection{Connection to 3-braids}\label{3braids} There is an intimate connection between the rational homology 3-spheres $Y_{\textbf{a}}^t$ and 3-braid closures; we will show that $Y_{\textbf{a}}^t$ is the double cover of $S^3$ branched over the link given by the closure of the 3-braid word $(\sigma_1\sigma_2)^{3t}\sigma_1\sigma_2^{-(a_1-2)}\cdots\sigma_1\sigma_2^{-(a_n-2)}$, where $\sigma_1$ and $\sigma_2$ are the standard generators of the braid group on three strands. 

Let $\textbf{a}=(a_1,\ldots,a_n)$ and consider $Y^{-1}_{\textbf{a}}$ and $Y^0_{\textbf{a}}$, as shown in the top of Figure \ref{dbcbraid}. Using the techniques in \cite{akbulutkirby}, it is clear that $Y^{-1}_{\textbf{a}}$ and $Y^0_{\textbf{a}}$ are the double covers of $S^3$ branched over the links shown in the middle of Figure \ref{dbcbraid}. The $\ZZ_2$-action inducing these covers are the $180^{\circ}$ rotations shown at the top of Figure \ref{dbcbraid}. By isotoping these links, we obtain the closures of the 3-braids words $(\sigma_1\sigma_2)^{-3}\sigma_1\sigma_2^{-(a_1-2)}\cdots\sigma_1\sigma_2^{-(a_n-2)}$ and $\sigma_1\sigma_2^{-(a_1-2)}\cdots\sigma_1\sigma_2^{-(a_n-2)}$, respectively, as shown at the bottom of Figure \ref{dbcbraid}. Note that, in the figure, the blue box labeled $t$ indicates the number of full-twists, while all other boxes indicate the number of half-twists. 

Using Kirby calculus, we can argue that for any $t$, $Y_{\textbf{a}}^{t}$ is the double cover of $S^3$ branched over the closure of the 3-braid word $(\sigma_1\sigma_2)^{3t}\sigma_1\sigma_2^{-(a_1-2)}\cdots\sigma_1\sigma_2^{-(a_n-2)}$. Notice that if $t=2m-1\ge -1$ is odd, then $Y_{\textbf{a}}^{t}$ can be realized as $(-1^{[m]})$-surgery along a link in $Y_{\textbf{a}}^{-1}$, as shown in the top of Figure \ref{oddt}, and if $t=2m\ge 0$ is even, then $Y_{\textbf{a}}^{t}$ can be realized as $(-1^{[m]})$-surgery along a link in $Y_{\textbf{a}}^{0}$, as shown in the top of Figure \ref{event}. Under the $\ZZ_2$-action, each of these surgery curves double covers a curve isotopic to the braid axis of the 3-braid. Thus each $-1$-surgery curve maps to a $-1/2$-surgery curve isotopic to the braid axis, as shown in the middle of Figures \ref{oddt} and \ref{event}. By blowing down these curves, we obtain the desired 3-braid closures at the bottom of the figures. Note that the same argument can be used when $t<-1$; the only difference is that the surgery curves would all have positive coefficients.

\begin{figure}
\centering
\includegraphics[scale=.525]{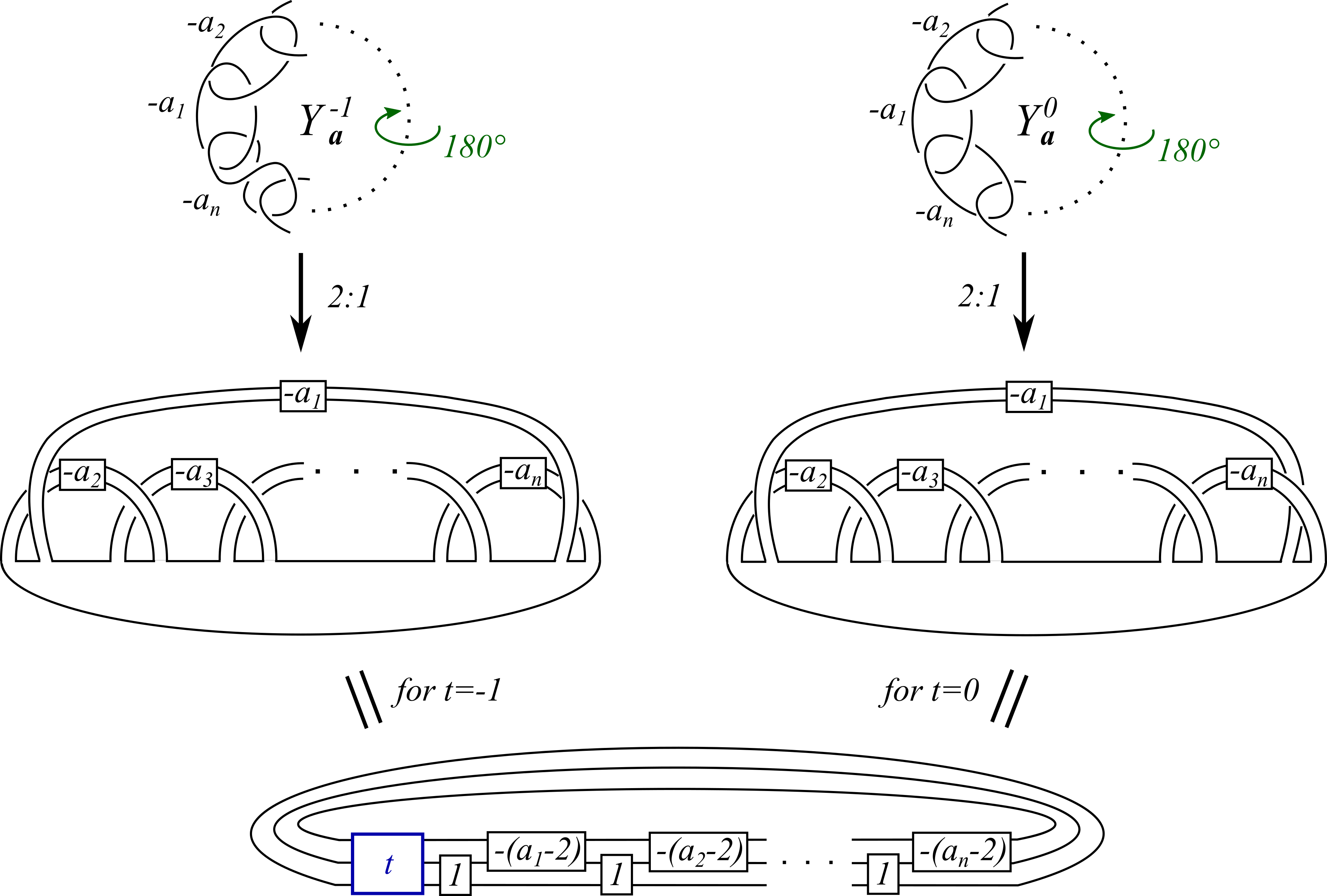}
\caption{$Y^{-1}_{\textbf{a}}$ and $Y^{0}_{\textbf{a}}$ are the double covers of $S^3$ branched over the closure of the 3-braid word $(\sigma_1\sigma_2)^{3t}\sigma_1\sigma_2^{-(a_1-2)}\cdots\sigma_1\sigma_2^{-(a_n-2)}$, where $t=-1$ and $t=0$, respectively. The blue box labeled t indicates the number of full-twists, while all other boxes in all other diagrams indicated the number of half-twists.}\label{dbcbraid}
\end{figure}

\begin{figure}
\centering
\begin{subfigure}{.48\textwidth}
	\centering
	\includegraphics[scale=.51]{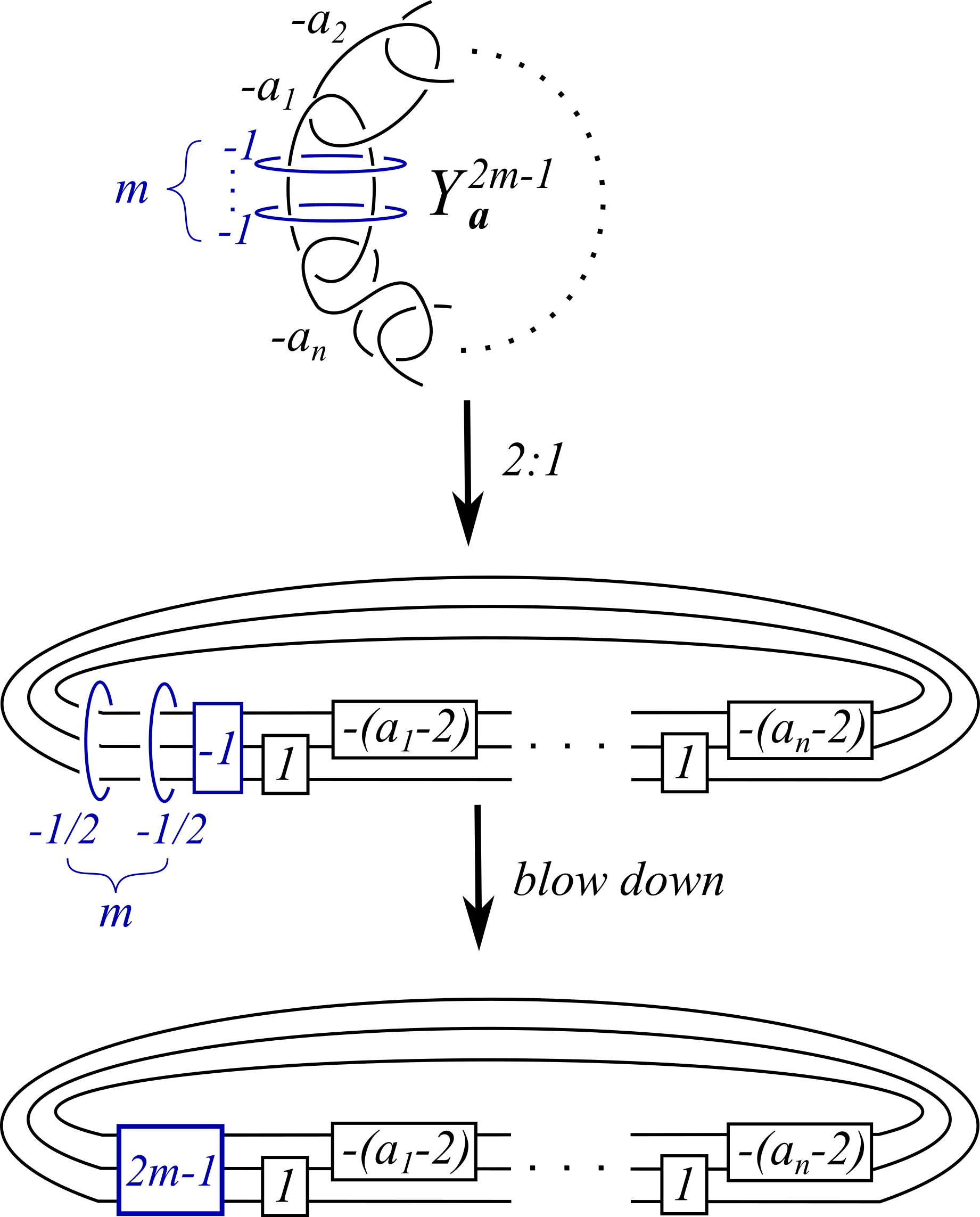}
	\caption{}\label{oddt}
\end{subfigure}
\begin{subfigure}{.48\textwidth}
	\centering
	\includegraphics[scale=.51]{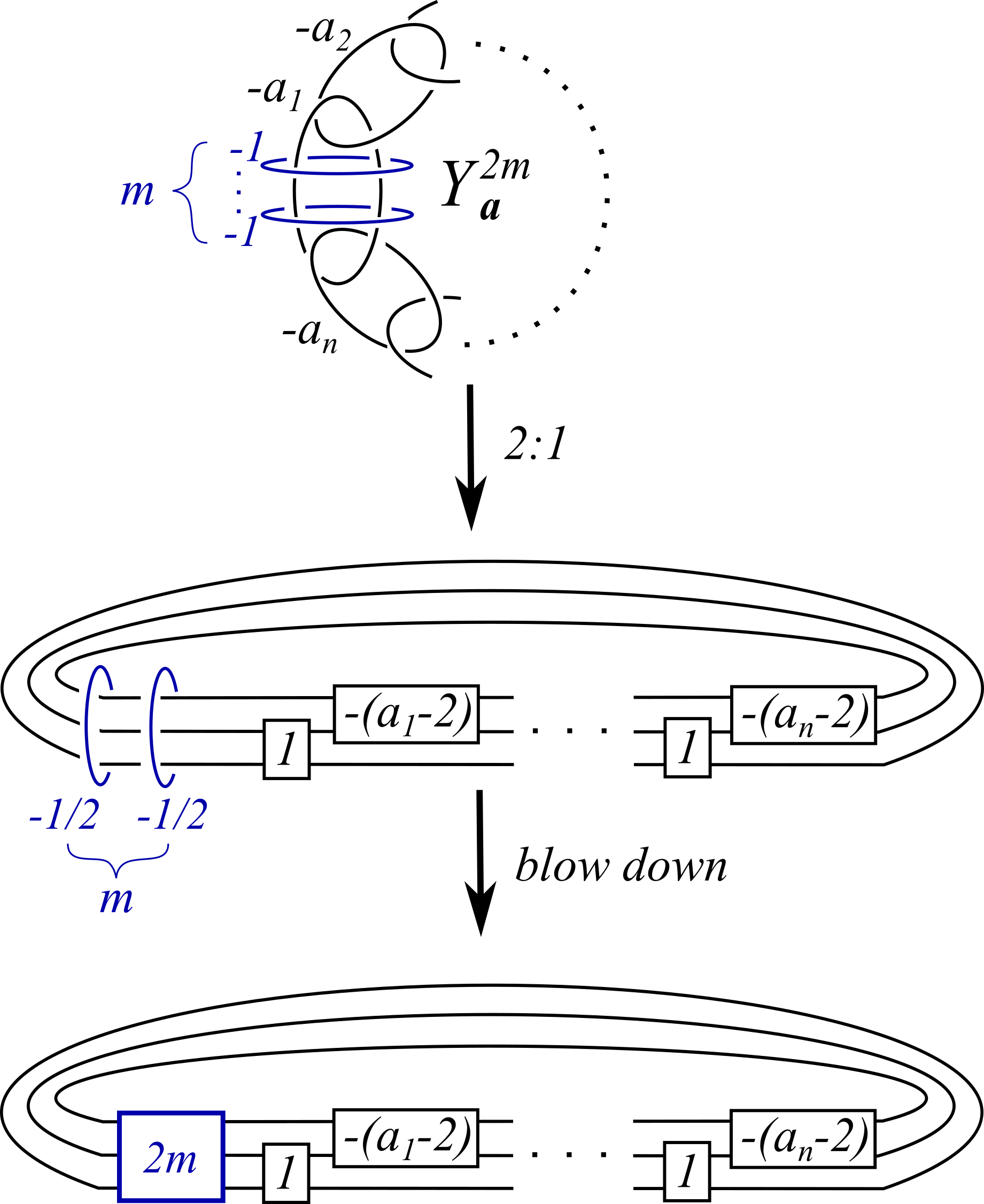}
	\caption{}\label{event}
\end{subfigure}
\caption{This figure shows that when $t\ge -1$, $Y^{t}_{\textbf{a}}$ is the double cover of $S^3$ branched over the closure of the 3-braid word $(\sigma_1\sigma_2)^{3t}\sigma_1\sigma_2^{-(a_1-2)}\cdots\sigma_1\sigma_2^{-(a_n-2)}$. The same is true when $t<-1$.}\label{evenoddt}
\end{figure}

Coupling this characterization with Theorem \ref{thm1}, Theorem \ref{mainthm}, and Lemma \ref{rationalballobstruction}, we can classify certain families of 3-braid closures admitting double branched covers  bounding $\QQ B^4s$.

\begin{cor} Let $\textbf{a}=(a_1,\ldots,a_n)$, where $n\ge1$, $a_i\ge 2$ for all $i$ and $a_j\ge 3$ for some $j$, and let $\textbf{d}$ be the cyclic-dual of $\textbf{a}$.
	\begin{itemize}
		\item Suppose $\textbf{d}\notin\mathcal{S}_{1a}\cup\mathcal{O}$. Then the double cover of $S^3$ branched over the closure of the 3-braid word $(\sigma_1\sigma_2)^{-3}\sigma_1\sigma_2^{-(a_1-2)}\cdots\sigma_1\sigma_2^{-(a_n-2)}$ bounds a $\QQ B^4$ if and only if $\textbf{a}\in\mathcal{S}_1$ or $\textbf{d}\in\mathcal{S}_{1b}\cup\mathcal{S}_{1c}\cup\mathcal{S}_{1d}\cup\mathcal{S}_{1e}$. 
		\item  Suppose $\textbf{a}\notin\mathcal{S}_{1a}\cup\mathcal{O}$. Then the double cover of $S^3$ branched over the closure of the 3-braid word $(\sigma_1\sigma_2)^{3}\sigma_1\sigma_2^{-(a_1-2)}\cdots\sigma_1\sigma_2^{-(a_n-2)}$ bounds a $\QQ B^4$ if and only if $\textbf{d}\in\mathcal{S}_1$ or $\textbf{a}\in\mathcal{S}_{1b}\cup\mathcal{S}_{1c}\cup\mathcal{S}_{1d}\cup\mathcal{S}_{1e}$. 
		\item The double cover of $S^3$ branched over the closure of the 3-braid word\\ $\sigma_1\sigma_2^{-(a_1-2)}\cdots\sigma_1\sigma_2^{-(a_n-2)}$ bounds a $\QQ B^4$ if and only if $\textbf{a}\in\mathcal{S}_2$. 
		\item If $\textbf{a}\in\mathcal{S}_{2c}$, then the double cover of $S^3$ branched over the closure of the 3-braid word $(\sigma_1\sigma_2)^{3t}\sigma_1\sigma_2^{-(a_1-2)}\cdots\sigma_1\sigma_2^{-(a_n-2)}$ bounds a $\QQ B^4$ for all even $t$.
		\end{itemize}		
		\label{cor1}\end{cor}

The 3-braid knots corresponding to strings in $\mathcal{S}_{1a}\cup\mathcal{S}_{2a}\cup\mathcal{S}_{2b}\cup\mathcal{S}_{2c}$ (and their mirrors) were shown in \cite{lisca3braid} to be 3-braid knots of finite concordance order. Moreover, some of them were shown be slice knots and so the corresponding double branched covers are already known to bound $\QQ B^4s$. Furthermore, by the classification in \cite{lisca3braid}, many of the the remaining strings in $\mathcal{S}$ correspond to infinite-concordance order 3-braid knots. Thus, these give examples of infinite concordance order knots whose double branched covers bound $\QQ B^4s$. Rewording Question \ref{questionballs} in terms of 3-braids, a natural question is the following.

\begin{question} Which other 3-braid closures admit double branched covers bounding $\QQ B^4s$?\end{question}

\subsection{Organization} In Section \ref{obstructions}, we will highlight some simple obstructions to $\QQ S^1\times S^2s$ bounding $\QQ S^1\times B^3s$, recall Heegaard Floer homology calculations of 3-braid closures due to Baldwin, and use these calculations to explore the orientation reversal of the 3-manifold $Y^t_{\textbf{a}}$. These obstructions and calculations will be used in Sections \ref{tbundles} and \ref{spheres}. In particular, in Section \ref{tbundles}, we will use the obstructions and other techniques to prove Theorem \ref{mainthm} and in Section \ref{spheres}, we will show that the $\QQ S^3s$ of Theorem \ref{thm1} do indeed bound $\QQ B^4s$ by explicitly constructing them. In Sections \ref{lattice}$-$\ref{p1=0}, we will use lattice analysis to prove that the $\QQ S^3s$ of Theorem \ref{thm1} are the only such $\QQ S^3s$ that bound $\QQ B^4s$. Finally, Section \ref{appendix} (the Appendix) provides some continued fraction calculations that are used in Sections \ref{obstructions} and \ref{spheres}.

\subsection{Acknowledgements} Thanks to Vitalijs Brejevs for pointing out a missing case in Lemma \ref{lem1.3} and thanks to the anonymous referee for carefully reading through the technical aspects of the paper and suggesting ways to greatly improve the flow of the paper.


\section{Obstructions}\label{obstructions}

In this section, we highlight some simple ways to obstruct a $\QQ S^1\times S^2$ from bounding a  $\QQ S^1\times B^3$, recall Baldwin's calculations of the Heegaard Floer homology of double covers of $S^3$ branched over certain 3-braid closures \cite{baldwinfloertorusbundles} (i.e. the rational homology 3-spheres $Y_{\textbf{a}}^t$), and show that reversing the orientation of the rational homology sphere $Y^t_{\textbf{a}}$ yields $Y^{-t}_{\textbf{d}}$, where $\textbf{d}$ is the cyclic-dual of $\textbf{a}$.
The first obstruction is a consequence of Proposition 1.5 and Corollary 1.6 in \cite{cochranetal}.

\begin{lem}[\cite{cochranetal}] If $K\subset S^3$ is an alternating knot and $S^3_0(K)$ bounds a $\QQ S^1\times B^3$, then $\sigma(K)=0$.
	\label{d1/2invtcor}\end{lem}

The next obstruction is akin to a well-known homology obstruction of $\QQ S^3s$ bounding $\QQ B^4s$ (Lemma 3 in \cite{cassongordon86}).

\begin{lem} If $Y$ bounds a $\QQ S^1\times B^3$, then the torsion part of $H_1(Y)$ has square order.\label{firsthomologyobstruction}\end{lem}
\begin{proof}
It is well-known that if a $\QQ S^3$ bounds a $\QQ B^4$, then its first homology group has square order (Lemma 3 in \cite{cassongordon86}). A similar, but more complicated argument will prove the lemma. 

Let $A=\text{Tor}(H_1(Y))$. We aim to show that $|A|$ is a perfect square. Let $W$ be a $\QQ S^1\times B^3$ bounded by $Y$. Then
\[ H_i(W)\cong\begin{cases} 
T_2 & i=2\\
\ZZ\oplus T_1 & i=1\\
\ZZ & i= 0
\end{cases},
\]
where $T_1$ and $T_2$ are torsion groups. By duality and universal coefficient theorems,
\[ H_i(W,Y)\cong\begin{cases} 
\ZZ & i=3\\
T_1 & i=2\\
T_2 & i=1\\
\end{cases}.
\]
Since $H_3(W)$ and $H_1(W,Y)$ are torsion groups, and $H_3(W,Y)\cong H_0(Y)\cong\ZZ$, the maps $H_3(W)\to H_3(W,Y)$ and $H_1(W,Y)\to H_0(Y)$ in the long exact sequence of the pair $(W,Y)$ are trivial. Thus we have the following long exact sequence:\\

\noindent\begin{tikzcd}[
	ar symbol/.style = {draw=none,"#1" description,sloped},
	isomorphic/.style = {ar symbol={\cong}},
	equals/.style = {ar symbol={=}},
	cramped,
	sep=small]
	0 \ar[r] & H_3(W,Y) \ar[r,"f"] & H_2(Y) \ar[r] & H_2(W) \ar[r] & H_2(W,Y)  \ar[r] & H_1(Y)  \ar[r,"g"] & H_1(W) \ar[r,"h"] & H_1(W,Y)  \ar[r] & 0.\\
	& \ZZ \ar[u,isomorphic] & \ZZ \ar[u,isomorphic] & T_2 \ar[u,isomorphic] & T_1 \ar[u,isomorphic]& \ZZ\oplus A \ar[u,isomorphic]& \ZZ\oplus T_1 \ar[u,isomorphic]& T_2 \ar[u,isomorphic]\\
\end{tikzcd}

Express the map $g$ as $g=g_1+g_2$, where $g_1:\ZZ \to \ZZ\oplus T_1$ and $g_2:A\to \{0\}\oplus T_1$. Notice that $\text{Im}g\cong \text{Im}g_1\oplus\text{Im}g_2$ and $g_1$ is injective. Thus $\text{Im}g_2$ can be identified with a subgroup of $\text{coker}g_1$ and $T_2\cong\text{coker}g\cong \displaystyle\frac{\text{coker}g_1}{\text{Im}g_2}$.  Moreover, it follows from duality that if $f$ is given by multiplication by $n$, then $g_1$ is of the form $g_1(x)=\pm nz+\sum\lambda_ib_i$, where $x$ is a generator of the domain of $g_1$ and $\{z,b_i\}$ is a basis for $\ZZ\oplus T_1$ such that $z$ is an infinite order element and the $b_i$ are torsion elements. Thus $|\text{coker}g_1|=n|T_1|=|\text{coker}f||T_1|$. 

By exactness, we can reduce the above sequence to the following short exact sequence
\begin{center}
\begin{tikzcd}[
	ar symbol/.style = {draw=none,"#1" description,sloped},
	isomorphic/.style = {ar symbol={\cong}},
	equals/.style = {ar symbol={=}},
	cramped,
	sep=small]
	0 \ar[r] & T_1/(T_2/\text{coker}f) \ar[r,"i"] & \ZZ\oplus A \ar[r,"g"] & \text{Im}g  \ar[r] &  0,
\end{tikzcd}
\end{center}
where we identify $\text{coker}f$ with its image in $T_2$ and $T_2/\text{coker}f$ with its image in $T_1$. Since $g_1:\ZZ \to \text{Im}g_1$ is an isomorphism, we have the following short exact sequence of finite groups: 

\begin{center}
	\begin{tikzcd}[
		ar symbol/.style = {draw=none,"#1" description,sloped},
		isomorphic/.style = {ar symbol={\cong}},
		equals/.style = {ar symbol={=}},
		cramped,
		sep=small]
		0 \ar[r] & T_1/(T_2/\text{coker}f) \ar[r,"i"] & A \ar[r,"g_2"] & \text{Im}g_2  \ar[r] &  0.
	\end{tikzcd}
\end{center}

\noindent Consequently, $|A|=\displaystyle\Big|\frac{T_1}{T_2/\text{coker}f}\Big||\text{Im}g_2|$. \\

\noindent Moreover, $\displaystyle\Big|\frac{T_1}{T_2/\text{coker}f}\Big|=\frac{|T_1||\text{coker}f|}{|T_2|}=\frac{|\text{coker}g_1|}{|\text{coker}g_1|/|\text{Im}g_2|}=|\text{Im}g_2|$.\\

 \noindent Thus $|A|=|\text{Im}g_2|^2$ is a square.
\end{proof}

\subsection{Heegaard Floer homology calculations}\label{baldwinsection}
Let $\textbf{a}=(a_1,\ldots,a_n)$, where $a_i\ge2$ for all $1\le i\le n$ and $a_j\ge3$ for some $j$. As mentioned in Section \ref{3braids}, the rational sphere $Y^t_{\textbf{a}}$ is the double cover of $S^3$ branched over the closure of the 3-braid represented by the word $(\sigma_1\sigma_2)^{3t}\sigma_1\sigma_2^{-(a_1-2)}\cdots\sigma_1\sigma_2^{-(a_n-2)}$. In \cite{baldwinfloertorusbundles}, Baldwin calculated the Heegaard Floer homology of these 3-manifolds equipped with a canonical spin$^c$ structure $\mathfrak{s}_0$. In particular, he showed that:

\[ HF^+(Y_{\textbf{a}}^{2m},\mathfrak{s}_0)=\begin{cases} 
       \Big(\mathcal{T}^+_0\oplus\ZZ_{0}^{m}\Big)\Big\{(3n-\sum a_i)/4\Big\} & \text{if }m\ge0 \\
      \Big(\mathcal{T}^+_0\oplus\ZZ_{-1}^{-m}\Big)\Big\{(3n-\sum a_i)/4\Big\} & \text{if }m<0 \\
   \end{cases} 
\]
\vspace{.2in}
 \[ HF^+(Y_{\textbf{a}}^{2m+1},\mathfrak{s}_0)=\begin{cases} 
     \Big( \mathcal{T}^+_0\oplus\ZZ_{-1}^{m}\Big)\Big\{(3n+4-\sum a_i)/4\Big\} & \text{if }m\ge0 \\
     \Big( \mathcal{T}^+_{-2}\oplus\ZZ_{-2}^{-(m+1)}\Big)\Big\{(3n+4-\sum a_i)/4\Big\} &\text{if } m<0 \\
   \end{cases}
\] 

$$\text{and}$$

$$ \{d(Y_{\textbf{a}}^{t},\mathfrak{s})\text{ }|\text{ }\mathfrak{s}\neq\mathfrak{s}_0\}=\{d(Y_{\textbf{a}}^{s},\mathfrak{s})\text{ }|\text{ }\mathfrak{s}\neq\mathfrak{s}_0\} \text{ for all } s,t\in\ZZ.$$

\subsection{Reversing Orientation}\label{reverseorientation} Let $\textbf{a}=(a_1,\ldots,a_n)$, where $a_i\ge2$ for all $1\le i\le n$ and $a_j\ge3$ for some $j$. As discussed in the introduction, reversing the orientation of the hyperbolic torus bundle $\textbf{T}_{\pm A(\textbf{a})}$ yields the hyperbolic torus bundle $\overline{\textbf{T}}_{\pm A(\textbf{a})}=\textbf{T}_{\pm A(\textbf{d})}$, where $\textbf{d}=(d_1,\ldots,d_m)$ is the cyclic-dual of $\textbf{a}$ (\cite{neumann}). Therefore, by construction, reversing the orientation on $Y_{\textbf{a}}^{t}$ yields $\overline{Y_{\textbf{a}}^{t}}=Y_{\textbf{d}}^{s}$ for some integer $s$. The following lemma shows that $s=-t$.

\begin{lem} Let $\textbf{a}=(a_1,\ldots,a_n)$ and $\textbf{d}=(d_1,\ldots,d_m)$ be cyclic-dual. Then $\overline{Y_{\textbf{a}}^{t}}=Y_{\textbf{d}}^{-t}$. 
\label{revolemma}\end{lem}

\begin{proof} This is an exercise in Kirby calculus. We will focus on the case $n>1$. The case $n=1$ similar, but much simpler. Start with the surgery diagram of $Y_{\textbf{a}}^{t}$ that is made up of a $t$-half-twisted chain link with surgery coefficients $(-a_1,\ldots,-a_n)$, as in the left of Figure \ref{chainlinksurgery2}. We will produce a different surgery diagram for $Y_{\textbf{a}}^{t}$ using blowups and blowdowns. Without loss of generality, assume that $a_1\ge3$. Let $i>1$ be the smallest integer such that $a_i\ge3$ and let $K_i$ denote the unknot with surgery coefficient $-a_i$. If $a_i=2$ for all $2\le i\le n$, then set $i=n+1$, with the understanding that $a_{n+1}=a_1$ and $K_{n+1}=K_1$. We will prove the lemma in the case $i\le n$. The case of $i=n+1$ is similar and requires less steps. Blow up the linking of the $-a_1$- and $-a_2$-framed unknots with a $+1$-framed unknot to obtain the second diagram in Figure \ref{chainlinksurgery2}. We can now perform $i-2$ successive blowdowns of $-1$-framed unknots (with $i-2=0$ a possibility). Next, perform $a_i-2$ successive $+1$-blowups of the linking between $K_i$ and the adjacent positively framed unknot; the resulting framing on $K_i$ is $-1$. Continue to perform blowdowns and blowups in this way until every surgery coefficient is a positive number; we obtain the surgery diagram for $Y_{\textbf{a}}^{t}$ made up of a chain link with positive surgery coefficients $(d_1,\ldots,d_m)$, as in the third diagram of Figure \ref{chainlinksurgery2}, where $\textbf{d}=(d_1,\ldots,d_m)$ is the cyclic-dual of $\textbf{a}$. Now we can change the orientation of $Y_{\textbf{a}}^{t}$ by reflecting this new surgery diagram through the page. This yields a surgery diagram of $\overline{Y_{\textbf{a}}^{t}}$ that is made up of a $-t$-half-twisted chain link with surgery coefficients $(-d_1,\ldots,-d_n)$, as shown at the end of Figure \ref{chainlinksurgery2}. Thus $\overline{Y_{\textbf{a}}^{t}}=Y_{\textbf{d}}^{-t}$. 
\end{proof}

\begin{figure}
	\centering
   \includegraphics[scale=.55]{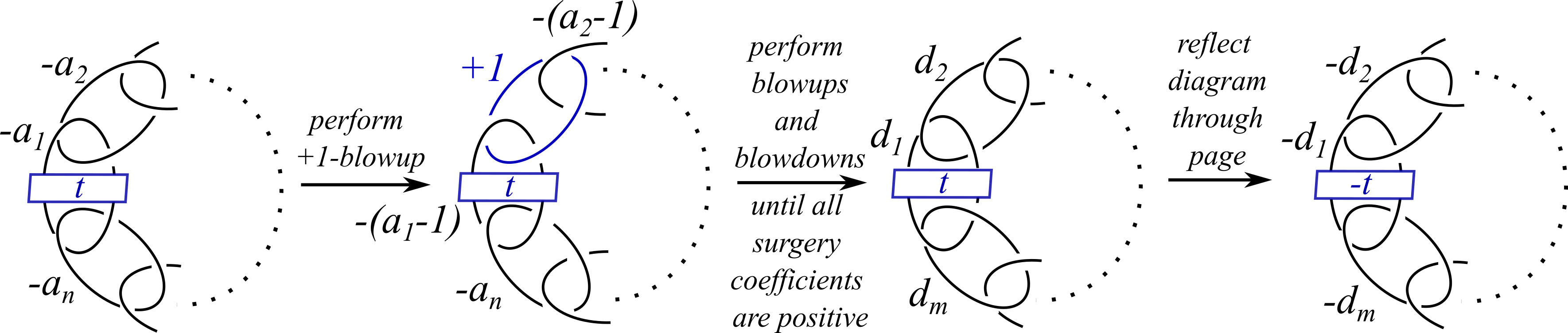}
	\caption{Proving that $\overline{Y_{\textbf{a}}^t}=Y_{\textbf{d}}^{-t}$, where $(d_1,\ldots,d_m)$ is the cyclic-dual of $\textbf{a}=(a_1,\ldots,a_n)$ and $n>1$.}\label{chainlinksurgery2}
\end{figure}


\section{Torus bundles over $S^1$ that bound rational homology circles}\label{tbundles}

In this section, we will prove Theorem \ref{mainthm}. By considering the obvious handlebody diagrams of the plumbings shown in Figure \ref{torusbundles}, it is rather straight forward to classify elliptic and parabolic torus bundles over $S^1$ that bound $\QQ S^1\times B^3s$. In fact, through Kirby calculus, we will explicitly construct $\QQ S^1\times B^3s$ bounded by negative parabolic torus bundles and use the obstructions in Section \ref{obstructions} to obstruct positive parabolic torus bundles and elliptic torus bundles from bounding $\QQ S^1\times B^3s$. 

\begin{prop} No elliptic torus bundle bounds a $\QQ S^1\times B^3$.\label{ellipticprop}\end{prop}

\begin{proof} According to Figure \ref{torusbundles}, there are only six elliptic torus bundles; they have monodromies $\pm S$, $\pm T^{-1}S$, and $\pm (T^{-1}S)^2$. By Proposition \ref{firsthomologyobstruction}, if one of these torus bundles bounds a $\QQ B^4$, then the torsion part of its first homology group must be a square. By considering the surgery diagrams in Figure \ref{torusbundles}, it is easy to see that the only elliptic torus bundles  that have the correct first homology are those with monodromy $T^{-1}S$ and $-(T^{-1}S)^2$. Moreover, note that by reversing the orientation on the torus bundle with monodromy $T^{-1}S$, we obtain the torus bundle with monodromy $-(T^{-1}S)^2$. Thus we need only show that one of these torus bundles does not bound a $\QQ S^1\times B^3$. Consider the leftmost surgery diagram of the elliptic torus bundle with monodromy $T^{-1}S$ in Figure \ref{ellobstruction}. By blowing down the $1$-framed unknot, we obtain $0$-surgery on the right-handed trefoil. Since the signature of the right-handed trefoil is 2, by Corollary \ref{d1/2invtcor}, the elliptic torus bundle does not bound a $\QQ S^1\times B^3$.\end{proof}

\begin{prop} Every negative parabolic torus bundle bounds a $\QQ S^1\times B^3$. No positive parabolic torus bundle bounds a $\QQ S^1\times B^3.$\label{parabolicprop}\end{prop}

\begin{proof} By considering the surgery diagrams of the parabolic torus bundles in Figure \ref{torusbundles}, it is easy to see that positive parabolic torus bundles, which have monodromy $T^n$, satisfy $b_1=2$. Thus, by the homology long exact sequence of the pair, it is easy to see that no such torus bundle can bound a $\QQ S^1\times B^3$. On the other hand, the negative parabolic torus bundles with monodromy  $-T^n$ bound obvious $\QQ S^1\times B^3s$, as shown in Figure \ref{parabolicrationalcircles}.\end{proof}

\begin{figure}[t!]
	\centering
	\includegraphics[scale=.6 ]{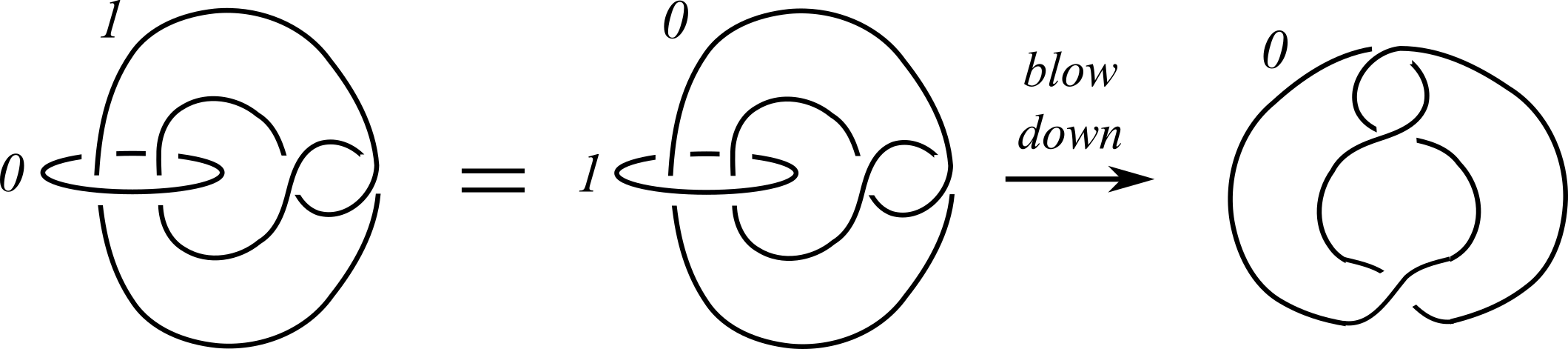}
	\caption{The elliptic torus bundle with monodromy $T^{-1}S$ does not bound a rational homology circle.}\label{ellobstruction}
\end{figure}  
\begin{figure}[t!]
\centering
\includegraphics[scale=.6]{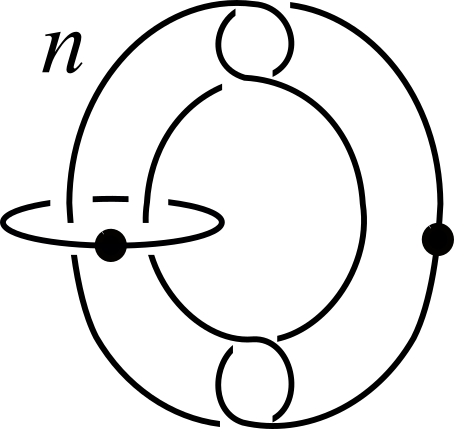}
\caption{$\QQ S^1\times B^3$ bounded by the negative parabolic torus bundle with monodromy $-T^n$.}\label{parabolicrationalcircles}
\end{figure}

Classifying hyperbolic torus bundles that bound $\QQ S^1\times B^3s$ is not as simple as the elliptic and parabolic cases. The hyperbolic torus bundles listed in Theorem \ref{mainthm} were shown to bound $\QQ S^1\times B^3s$ in \cite{simone3}. 

\begin{prop}[\cite{simone3}] Let $$\textbf{a}=(3+x_1,2^{[x_2]},\ldots,3+x_{2m+1},2^{[x_1]}, 3+x_2,2^{[x_3]},\ldots,3+x_{2m},2^{[x_{2m+1}]})\in\mathcal{S}_{2c},$$ where $m\ge0$ and $x_i\ge0$ for all $i$. Then $\textbf{T}_{A(\textbf{a})}$ bounds a $\QQ S^1\times B^3$.\label{hypprop}\end{prop}

It remains to obstruct all other hyperbolic torus bundles from bounding $\QQ S^1\times B^3s$. A major ingredient towards proving this fact is Theorem \ref{thm1}, which we assume to be true throughout the remainder of this section. The proof of Theorem \ref{thm1} will be covered in Sections \ref{spheres}-\ref{p1=0}. Note that ``most" hyperbolic torus bundles are obstructed by Theorem \ref{thm1}. In particular, by Theorem \ref{thm1}, if $\textbf{a},\textbf{d}\not\in\mathcal{S}_1\cup\mathcal{O}$, then $\textbf{T}_{- A(\textbf{a})}$ does not bound a $\QQ S^1\times B^3$ and if  $\textbf{a},\textbf{d}\not\in\mathcal{S}_2$, then $\textbf{T}_{A(\textbf{a})}$ does not bound a $\QQ S^1\times B^3$ (where $\textbf{d}$ is the cyclic-dual of $\textbf{a}$). Thus it remains to prove that if $\textbf{a}\text{ or }\textbf{d}\in\mathcal{S}_1\cup\mathcal{O}$, then $\textbf{T}_{ -A(\textbf{a})}$ does not bound a $\QQ S^1\times B^3$ and if $\textbf{a}\text{ or }\textbf{d}\in\mathcal{S}_2\setminus \mathcal{S}_{2c}$, then $\textbf{T}_{ A(\textbf{a})}$ does not bound a $\QQ S^1\times B^3$ (recall that $\textbf{a}\in\mathcal{S}_{2c}$ if and only if $\textbf{d}\in\mathcal{S}_{2c}$ by Example \ref{2crelabel}).
We will prove this by considering cyclic covers of these torus bundles. But first we need to better understand the set $\mathcal{S}$. In the upcoming subsection, we will round up some necessary technical results regarding $\mathcal{S}$ and in the subsequent subsection, we will explore  cyclic covers and finish the proof of Theorem \ref{mainthm}. 

\subsection{Analyzing $\mathcal{S}$}\label{sec:S} The first technical lemma shows that the sets $\mathcal{S}_1$ and $\mathcal{S}_2$ are disjoint.

\begin{lem} For a fixed string $\textbf{a}$, $Y^0_{\textbf{a}}$ and $Y^{-1}_{\textbf{a}}$ do not both bound $\QQ B^4s$ (and consequently, $\textbf{T}_{A(\textbf{a})}$ and $\textbf{T}_{-A(\textbf{a})}$ do not both bound $\QQ S^1\times B^3s$). It follows that $\mathcal{S}_1\cap\mathcal{S}_2=\emptyset$. \label{orderbound}\end{lem}

\begin{proof}
By construction, $|H_1(Y^0_{\textbf{a}})|=|\text{Tor}(H_1(\textbf{T}_{A(\textbf{a})}))|$ and $|H_1(Y^{-1}_{\textbf{a}})|=|\text{Tor}(H_1(\textbf{T}_{-A(\textbf{a})}))|$. By Lemma \ref{order} in the Appendix (Section \ref{appendix}), $|\text{Tor}(H_1(\textbf{T}_{A(\textbf{a})}))|=|\text{Tor}(H_1(\textbf{T}_{- A(\textbf{a})}))|-4$. Thus $|H_1(Y^0_{\textbf{a}})|$ and $|H_1(Y^{-1}_{\textbf{a}})|$ cannot simultaneously be squares and so by by Lemma 3 in \cite{cassongordon86}, $Y^0_{\textbf{a}}$ and $Y^{-1}_{\textbf{a}}$ do not both bound $\QQ B^4s$. 
Now suppose $\textbf{a}\in\mathcal{S}_1\cap\mathcal{S}_2$. Then by Theorem \ref{thm1}, $Y^{-1}_{\textbf{a}}$ and $Y^{0}_{\textbf{a}}$ both bound $\QQ B^4s$, which is not possible. Therefore, $\mathcal{S}_1\cap\mathcal{S}_2=\emptyset$.
\end{proof}

Recall from Example \ref{2crelabel} that a string $\textbf{a}\in\mathcal{S}_{2c}$ can be expressed in two different, but equivalent, ways:
\begin{equation}
	\textbf{a}=(3+x_1,2^{[x_2]},\ldots,3+x_{2m+1},2^{[x_1]}, 3+x_2,2^{[x_3]},\ldots,3+x_{2m},2^{[x_{2m+1}]})\label{eqn1}
\end{equation}
$$\text{ and }$$
\begin{equation}
\textbf{a}=(b_1+1,b_2,\ldots,b_{k-1},b_k+1,c_1,\ldots,c_l),\label{eqn2}
\end{equation}

\noindent where $m\ge 0$, $x_i\ge0$ for all $i$, and $(b_1,\ldots,b_k)$ and $(c_1,\ldots,c_l)$ are linear-dual strings with $k+l\ge2$. This relationship is easy to see:

 $$(b_1+1,b_2,\ldots,b_{k-1},b_k+1)=(3+x_1,2^{[x_2]},\ldots,3+x_{2m+1})$$
$$\text{ and }$$
$$(c_1,\ldots,c_l)=(2^{[x_1]}, 3+x_2,2^{[x_3]},\ldots,3+x_{2m},2^{[x_{2m+1}]}).$$
\vspace{.1cm}

Also recall that $\mathcal{S}$ is defined up to cyclic reordering and reversing strings. Thus a string $\textbf{a}=(a_1,\ldots,a_n)\in\mathcal{S}_{2c}$  may not be of the form (\ref{eqn1}) written above. However, by simply cyclic reordering $\textbf{a}$, we can put $\textbf{a}$ in the form of (\ref{eqn1}), which is equivalent to (\ref{eqn2}). Moreover, it is clear that if $a_1\ge 3$, then $\textbf{a}$ is already in the form (\ref{eqn1}) and thus already in the form (\ref{eqn2}). This simple observation will be used throughout the rest of this subsection.

\begin{definition} Let $\textbf{a}$ and $\textbf{b}$ be strings. Then $\textbf{a}\textbf{b}$ denotes the string obtained by concatenating $\textbf{a}$ and $\textbf{b}$, and $\textbf{a}^p$ denotes the string obtained by concatenating $\textbf{a}$ with itself $p$ times. 
\end{definition}

The next lemma follows directly from the definitions of linear-dual and cyclic-dual strings. We leave the proof to the reader.

\begin{lem} \hfill
	\begin{enumerate}[(a)]
		\item 	Suppose $\textbf{a}$ has linear-dual $\textbf{x}=(x_1,\ldots,x_p)$ and $\textbf{b}$ has linear-dual $\textbf{y}=(y_1,\ldots,y_q)$. Then
		\begin{enumerate}[(i)]
			\item  $\textbf{a}\textbf{b}$ has linear-dual $(x_1,\ldots,x_{p-1},x_p-1+y_1,y_2,\ldots,y_q)$ and
			\item $\textbf{a}\textbf{b}$ has cyclic-dual $(x_2\ldots,x_{p-1},x_p-1+y_1,y_2,\ldots,y_{q-1},y_q-1+x_1)$\\ (up to cyclic reordering).
		\end{enumerate}
		\item If $\textbf{a}$ has cyclic-dual $\textbf{d}$, then $\textbf{a}^p$ has cyclic-dual $\textbf{d}^p$.
	\end{enumerate}
	\label{lem:dualconcat}
\end{lem}

\begin{definition}
We call a string $(a_1,\ldots,a_n)$ a \textit{palindrome} if $a_i=a_{n-(i-1)}$ for all $1\le i\le n$. 
\end{definition}

\begin{lem} 
Let $\textbf{a}=(b_1+3,b_2,\ldots,b_k,2,c_l,\ldots,c_1)\in\mathcal{S}_{2a}$ and $\textbf{b}=(3+x,b_1,\ldots,b_{k-1},b_k+1,2^{[x]},c_l+1,c_{l-1},\ldots,c_1)\in\mathcal{S}_{2b}.$ 
\begin{enumerate}[(a)]
	\item $\textbf{a}\in \mathcal{S}_{2c}$ if and only if $(b_1+1,b_2,\ldots,b_k)$ is a palindrome.
	\item$\textbf{b}\in\mathcal{S}_{2c}$ if and only if $(b_1\ldots,b_k)$ is a palindrome. 
\end{enumerate}
\label{palindrome}\end{lem}

\begin{proof} 	$(a)$ Let $\textbf{a}=(b_1+3,b_2,\ldots,b_k,2,c_l,\ldots,c_1)\in\mathcal{S}_{2a}$.
	Notice that since $(c_1,\ldots,c_l)$ is the linear-dual of $(b_1,\ldots,b_k)$, we have that $(2,c_1,\ldots,c_l)$ is the linear-dual of $(b_1+1,b_2,\ldots,b_k)$. Consequently, $(b_1+1,b_2,\ldots,b_k)$ is a palindrome if and only if $(2,c_1,\ldots,c_l)$ is a palindrome if and only if $c_l=2$ and $c_i=c_{l-i}$ for all $1\le i\le l-1$.
	
	Assume that $(b_1+1,b_2,\ldots,b_k)$ is a palindrome.  Then $b_k=b_1+1\ge 3$ and consequently, $c_l=2$. Let $d_1=b_1+2, d_k=b_k-1,$ and $d_i=b_i$ for all $2\le i\le k-1$ so that  $\textbf{a}=(d_1+1,d_2,\ldots,d_{k-1},d_k+1,2,c_l,\ldots,c_1)$. By Lemma \ref{lem:dualconcat}, $(2,2,c_1, c_2,\ldots,c_{l-1})$ has linear-dual	$(b_1+2,b_2,\ldots,b_{k-1},b_k-1)=(d_1,\ldots,d_k)$. On the other hand, since $(2,c_1,\ldots,c_l)$ is a palindrome, $(2,2,c_1, c_2,\ldots,c_{l-1})=(2,c_l,c_{l-1},c_{l-2},\ldots,c_1)$. Set $e_1=e_2=2$, $e_i=c_{i-2}$ for all $3\le i\le l+1$. Then  $(d_1,\ldots,d_k)$ has linear-dual $(e_1,\ldots,e_{l+1})$ and thus $(b_1+3,b_2,\ldots,b_k,2,c_l,\ldots,c_1)=(d_1+1,d_2,\ldots,d_{k-1},d_k+1,e_1,\ldots,e_{l+1})\in\mathcal{S}_{2c}$.

	Now assume $\textbf{a}\in\mathcal{S}_{2c}$. Since $b_1+3>3$, $\textbf{a}$ is of the form $\textbf{a}=(d_1+1,d_2,\ldots,d_{p-1},d_p+1,e_1,\ldots,e_q)$, where $(d_1,\ldots,d_p)$ are $(e_1,\ldots,e_q)$ are linear-dual. Thus $d_1=b_1+2$ and $e_q=c_1$. Note that the length of $\textbf{a}$ is $k+l+1=p+q$. We claim that $p=k$. Indeed, if $p>k$, then $(d_1,\ldots,d_k)=(b_1+2,b_2,\ldots,b_k)$ has linear-dual $(2,2,c_1,\ldots,c_l)$, implying that the length of $\textbf{a}$ is greater than $k+l+1$, a contradiction; if $p<k$, we arrive at a similar contradiction.
	Therefore $p=k$ and $q=l+1$; consequently, $e_1=2$, and $e_i=c_{l-i+2}$ for all $2\le i\le l+1$. On the other hand, by Lemma \ref{lem:dualconcat} the linear-dual of $(d_1,\ldots,d_p)=(b_1+2,b_2,\ldots,b_k-1)$ is $(e_1,\ldots,e_q)=(2,2,c_1,\ldots,c_{l-1})$. Thus $c_l=e_2=2$ and $c_i=c_{l-i}$ for all $1\le i\le l-1$. As mentioned above, this implies that   $(b_1+1,b_2,\ldots,b_k)$ is a palindrome.
	
	$(b)$ Let $\textbf{b}=(3+x,b_1,\ldots,b_{k-1},b_k+1,2^{[x]},c_l+1,c_{l-1},\ldots,c_1)\in\mathcal{S}_{2b}.$ Note that $(b_1,\ldots,b_k)$ is a palindrome if and only if $(c_1,\ldots,c_l)$ is a palindrome.
	
	Assume $(b_1,\ldots,b_k)$ is a palindrome. Let $d_1=2+x$ and  $d_i=b_{i-1}$ for all $2\le i\le k+1$. By Lemma \ref{lem:dualconcat}, the linear-dual of $(d_1,\ldots,d_{k+1})=(2+x,b_1,\ldots,b_{k-1},b_k)$ is $(2^{[x]}, c_1+1,c_2,\ldots,c_l)=(2^{[x]},c_l+1,c_{l-1},\ldots,c_1)$, since $(c_1,\ldots,c_l)$ is a palindrome. Relabel this string as $(e_1,\ldots,e_q)$. Then $\textbf{b}=(d_1+1,d_2,\ldots,d_k,d_{k+1}+1,e_1,\ldots,e_q)\in\mathcal{S}_{2c}$.
	
	Now assume $\textbf{b}\in\mathcal{S}_{2c}$. Since $3+x\ge3$, $\textbf{b}$ is of the form $\textbf{a}=(d_1+1,d_2,\ldots,d_{p-1},d_p+1,e_1,\ldots,e_q)$, where $(d_1,\ldots,d_p)$ are $(e_1,\ldots,e_q)$ are linear-dual. Thus $d_1+1=3+x$ and $e_q=c_1$. Following as in the proof of the first part, we have that $p=k+1$ and $q=l+x$.
	Consequently, $e_{x+1}=c_l+1$, and $e_{x+j}=c_{l-j+1}$ for all $l\le j\le l$. On the other hand, the linear-dual of $(d_1,\ldots,d_p)=(2+x,b_1,\ldots,b_k)$ is $(e_1,\ldots,e_q)=(2^{[x]},c_1+1,c_2\ldots,c_l)$. Thus $c_1=e_{x+1}-1=c_l$ and $c_j=e_{x+j}=c_{l-j+1}$ for all $2\le j\le l$. That is, $(c_1,\ldots,c_l)$ is a palindrome and thus so is $(b_1,\ldots,b_k)$.
	\end{proof}

\begin{lem} Let $\textbf{b}\in\mathcal{S}_{2a}\cup\mathcal{S}_{2b}$ and $p\ge 4$. Then there does not exist some proper substring $\textbf{a}$ of $\textbf{b}$ such that $\textbf{a}^p=\textbf{b}$.
\label{lem:S2ab}\end{lem}

\begin{proof}
Let $\textbf{b}=(3+x,b_1,\ldots,b_{k-1},b_k+1,2^{[x]},c_l+1,c_{l-1},\ldots,c_1)\in\mathcal{S}_{2b}$. Suppose $\textbf{a}$ is a proper substring of $\textbf{b}$ satisfying $\textbf{a}^p=\textbf{b}$ for some $p\ge4$. Then $\textbf{a}=(3+x,b_1,\ldots,b_m)$ for some $m$. If $m=0$, then $\textbf{a}=(3+x)$ and every entry of $\textbf{b}$ equals $3+x$. The only such string satisfies $x=0$ and $(b_1,\ldots,b_k)=(2)=(c_1,\ldots,c_l)$; that is, $\textbf{b}=(3,3,3)$. But then $p=3$, a contradiction. 

Assume $m\ge 1$. Since $\textbf{a}^p=\textbf{b}$, we have that $b_{m+1}=3+x\ge 3$; consequently, either $m\le k$ or $m\ge k+x$. If $m\ge k+x$, then $m\le l$. Thus up to switching the roles of $(b_1,\ldots,b_k)$ and $(c_1,\ldots,c_l)$ we may assume without loss of generality that $m\le k$. 

By Lemma \ref{lem:dualconcat}, the linear-dual of $(b_1,\ldots,b_{m})$ is of the form $(c_1,\ldots,c_{n-1},c'_{n})$, where $n\le l$ and $c_n'\le c_n$. 
We claim that $m=n$. First suppose $m<n$. Then since $\textbf{a}^p=\textbf{b}$, we have $b_m=c_1, b_{m-1}=c_2,\ldots,b_2=c_{m-1}, b_1=c_{m}$; that is, $(b_1,\ldots,b_{m})$ is a proper substring of $(c_1,\ldots,c_{n-1},c'_{n})$. But then the linear-dual of $(b_1,\ldots,b_{m})$ (i.e. $(c_1,\ldots,c_{n-1},c'_{n})$) is a proper substring of the linear-dual of $(c_1,\ldots,c_{n-1},c'_{n})$ (i.e. $(b_1,\ldots,b_m)$), which is a contradiction. A similar argument shows that $n<m$ is also not possible. Thus $m=n$.

Since $m=n$ and $\textbf{a}^p=\textbf{b}$, we have that $b_m=c_1, b_{m-1}=c_2,\ldots,b_2=c_{m-1}, b_1=c_{m},$ and $c_{m+1}=3+x\ge3$. 
If $m=k$, then since $c_{m+1}\ge 3$, we necessarily have that $x=0$ and $p=2$, a contradiction.
If $m=k-1$, then $b_k+1=b_{m+1}=3+x$ and by Lemma \ref{lem:dualconcat}, $(c_1,\ldots,c_l)=(c_1,\ldots,c'_m+1,2^{[x]})$; since $c_{m+1}\ge 3$, we once again have $x=0$ and $p=2$, a contradiction.
Thus either $x=0$ or $m\le k-2$. In the latter case, since $(b_1,\ldots,b_k)$ has linear-dual $(c_1,\ldots,c_{m-1},c_m')$, by Lemma \ref{lem:dualconcat}, $(b_1,\ldots,b_m,3+x,b_1)$ has linear-dual $(c_1,\ldots,c_{m-1},c_m'+1,2^{[x]},3,2^{[b_1-2]})$; since $c_{m+1}=3+x\ge3$, we necessarily have that $x=0$. Thus $c_{m+1}=b_{m+1}=3$.  Moreover, since $(b_1,\ldots,b_m)$ has linear-dual $(c_1,\ldots,c_{m-1},c_m')$, by Lemma \ref{lem:dualconcat}, $(b_1,\ldots,b_m,3)$ has linear-dual $(c_1,\ldots,c_{m-1},c_m'+1,2)$. Therefore, $c_m=c_m'+1$.

Since $p\ge 4$, it follows that either $2m+2\le k$ or $2m+2\le l$. Without loss of generality, assume $2m+2\le k$. Then $(b_1,\ldots,b_m,3,b_1,\ldots,b_m,3)$ is a substring of $(b_1,\ldots,b_k)$ and its linear-dual is a substring of $(c_1,\ldots,c_l)$. 
By Lemma \ref{lem:dualconcat}, $(b_1,\ldots,b_m,3)$ has linear-dual $(c_1,\ldots,c_m,2)$ and consequently $(b_1,\ldots,b_m,3,b_1,\ldots,b_m,3)$ has linear-dual $(c_1,\ldots,c_m,c_1+1,c_2,\ldots,c_m,2)$. But, since $\textbf{a}^p=\textbf{b}$, the latter string is also of the form $(b_m,\ldots,b_1,3,b_m,\ldots,b_2,b_1)$. Thus $c_1=2$ and $b_1=2$. But since $(b_1,\ldots,b_m)$ and $(c_1,\ldots,c_m')$ are linear-dual and $c_1=b_1=2$, we necessarily have $(b_1,\ldots,b_k)=(2)=(c_1,\ldots,c_l)$; therefore $\textbf{b}=(3,3,3)$ and $p=3$, a contradiction. We have thus shown that there does not exist a proper substring $\textbf{a}$ of $\textbf{b}$ such that $\textbf{b}=\textbf{a}^p$ for some $p\ge 4$. 

Next suppose $\textbf{b}=(b_1+3,b_2\ldots,b_k,2,c_l,\ldots,c_1)\in\mathcal{S}_{2a}$.
Let $\textbf{a}=(b_1+3,b_2,\ldots,b_m)$ be a substring of $\textbf{b}$ such that $\textbf{a}^p=\textbf{b}$, where $p\ge 4$. We first claim that $m< k$. Assume otherwise. Then $m\le l$ and since $\textbf{a}^p=\textbf{b}$, $(b_1+3,b_2,\ldots,b_k)$ is a substring of $(c_1,\ldots,c_l)$. Consequently, the linear-dual of $(b_1+3,b_2,\ldots,b_k)$ (i.e. $(2,2,2,c_1,\ldots,c_l)$) is a substring of the linear-dual of $(c_1,\ldots,c_l)$ (i.e. $(b_1,\ldots,b_k)$), implying that $l<k<m$, a contradiction. Thus $m\le k$. If $m=k$, then $b_{m+1}=b_1+3\ge3$; on the other hand, $b_{m+1}=b_{k+1}=2$, a contradiction. Thus $k<m$. Now following the same argument as in the first part of the proof, we see that the linear dual of $(b_1+3,b_2,\ldots,b_m)$ is of the form $(c_1,\ldots,c_m')$, where $c_m'\le c_m$ and $m\le l$. Thus $b_{m+1}=c_{m+1}=b_1+3\ge5$.  But by Lemma \ref{lem:dualconcat}, $(b_1+3,b_2,\ldots,b_m,b_{m+1})=(b_1+3,b_2,\ldots,b_m,b_1+3)$ has linear-dual $(c_1,\ldots,c_m,2^{[b_1+1]})$, implying that $c_{m+1}\ge 5$, which is another contradiction. 
\end{proof}

\begin{lem} Suppose $\textbf{a}\in\mathcal{S}_{2a}\cup\mathcal{S}_{2b}\cup\mathcal{S}_{2c}$ and $\textbf{a}^p\in\mathcal{S}_{2c}$ for some $p$. Then $\textbf{a}\in\mathcal{S}_{2c}$.
\label{lem:S2c}
\end{lem}

\begin{proof}
It suffices to show that if $\textbf{a}\in\mathcal{S}_{2a}$ or $\textbf{a}\in\mathcal{S}_{2b}$, then $\textbf{a}\in\mathcal{S}_{2c}$.
Let $\textbf{a}\in\mathcal{S}_{2a}$ so that $\textbf{a}^p$ is of the form 
\begin{equation*}
\begin{split}
\textbf{a}^p=(&\textcolor{blue}{b_1+3,b_2,\ldots,b_k,2,c_l,\ldots,c_1,}\\ 
&\hspace{.9in}\textcolor{blue}{ \vdots  \text{ }{\scriptstyle \frac{p-2}{2} \text{ times}}}\\
&\textcolor{blue}{b_1+3,b_2,\ldots,b_k,2,c_l,\ldots,c_1,}\\
&\textcolor{blue}{b_1+3,b_2,\ldots,b_k},2,c_l,\ldots,c_1,\\ 
&b_1+3,b_2,\ldots,b_k,2,c_l,\ldots,c_1,\\
&\hspace{.9in}\vdots  \text{ }{\scriptstyle \frac{p-2}{2} \text{ times}}\\
&b_1+3,b_2,\ldots,b_k,2,c_l,\ldots,c_1).
\end{split}
\end{equation*}

Since $\textbf{a}^p\in\mathcal{S}_{2c}$ and $b_1+3>3$, $\textbf{a}^p=(d_1+1,d_2,\ldots,d_{q-1},d_q+1,e_1,\ldots,e_r)$, where $(d_1,\ldots,d_q)$ and $(e_1,\ldots,e_r)$ are linear-dual strings. Following as in the proof of Lemma \ref{palindrome}, we have that $q=\frac{p-2}{2}(k+l+1)+k$, which is the length of the blue substring above. 
Thus, on the one hand, $(e_1,\ldots,e_q)$ is the black substring of $\textbf{a}^p$ above.
Comparing the end of both strings, it is clear that $c_l=2$ and $c_i=c_{l-i}$ for all $1\le i\le l-1$. As mentioned in the first paragraph of the proof of Lemma \ref{palindrome}, this implies that $(b_1+1,b_2,\ldots,b_k)$ is a palindrome. By Lemma \ref{palindrome}, $\textbf{a}\in\mathcal{S}_{2c}$.

Now assume $\textbf{a}\in\mathcal{S}_{2b}$. Then $\textbf{a}^p$ is of the form 
\begin{equation*}
\begin{split}
\textbf{a}^p=(&\textcolor{blue}{3+x,b_1,\ldots,b_{k-1},b_k+1,2^{[x]},c_l+1,c_{l-1},\ldots,c_1,}\\
&\hspace{1.5in} \textcolor{blue}{\vdots \text{ }{\scriptstyle \frac{p-2}{2} \text{ times}}}\\
&\textcolor{blue}{3+x,b_1,\ldots,b_{k-1},b_k+1,2^{[x]},c_l+1,c_{l-1},\ldots,c_1,}\\
&\textcolor{blue}{3+x,b_1,\ldots,b_{k-1},b_k+1},2^{[x]},c_l+1,c_{l-1},\ldots,c_1,\\
&3+x,b_1,\ldots,b_{k-1},b_k+1,2^{[x]},c_l+1,c_{l-1},\ldots,c_1,\\
&\hspace{1.5in} \vdots \text{ }{\scriptstyle \frac{p-2}{2} \text{ times}}\\
&3+x,b_1,\ldots,b_{k-1},b_k+1,2^{[x]},c_l+1,c_{l-1},\ldots,c_1).
\end{split}
\end{equation*}

Since $\textbf{a}^p\in\mathcal{S}_{2c}$, $\textbf{a}^p=(d_1+1,d_2,\ldots,d_{q-1},d_q+1,e_1,\ldots,e_r)$, where $(d_1,\ldots,d_q)$ and $(e_1,\ldots,e_r)$ are linear-dual strings. Following as above, we have that $q=\frac{p-2}{2}(k+l+x+1)+k+1$, which is the length of the blue substring above. 
Thus, on the one hand, $(e_1,\ldots,e_q)$ is the black substring of $\textbf{a}^p$ above. On the other hand, by computing the linear-dual of $(d_1,\ldots,d_q)$ from the blue string above, $(e_1,\ldots,e_q)$ ends in the substring $(c_1+1,\ldots,c_l)$.
Comparing the end of both strings, it is clear that $(c_1,\ldots,c_l)=(c_l,\ldots,c_1)$ and thus $(b_1,\ldots,b_k)$ is also a palindrome. By Lemma \ref{palindrome}, $\textbf{a}\in\mathcal{S}_{2c}$.
\end{proof}

\begin{cor} If $\textbf{a},\textbf{a}^p\in\mathcal{S}_{2a}\cup\mathcal{S}_{2b}\cup\mathcal{S}_{2c}$, where $p\ge 4$, then $\textbf{a}\in\mathcal{S}_{2c}.$
\label{cor:a^p}
\end{cor}

\begin{proof} It follows from Lemma \ref{lem:S2ab} that $\textbf{a}^p\in\mathcal{S}_{2c}$; thus $\textbf{a}^p\in\mathcal{S}_{2c}$. By Lemma \ref{lem:S2c}, $\textbf{a}\in\mathcal{S}_{2c}$. 
\end{proof}

The final technical lemma shows that the cyclic-duals of strings in $\mathcal{S}_{2a}\cup\mathcal{S}_{2b}\cup\mathcal{S}_{2c}$ are also in $\mathcal{S}_{2a}\cup\mathcal{S}_{2b}\cup\mathcal{S}_{2c}$. Although this result is implicit in the proof of Theorem \ref{thm1}, it is also relatively simple to prove directly, with the help of Lemma \ref{lem:dualconcat}.

\begin{lem} Let $\textbf{d}$ be the cyclic-dual of $\textbf{a}$. If $\textbf{a}\in\mathcal{S}_{2a}\cup\mathcal{S}_{2b}\cup\mathcal{S}_{2c}$, then $\textbf{d}\in\mathcal{S}_{2a}\cup\mathcal{S}_{2b}\cup\mathcal{S}_{2c}$.
	\label{lem:dualsof2}
\end{lem}

\begin{proof}
	Let $\textbf{a}\in\mathcal{S}_{2c}$. Using the description of $\textbf{a}$ as in (\ref{eqn1}) at the beginning of this subsection, it is easy to see that $\textbf{d}\in\mathcal{S}_{2c}$. Next, let $\textbf{a}=(3+x,b_1,\ldots,b_k+1,2^{[x]},c_l+1,c_{l-1},\ldots,c_1)\in\mathcal{S}_{2b}$. Notice that $(3+x,b_1,\ldots,b_k+1)$ has linear-dual $(2^{[x+1]},c_1+1,\ldots,c_l,2)$ and $(2^{[x]},c_l+1,c_{l-1},\ldots,c_1)$ has linear-dual $(2+x,b_k,\ldots,b_1)$. Thus by Lemma \ref{lem:dualconcat}, $\textbf{d}=(2^{[x]},c_1+1,\ldots,c_l,3+x,b_k,\ldots,b_1+1)\in\mathcal{S}_{2b}$. 
	
	Finally, let $\textbf{a}=(b_1+3,b_2,\ldots,b_k,2,c_l,\ldots,c_1)\in\mathcal{S}_{2a}$. If $k+l=1$, then $\textbf{a}=(4,2)$ and $\textbf{d}=(2,4)\in\mathcal{S}_{2a}$. If $k+l=2$, then $\textbf{a}=(5,2,2)$ and $\textbf{d}=(2,2,5)\in\mathcal{S}_{2a}$. Now let $k+l\ge3$. Then either $b_k\ge3$ and $c_l=2$ or vice versa. Assume the former. Since $(b_1+3,b_2,\ldots,b_k)$ has linear-dual $(2,2,2,c_1,\ldots,c_l)$ and $(2,c_l,\ldots,c_1)$ has linear-dual $(b_k+1,b_{k-1},\ldots,b_1)$, by Lemma \ref{lem:dualconcat}, $\textbf{d}=(2,2,c_1,\ldots,c_{l-1},c_l+b_k,b_{k-1},\ldots,b_2,b_1+1)$.  Let $d_1=c_l+b_k-3,$ $d_k=b_1+1$, and $d_i=b_{k-i+1}$ for all $2\le i\le k-1$. Also let $e_1=c_{l-1}$, $e_l=2$, and $e_i=c_{l-i}$ for all $2\le i\le l-1$. Then $\textbf{d}=(2,e_l,\ldots,e_1,d_1+3,d_2,\ldots,d_k)$ and $(d_1,\ldots,d_k)=(b_k-1,b_{k-1},\ldots,b_2,b_1+1)$ and $(e_1,\ldots,e_l)=(c_{l-1},\ldots,c_1,2)$ are linear-dual; thus $\textbf{d}\in\mathcal{S}_{2a}$. Now assume $b_k=2$ and $c_l\ge 3$. Set $d_1=c_l+b_k-3$, $d_{l+1}=2$, $d_i=c_{l-i+1}$ for all $2\le i\le l$, $e_1=b_{k-1}$, $e_{k-1}=b_1+1$, and $e_i=b_{k-i}$ for all $2\le i\le k-2$. Proceeding as above, we see that $\textbf{d}\in\mathcal{S}_{2a}$. 
\end{proof}

\subsection{Cyclic Covers and Proving Theorem \ref{mainthm}} We are now ready to finish the proof of Theorem \ref{mainthm}. The next two results explore cyclic covers of $\QQ S^1\times B^3s$ and cyclic covers of hyperbolic torus bundles over $S^1$. Coupling these results with the results in Section \ref{sec:S}, we complete the proof of Theorem \ref{mainthm} in the subsequent corollaries.

\begin{lem} Let $W$ be a $\QQ S^1\times B^3$ and let $\widetilde{W}$ be a $p$-fold cyclic cover of $W$, where $p$ is prime and not a divisor of $|\textup{Tor}(H_2(W;\ZZ))|$. If $\partial \widetilde{W}$ is a $\QQ S^1\times S^2$, then $\widetilde{W}$ is a $\QQ S^1\times B^3$.\label{coverlem}\end{lem}

\begin{proof} Let $Y=\partial W$ and $\widetilde{Y}=\partial \widetilde{W}$. Since $W$ is a $\QQ S^1\times B^3$ and $H_3(W;\ZZ)$ has no torsion, it follows that $H_3(W;\ZZ)=0$. Thus by Poincar{\'e} duality and the Universal Coefficient Theorem, we have the following isomorphisms:
	$$H_1(W,Y;\ZZ_p)\cong H^3(W;\ZZ_p)\cong \text{Ext}(H_2(W;\ZZ),\ZZ_p).$$
	Since $p$ is relatively prime to $|\textup{Tor}(H_2(W;\ZZ))|$, we have $H_1(W,Y;\ZZ_p)\cong\text{Ext}(H_2(W;\ZZ),\ZZ_p)=0$. By the proof of Theorem 1.2 in \cite{goldsmith}, since $p$ is prime, it follows that $H_1(\widetilde{W},\widetilde{Y};\ZZ_p)=0$. Once again applying Poincar{\'e} duality and the Universal Coefficient Theorem, we have the following isomorphisms:
	$$0=H_1(\widetilde{W},\widetilde{Y};\ZZ_p)\cong H^3(\widetilde{W};\ZZ_p)\cong \text{Hom}(H_3(\widetilde{W};\ZZ),\ZZ_p)\oplus\text{Ext}(H_2(\widetilde{W};\ZZ),\ZZ_p).$$
	Thus $H_3(\widetilde{W};\ZZ)$ is a torsion group. Thus if we apply Poincare{\'e} duality and the Universal Coefficent Theorem as above, but with $\QQ$-coefficients, we obtain:
	$$H_1(\widetilde{W},\widetilde{Y};\QQ)\cong H^3(\widetilde{W};\QQ)\cong \text{Hom}(H_3(\widetilde{W};\ZZ),\QQ)\oplus\text{Ext}(H_2(\widetilde{W};\ZZ),\QQ)=0.$$
	 Thus the map $H_1(\widetilde{Y};\QQ)\to H_1(\widetilde{W};\QQ)$ induced by inclusion is surjective. Since $\widetilde{Y}$ is a $\QQ S^1\times S^2$, it follows that $\text{rank}(H_1(\widetilde{W};\QQ))\le 1$. Finally, since $\chi(\widetilde{W})=p\chi(W)=0$ and  $H_3(\widetilde{W};\QQ)=0$, we necessarily have that $H_1(\widetilde{W};\QQ)=\QQ$ and $H_2(\widetilde{W};\QQ)=0$, proving that $\tilde{W}$ is indeed a $\QQ S^1\times B^3$.\end{proof}

\begin{prop} Let $\textbf{T}_{\pm A(\textbf{a})}$ be a hyperbolic torus bundle that bounds a $\QQ S^1\times B^3$, $W$. If $p$ is an odd prime that does not divide $|\text{Tor}(H_2(W;\ZZ))|$,  then $\textbf{T}_{\pm A(\textbf{a}^p)}$ bounds a $\QQ S^1\times B^3$.
\label{prop:bounding}
\end{prop}

\begin{proof}
Let $W$ be a $\QQ S^1\times B^3$ bounded by some negative hyperbolic torus bundle $\textbf{T}_{\pm A(\textbf{a})}$, where $\textbf{a}=(a_1,\ldots,a_n)$. Let $p$ be an odd prime number that is not a factor of $|\text{Tor}(H_2(W;\ZZ))|$. Consider the obvious surgery diagrams of $\textbf{T}_{A(\textbf{a})}$ and $\textbf{T}_{-A(\textbf{a})}$ as in Figures \ref{base0} and \ref{base-1}, respectively. In both diagrams, let $\mu_i$ denote the homology class of the meridian of the $-a_i$-framed surgery curve and let $\mu_0$ denote the homology class of the meridian of the 0-framed surgery curve. Then $H_1(\textbf{T}_{\pm A(\textbf{a})};\ZZ)$ is generated by $\mu_0,\ldots,\mu_n$. 

Consider the torus bundle $\textbf{T}_{-A(\textbf{a}^p)}$, which has monodromy $-(T^{-a_1}S\cdots T^{-a_n}S)^p$. The standard surgery diagram of this torus bundle includes a $-1$-half-twisted chain link (as in Figure \ref{torusbundles}). Note that by sliding the chain link over the $0-$framed unknot $\frac{p-1}{2}$ times, we may arrange that the chain link has $-p$ half-twists, as in Figure \ref{covering-1} (for the case $p=3$). For the torus bundle $\textbf{T}_{A(\textbf{a}^p)}$, which has monodromy $(T^{-a_1}S\cdots T^{-a_n}S)^p$, consider the standard surgery diagram shown in Figure \ref{covering0} (for the case $p=3$). 
\begin{figure}
	\centering
	\begin{subfigure}{.45\textwidth}
		\centering
		\includegraphics[scale=.55]{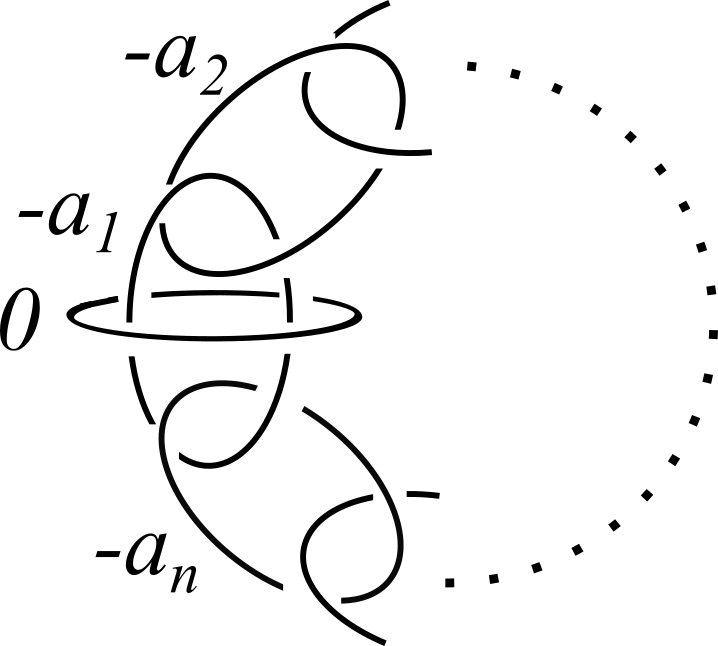}
		\caption{A surgery diagram for $\textbf{T}_{A(\textbf{a})}$}\label{base0}
	\end{subfigure}
	\begin{subfigure}{.45\textwidth}
		\centering
		\includegraphics[scale=.55]{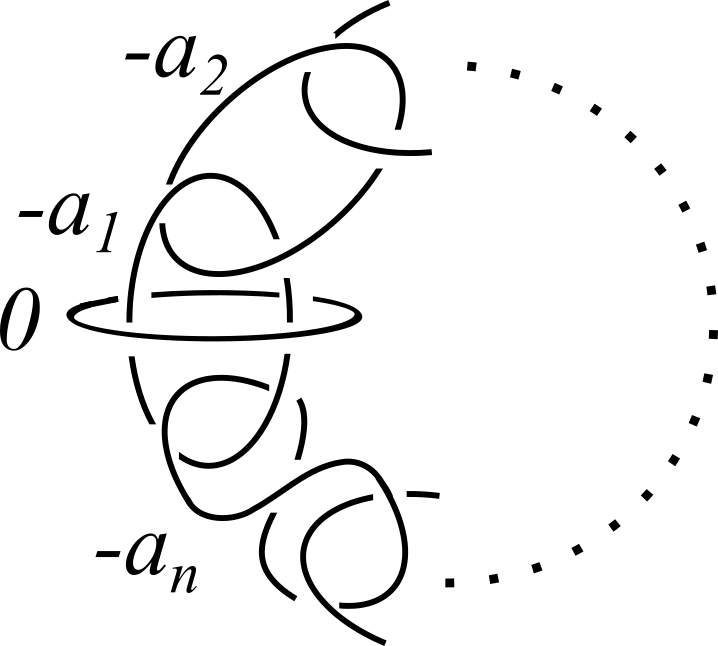}
		\caption{A surgery diagram for $\textbf{T}_{-A(\textbf{a})}$}\label{base-1}
	\end{subfigure}
	\begin{subfigure}{.45\textwidth}
	\vspace{.1cm}
	\centering
	\includegraphics[scale=.55]{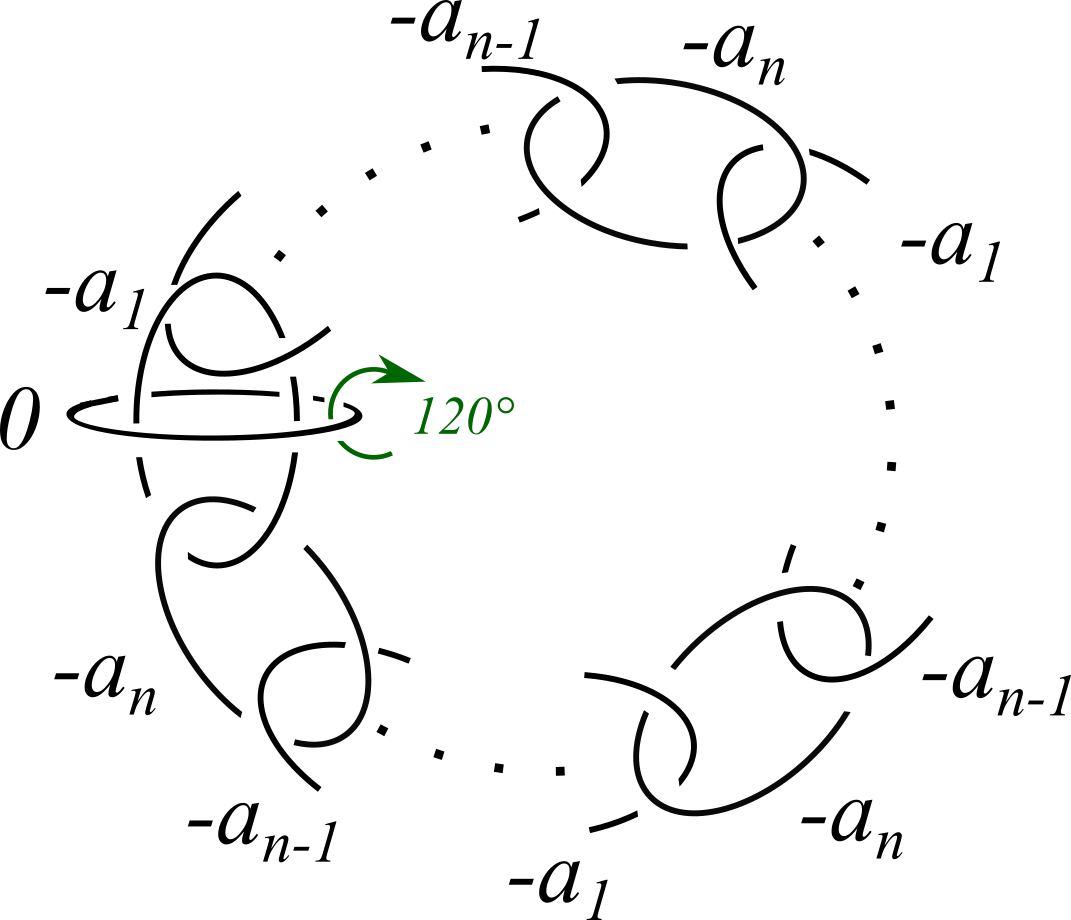}
	\caption{A surgery diagram for $\textbf{T}_{ A(\textbf{a}^3)}$}\label{covering0}
	\end{subfigure}
	\begin{subfigure}{.45\textwidth}
	\centering
	\includegraphics[scale=.55]{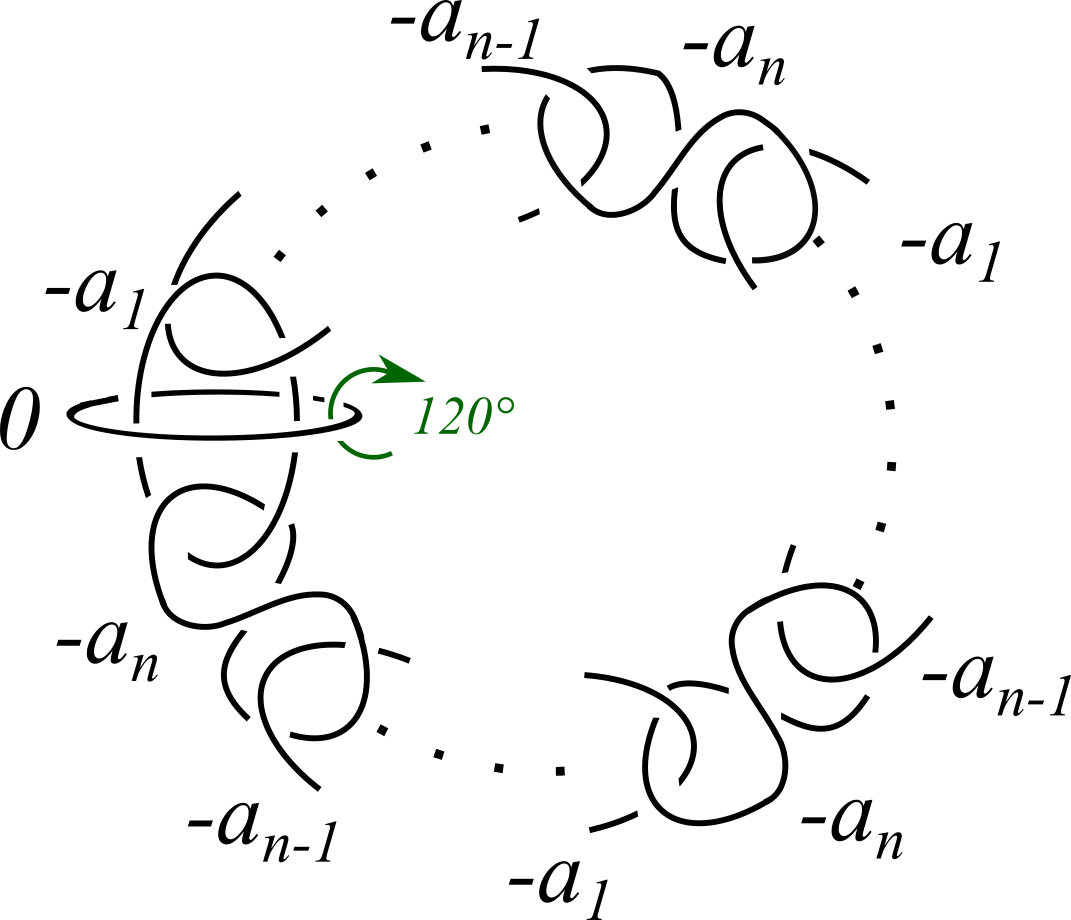}
	\caption{A surgery diagram for $\textbf{T}_{-A(\textbf{a}^3)}$}\label{covering-1}
	\end{subfigure}
\caption{$\textbf{T}_{\pm A(\textbf{a}^3)}$ is a $3$-fold cyclic cover of $\textbf{T}_{\pm A(\textbf{a})}$. There is an obvious $\ZZ_3$-action on $\textbf{T}_{\pm A(\textbf{a}^3)}$ given by a rotation of $120^{\circ}$ through the 0-framed unknot. The quotient of $\textbf{T}_{ \pm A(\textbf{a}^3)}$ by this action is $\textbf{T}_{\pm A(\textbf{a})}$.}
\end{figure}
There is an obvious $\ZZ_p$-action on $\textbf{T}_{\pm A(\textbf{a}^p)}$ obtained by rotating the chain link through the 0-framed unknot by an angle of $2\pi /p$, as indicated in Figures \ref{covering0} and \ref{covering-1}. The quotient of $\textbf{T}_{\pm A(\textbf{a}^p)}$ by this action is clearly $\textbf{T}_{\pm A(\textbf{a})}$ and the induced map $f:H_1(\textbf{T}_{\pm A(\textbf{a})};\ZZ)\to \ZZ_p$ satisfies $f(\mu_0)=1$ and $f(\mu_i)=0$ for all $1\le i\le n$. Consider the long exact sequence of the pair $(W,\textbf{T}_{\pm A(\textbf{a})})$:
\begin{center}
\begin{tikzcd}[
	ar symbol/.style = {draw=none,"#1" description,sloped},
	isomorphic/.style = {ar symbol={\cong}},
	equals/.style = {ar symbol={=}},
	cramped,
	sep=small]
	H_1(\textbf{T}_{\pm A(\textbf{a})};\ZZ) \ar[r,"i_*"] & H_1(W;\ZZ) \ar[r] & H_1(W,\textbf{T}_{\pm A(\textbf{a})};\ZZ) \ar[r] & 0.
\end{tikzcd}
\end{center}
Choose a basis $\{m_0,m_1,\ldots,m_k\}$ for $H_1(W;\ZZ)$ such that $m_0$ has infinite order and $m_i$ is a torsion element for all $1\le i\le k$.
Since $H_1(W,\textbf{T}_{\pm A(\textbf{a})};\ZZ)$ is a torsion group, $i_*(\mu_0)=\alpha m_0+\sum_{i=1}^k\beta_im_i$ for some $\alpha,\beta_i\in\ZZ$, where $\alpha\neq 0$. Since $p$ is not relatively prime to $|\text{Tor}(H_2(W;\ZZ))|=|H_1(W,\textbf{T}_{\pm A(\textbf{a})};\ZZ)|$ and $\alpha$ divides $|H_1(W,\textbf{T}_{\pm A(\textbf{a})};\ZZ)|$, it follows that $\alpha$ and $p$ are relatively prime; thus there exists an integer $t$ such that $t\alpha\equiv1\pmod p$. Define a map $g:H_1(W;\ZZ)\to\ZZ_p$ by $g(m_0)=t$ and $g(m_i)=0$ for all $1\le i\le k$. Then $g$ is a surjective homomorphism satisfying $f=g\circ i_*$. Let $\widetilde{W}$ be the $p$-fold cyclic cover of $W$ induced by $g$. Then $\partial \widetilde{W}=\textbf{T}_{\pm A(\textbf{a}^p)}$ and by Lemma \ref{coverlem}, $\widetilde{W}$ is a $\QQ S^1\times B^3$.
\end{proof}

The two following corollaries conclude the proof of Theorem \ref{mainthm}.

\begin{cor} No negative hyperbolic torus bundles bound a $\QQ S^1\times B^3$.  \label{cor:neghyp}\end{cor}

\begin{proof} Let $\textbf{T}_{-A(\textbf{a})}$ be a negative hyperbolic torus bundle that bounds a $\QQ S^1\times B^3$, $W$.
Let $p>3$ be an odd prime number that is not a factor of $|\text{Tor}(H_2(W;\ZZ))|$. By Proposition \ref{prop:bounding}, $\textbf{T}_{-A(\textbf{a}^p)}$ also bounds a $\QQ S^1\times B^3$. Let $\textbf{d}$ be the cyclic-dual of $\textbf{a}$; by Lemma \ref{lem:dualconcat} $\textbf{d}^p$ is the linear-dual of $\textbf{a}^p$. By Lemma \ref{rationalballobstruction}, $Y_{\textbf{a}}^{-1}$ and $Y_{\textbf{a}^p}^{-1}$ bound $\QQ B^4s$ and so by Theorem \ref{thm1}, $\textbf{a}\text{ or }\textbf{d}$ belongs to $\mathcal{S}_1\cup\mathcal{O}$ and $\textbf{a}^p\text{ or }\textbf{d}^p$ belongs to $\mathcal{S}_1\cup\mathcal{O}.$

First assume $\textbf{a},\textbf{a}^p\in \mathcal{S}_1\cup\mathcal{O}$. By Remark \ref{ibounds}, $-4\le I(\textbf{a}),I(\textbf{a}^p)\le 0$. Moreover, $I(\textbf{a}^p)=pI(\textbf{a})$. If $I(\textbf{a})<0$, then since $p>3$, we have $I(\textbf{a}^p)<-4$, which is a contradiction. Thus $I(\textbf{a}^p)=I(\textbf{a})=0$. By Remark \ref{ibounds}, $\textbf{a},\textbf{a}^p\in\mathcal{S}_{2a}\cup\mathcal{S}_{2b}\cup\mathcal{S}_{2c}\cup\mathcal{O}$. Since $\mathcal{S}_1\cap \mathcal{S}_2=\emptyset$, by Lemma \ref{orderbound},  we necessarily have that $\textbf{a},\textbf{a}^p\in\mathcal{O}$, which is not possible since $p\neq 1$. 

Next assume $\textbf{a},\textbf{d}^p\in \mathcal{S}_1\cup\mathcal{O}$. By Remark \ref{ibounds}, $-4\le I(\textbf{a}),I(\textbf{d}^p)\le 0$. Since $I(\textbf{d}^p)=pI(\textbf{d})=-pI(\textbf{a})$, we necessarily have that $I(\textbf{a})=I(\textbf{d}^p)=0$. As above, this implies that $\textbf{a},\textbf{d}^p\in\mathcal{O}$. But since $\textbf{a}\in\mathcal{O}$, it is clear that $\textbf{a}=\textbf{d}$ and thus $\textbf{d}\in\mathcal{O}$. As above, it is clear that $\textbf{d}$ and $\textbf{d}^p$ cannot both be contained in $\mathcal{O}$. 

Finally, if $\textbf{d},\textbf{d}^p\in \mathcal{S}_1\cup\mathcal{O}$ or $\textbf{d},\textbf{a}^p\in \mathcal{S}_1\cup\mathcal{O}$, similar arguments provide similar contradictions.
Therefore, $\partial W$ cannot be a negative hyperbolic torus bundle.
\end{proof}

\begin{cor} If a positive hyperbolic torus bundle $\textbf{T}_{A(\textbf{a})}$ bounds a $\QQ S^1\times B^3$, then $\textbf{a}\in\mathcal{S}_{2c}$. \label{cor:poshyp}\end{cor}

\begin{proof} 
Let $\textbf{T}_{A(\textbf{a})}$ be a positive hyperbolic torus bundle that bounds a $\QQ S^1\times B^3$, $W$, and let $p>3$ be an odd prime number that is not a factor of $|\text{Tor}(H_2(W;\ZZ))|$. Following as in the proof of Corollary \ref{cor:neghyp}, we have that $\textbf{a}$ or $\textbf{d}$ belongs to $\mathcal{S}_2$ and $\textbf{a}^p$ or $\textbf{d}^p$ belongs to $\mathcal{S}_2$, where $\textbf{d}$ is the cyclic-dual of $\textbf{a}$.
Suppose $\textbf{a},\textbf{a}^p\in\mathcal{S}_2$. Following the proof of Corollary \ref{cor:neghyp}, we have that $I(\textbf{a})=I(\textbf{a}^p)=0$ and so by Remark \ref{ibounds}, $\textbf{a},\textbf{a}^p\in\mathcal{S}_{2a}\cup\mathcal{S}_{2b}\cup\mathcal{S}_{2c}$. By Corollary \ref{cor:a^p}, we have that $\textbf{a}\in \mathcal{S}_{2c}$.
Next suppose $\textbf{a},\textbf{d}^p\in\mathcal{S}_2$. Once again, following the argument in Corollary \ref{cor:neghyp}, we have that $I(\textbf{a})=I(\textbf{d}^p)=0$ and so by Remark \ref{ibounds}, $\textbf{a},\textbf{d}^p\in\mathcal{S}_{2a}\cup\mathcal{S}_{2b}\cup\mathcal{S}_{2c}$. By Lemma \ref{lem:dualsof2}, we necessarily have that  $\textbf{a}^p\in\mathcal{S}_{2a}\cup\mathcal{S}_{2b}\cup\mathcal{S}_{2c}$; proceeding as in the previous case, we find $\textbf{a}\in \mathcal{S}_{2c}$.
Finally, if $\textbf{d},\textbf{a}^p\in\mathcal{S}_2$ or $\textbf{d},\textbf{d}^p\in\mathcal{S}_2$, we can similarly deduce that $\textbf{a}\in \mathcal{S}_{2c}$.
\end{proof}


\section{Surgeries on chain links bounding rational homology 4-balls}\label{spheres}

In this section, we will prove the necessary conditions of Theorem \ref{thm1}. Namely, we will show that the $\QQ S^3s$ of Theorem \ref{thm1} bound $\QQ B^4s$ by explicitly constructing such $\QQ B^4s$ via Kirby calculus. Notice that the necessary condition of Theorem \ref{thm1}(\ref{(2)}) follows from the the necessary condition of Theorem \ref{thm1}(\ref{(1)}) in light of Lemma \ref{revolemma}. Therefore, we need only show the following three cases (where $\textbf{a}$ and $\textbf{d}$ are cyclic-duals):
\begin{itemize}
	\item if $\textbf{a}\in\mathcal{S}_{1a}$, then $Y_{\textbf{a}}^{-1}$ bounds a $\QQ B^4$;
	\item if $\textbf{a}\in \mathcal{S}_{1b}\cup\mathcal{S}_{1c}\cup\mathcal{S}_{1d}\cup\mathcal{S}_{1e}$, then $Y_{\textbf{a}}^{-1}$ and $Y_{\textbf{d}}^{-1}$ bound $\QQ B^4s$; and
	\item if $\textbf{a}\in\mathcal{S}_2$, then $Y_{\textbf{a}}^{0}$ and $Y_{\textbf{d}}^{0}$ bound $\QQ B^4s$. 
\end{itemize}

Figures \ref{negativekirby} and \ref{positivekirby} exhibit the Kirby calculus needed to produce these $\QQ B^4s$. We will describe in detail the $\QQ B^4$ constructed in Figure \ref{negativekirby}a. The constructions in the other cases are similar. Notice that the leftmost figure of Figure \ref{negativekirby}a (without the $-1$-framed blue unknot) is a surgery diagram for  $Y_{\textbf{a}}^{-1}$, where $\textbf{a}=(b_1,\ldots,b_k,2,c_l,\ldots,c_1,2)\in\mathcal{S}_{1a}$. Thicken $Y_{\textbf{a}}^{-1}$ to the 4-manifold $Y_{\textbf{a}}^{-1}\times [0,1]$. By attaching a $-1$-framed 2-handle to $Y_{\textbf{a}}^{-1}\times\{1\}$ along the blue unknot in Figure \ref{negativekirby}a, we obtain a 2-handle cobordism from $Y_{\textbf{a}}^{-1}$ to a new 3-manifold, which we will show is $S^1\times S^2$. By performing a blowdown, we obtain the middle surgery diagram. Blowing down a second time, the surgery curves with framings $-b_1$ and $-c_1$ linking each other once and have framings $-(b_1-1)$ and $-(c_1-1)$, respectively. Since $(b_1,\ldots,b_k)$ and $(c_1,\ldots,c_l)$ are linear-dual, either $-(b_1-1)$ or $-(c_1-1)$ is equal to $-1$. We can thus blow down again. Continuing in this way, we can continue to blow down $-1$-framed unknots until we obtain $0$-surgery on the unknot, which is shown on the right side of the figure. Thus we have a 2-handle cobordism from $Y_{\textbf{a}}^{-1}$ to $S^1\times S^2$. By gluing this cobordism to $S^1\times B^3$, we obtain the desired $\QQ B^4$ bounded by $Y_{\textbf{a}}^{-1}$.

Suppose $\textbf{a}\in\mathcal{S}_{1b}\cup\mathcal{S}_{1c}\cup\mathcal{S}_{1d}\cup\mathcal{S}_{1e}$ and let $\textbf{d}$ be its cyclic-dual. Then by  Lemma \ref{revolemma}, $\overline{Y_{\textbf{d}}^{-1}}=Y_{\textbf{a}}^{1}$. To show that $Y_{\textbf{d}}^{-1}$ bounds a $\QQ B^4$, we will show that  $Y_{\textbf{a}}^{1}$ bounds a $\QQ B^4$. Figures \ref{1b}$-$\ref{1e} show that if $\textbf{a}\in \mathcal{S}_{1b}\cup\mathcal{S}_{1c}\cup\mathcal{S}_{1d}\cup\mathcal{S}_{2e}$, then  $Y_{\textbf{a}}^{-1}$ and $Y_{\textbf{a}}^{1}$ and bound $\QQ B^4s$. 
Note that Figure \ref{1b} depicts a cobordism similar to the one constructed in Figure \ref{1a}, which was described in the previous paragraph. However, the cobordisms constructed in Figures \ref{1c}$-$\ref{1e} are slightly different. In Figure \ref{1d}, we have a 2-handle cobordism from $Y_{\textbf{a}}^{\pm1}$ to $S^1\times S^2\#L(-4,1)$, which bounds a $\QQ S^1\times B^3$, since $L(-4,1)$ bounds a $\QQ B^4$ (\cite{liscalensspace}). Gluing this $\QQ S^1\times B^3$ to the cobordism yields the desired $\QQ B^4$. The cobordisms depicted in Figures \ref{1c} and \ref{1e} are built out of two 2-handles. These cobordisms are from $Y_{\textbf{a}}^{\pm1}$ to $S^1\times S^2\# S^1\times S^2$. Gluing these cobordisms to $S^1\times B^3\natural S^1\times B^3$ yields the desired $\QQ B^4s$. 

Lastly, suppose $\textbf{a}\in\mathcal{S}_2$. By Lemma \ref{revolemma}, $\overline{Y_{\textbf{a}}^{0}}=Y_{\textbf{d}}^{0}$. Thus once we show that $Y_{\textbf{a}}^{0}$ bounds a $\QQ B^4$, it will follow that $Y_{\textbf{d}}^{0}$ also bounds a $\QQ B^4$. Figures \ref{2a}$-$\ref{2e} show that if $\textbf{a}\in\mathcal{S}_2$, then $Y_{\textbf{a}}^{0}$ bounds a $\QQ B^4$. The $\QQ B^4s$ in almost all of the cases are constructed in very similar ways as in the negative cases. The last case, $Y^0_{(2,2,2,3)}$, is much simpler; Figure \ref{2f} shows that $Y^0_{(2,2,2,3)}=L(-4,1)$, which bounds a $\QQ B^4$.

As shown above, if $\textbf{a}\in\mathcal{S}_{1b}\cup\mathcal{S}_{1c}\cup\mathcal{S}_{1d}\cup\mathcal{S}_{1e}$, then $Y_{\textbf{d}}^{-1}$ bounds a $\QQ B^4$. However, as the next results will show, if $\textbf{a}\in\mathcal{S}_{1a}$, then $Y_{\textbf{d}}^{-1}$ does not necessarily bound a $\QQ B^4$. The key is that $|H_1(Y_{\textbf{a}}^{-1})|$ can be either even or odd when $\textbf{a}\in\mathcal{S}_{1a}$, but, in all other cases, $H_1(Y_{\textbf{a}}^{-1})$ has even order. Recall that $[b_1,\ldots,b_k]$ represents the Hirzebruch-Jung continued fraction (see the Appendix (Section \ref{appendix}) for details).

\begin{prop} Let $\textbf{a}=(b_1,\ldots,b_k,2,c_l,\ldots,c_1,2)\in\mathcal{S}_{1a}$, where $[b_1,\ldots,b_k]=\frac{p}{q}$. Then $|H_1(Y^{-1}_{\textbf{a}})|=|\textup{Tor}(H_1(\textbf{T}_{-A(\textbf{a})}))|=p^2$.\label{1(a)order}
\end{prop}

\begin{proof}
	The proof can be found in the Appendix (Section \ref{appendix}).
\end{proof}

\begin{lem} Let $\textbf{a}=(2,b_1,\ldots,b_k,2,c_l,\ldots,c_1)\in\mathcal{S}_{1a}$, where $[b_1,\ldots,b_k]=\frac{p}{q}$, and let $\textbf{d}=(d_1,\ldots,d_m)$ be the cyclic-dual of $\textbf{a}$. If $p$ is odd, then $Y_{\textbf{d}}^{-1}$ and $Y_{\textbf{a}}^{1}$ do not bound $\QQ B^4s$.\label{noboundlem}\end{lem}

\begin{proof} By Lemma \ref{revolemma}, $\overline{Y_{\textbf{d}}^{-1}}=Y_{\textbf{a}}^{1}$, so it suffices to show that $Y_{\textbf{a}}^{1}$ does not bound a $\QQ B^4$. Since $(b_1,\ldots,b_k)$ and $(c_1,\ldots,c_l)$ are linear-dual strings, it is clear that $\displaystyle\frac{I(\textbf{a})}{4}=-1$ (c.f. Remark \ref{ibounds}). By the calculations in Section \ref{baldwinsection}, $d(Y_{\textbf{a}}^{1},\mathfrak{s}_0)=1-\displaystyle\frac{I(\textbf{a})}{4}=2$. Since $p$ is odd, by Proposition \ref{1(a)order}, $|H_1(Y_{\textbf{a}}^{1})|=|H_1(Y_{\textbf{a}}^{-1})|$ has odd order and so $\mathfrak{s}_0$ extends over any $\QQ B^4$ bounded by $Y_{\textbf{a}}^{1}$. Thus if $Y_{\textbf{a}}^{1}$ bounds a $\QQ B^4$, then $d(Y_{\textbf{a}}^{1},\mathfrak{s}_0)=0$, which is not possible.\end{proof}

\begin{remark} By Lemma \ref{rationalballobstruction} and Theorem \ref{mainthm}, we already know that if  $\textbf{a}\in\mathcal{S}_{2c}$, then $Y^0_{\textbf{a}}$  bounds a $\QQ B^4$. However, by \cite{simone3}, the $\QQ B^4s$ constructed via Theorem \ref{mainthm} necessarily admit handlebody decompositions with 3-handles. On the other hand, the $\QQ B^4s$ constructed in this section do not contain 3-handles. Thus $Y^0_{\textbf{a}}$ bounds a $\QQ B^4s$ without 3-handles, even though $\textbf{T}_{A(\textbf{a})}$ only bounds $\QQ S^1\times B^3s$ containing 3-handles.\end{remark}

\begin{figure}[h!]
\centering
\begin{subfigure}{\textwidth}
\centering
\vspace{.3in}
\includegraphics[scale=.7]{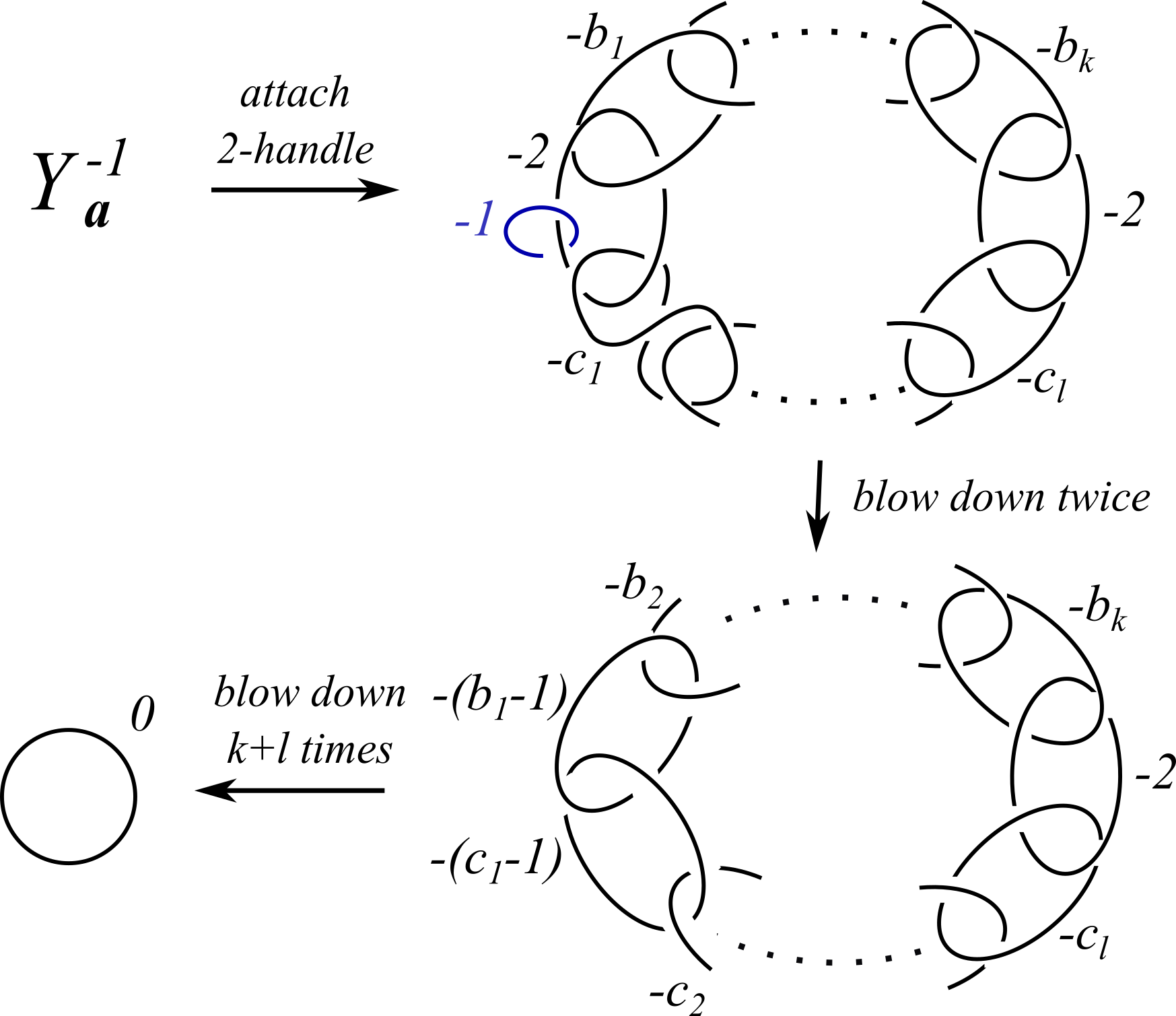}
\caption{If $\textbf{a}\in\mathcal{S}_{1a}$, then  $Y^{-1}_{\textbf{a}}$ bounds a $\QQ B^4$}\label{1a}
\end{subfigure}\vspace{.5cm}
\begin{subfigure}{\textwidth}
\centering
\includegraphics[scale=.7]{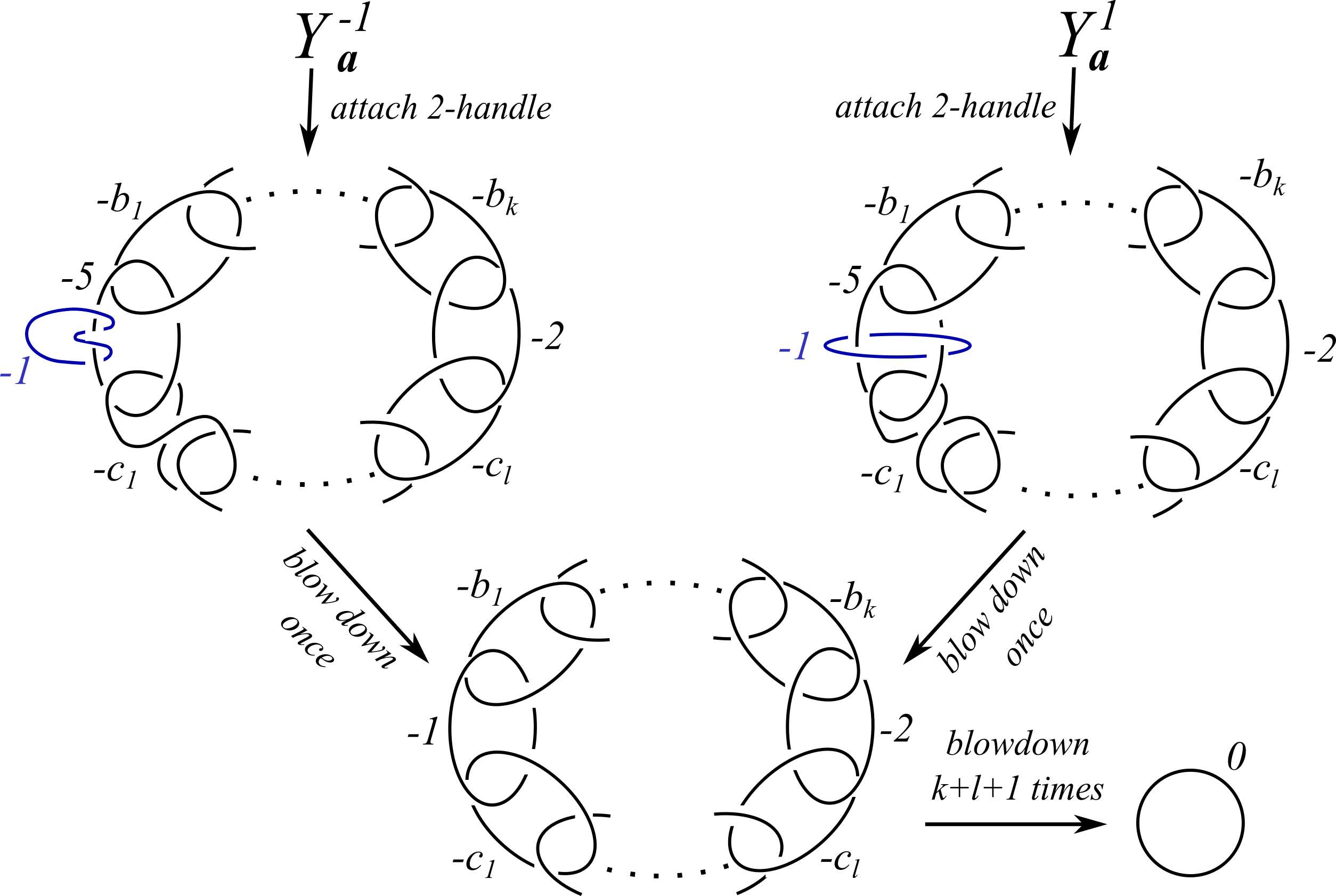}
\caption{If $\textbf{a}\in\mathcal{S}_{1b}$, then  $Y^{-1}_{\textbf{a}}$ and $Y^{1}_{\textbf{a}}$ bound  $\QQ B^4s$}\label{1b}
\end{subfigure}
\caption{The 3-manifolds in Theorem \ref{thm1}(\ref{(1)}) and (\ref{(2)}) bound rational balls}
\end{figure}

\begin{figure}\ContinuedFloat
\vspace{1.2in}
\begin{subfigure}{\textwidth}
		\centering
		\includegraphics[scale=.7]{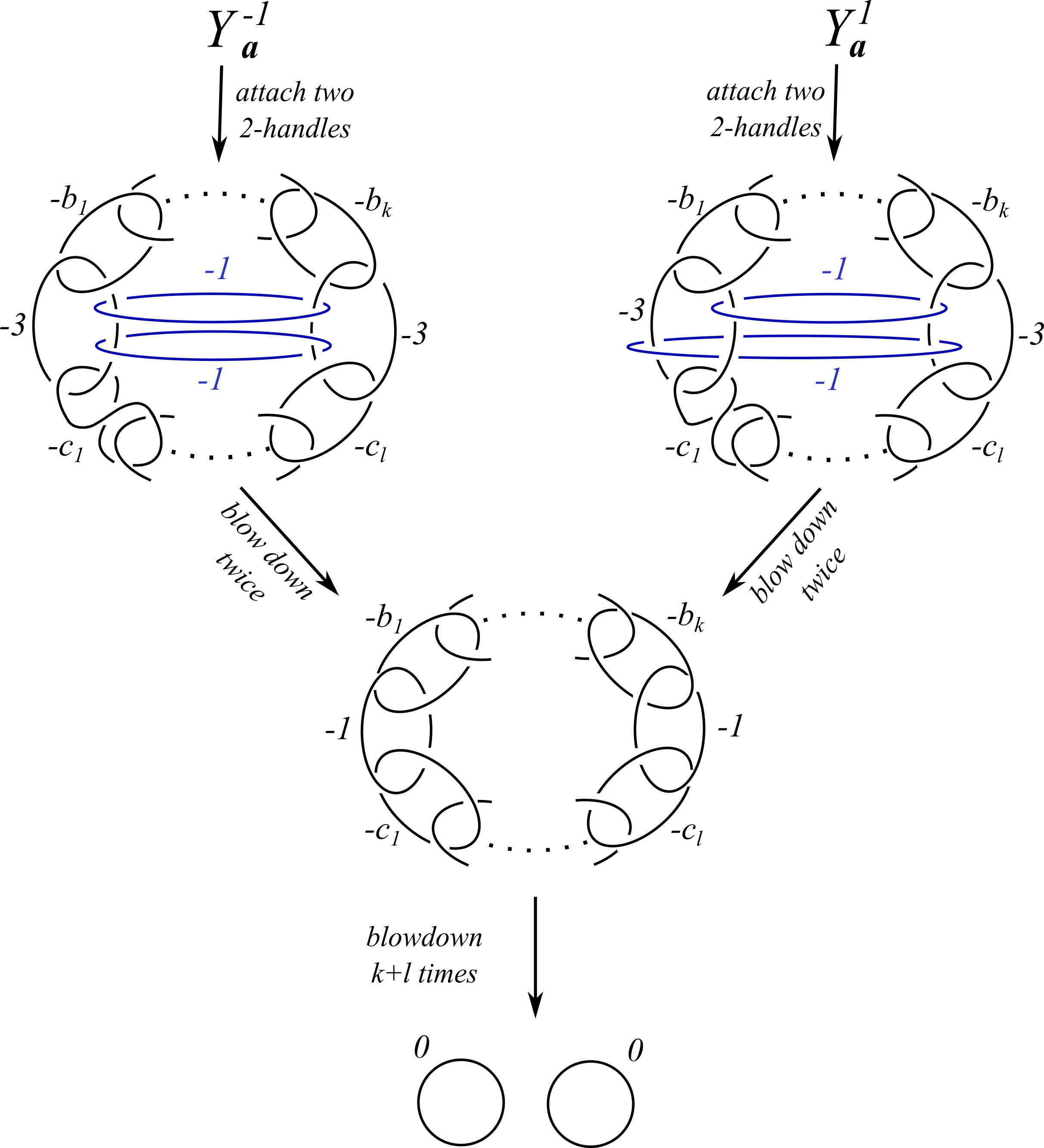}
		\caption{If $\textbf{a}\in\mathcal{S}_{1c}$, then  $Y^{-1}_{\textbf{a}}$ and $Y^{1}_{\textbf{a}}$ bound  $\QQ B^4s$}\label{1c}
\end{subfigure}
\caption{The 3-manifolds in Theorem \ref{thm1}(\ref{(1)}) and (\ref{(2)}) bound rational balls (continued)}\label{negativekirby}
\end{figure}

\begin{figure}\ContinuedFloat
	\begin{subfigure}{\textwidth}
		\centering
		\includegraphics[scale=.7]{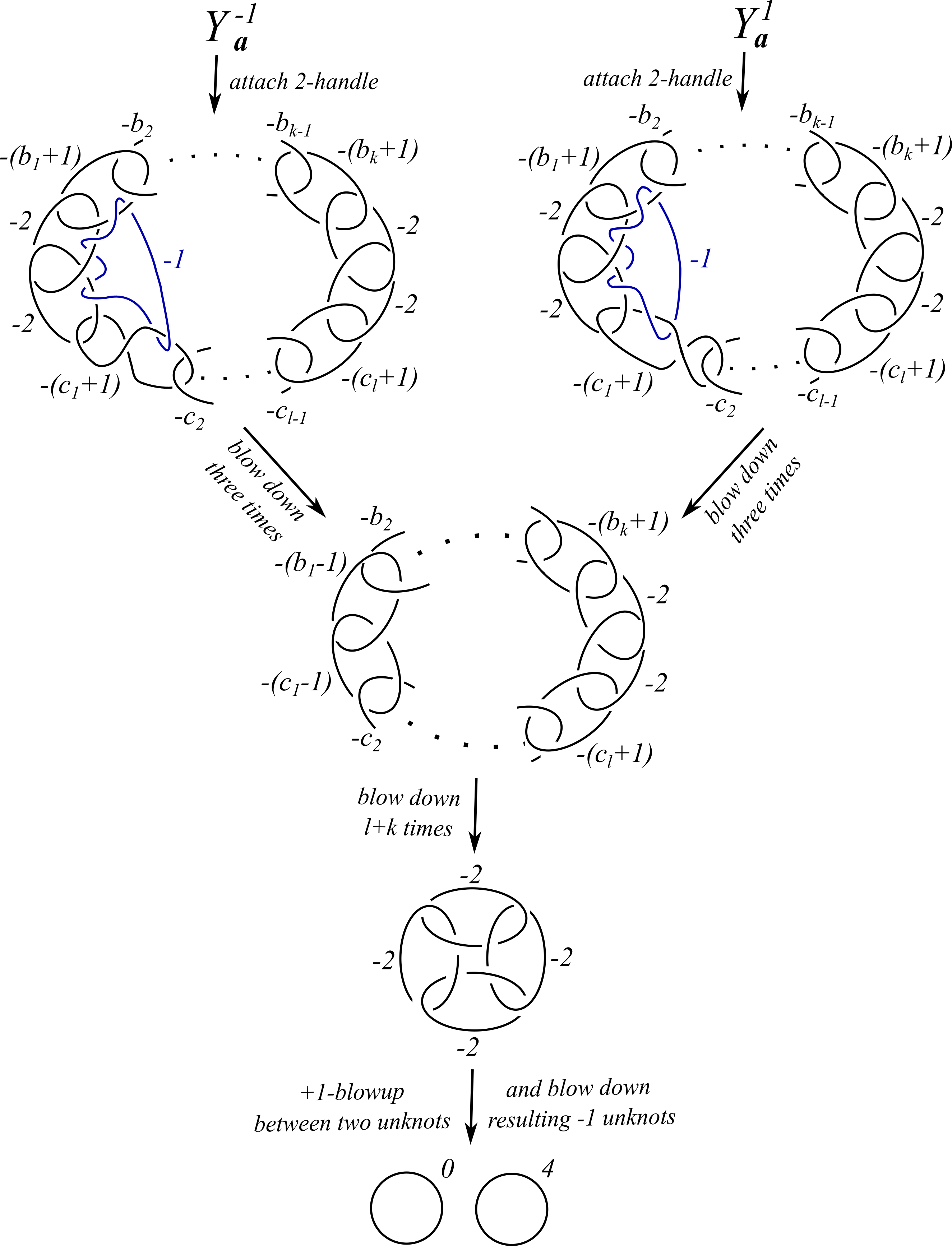}
		\caption{If $\textbf{a}\in\mathcal{S}_{1d}$, then  $Y^{-1}_{\textbf{a}}$ and $Y^{1}_{\textbf{a}}$ bound  $\QQ B^4s$}\label{1d}
	\end{subfigure}
	\caption{The 3-manifolds in Theorem \ref{thm1}(\ref{(1)}) and (\ref{(2)}) bound rational balls (continued)}\label{negativekirby}
\end{figure}

\begin{figure}\ContinuedFloat
	\begin{subfigure}{\textwidth}
		\centering
		\includegraphics[scale=.7]{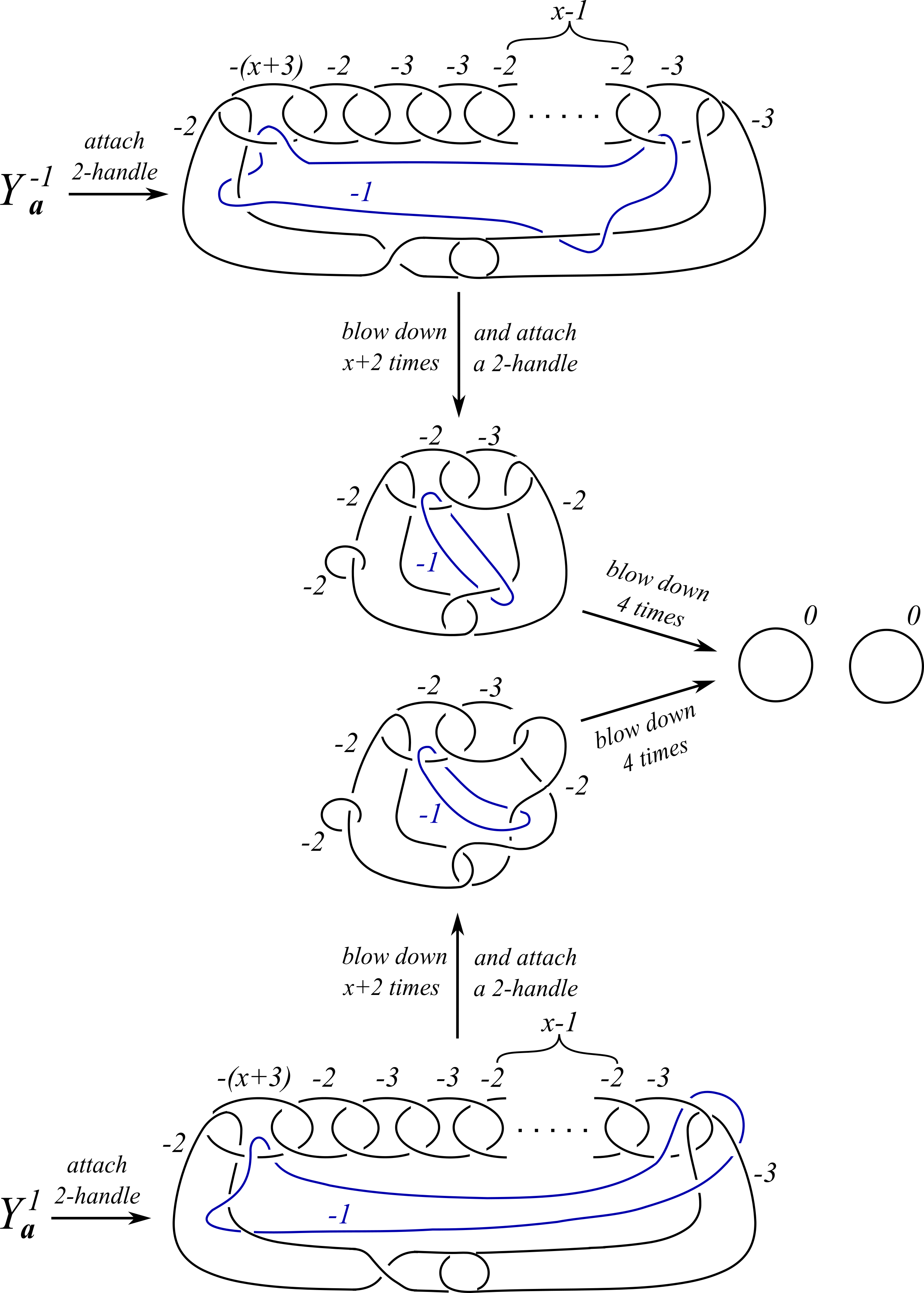}
		\caption{If $\textbf{a}\in\mathcal{S}_{1e}$, then  $Y^{-1}_{\textbf{a}}$ and $Y^{1}_{\textbf{a}}$ bound  $\QQ B^4s$}\label{1e}
	\end{subfigure}
	\caption{The 3-manifolds in Theorem \ref{thm1}(\ref{(1)}) and (\ref{(2)}) bound rational balls (continued)}\label{negativekirby}
\end{figure}

\newpage

\begin{figure}[h!]
	\vspace{.6in}
\begin{subfigure}{\textwidth}
\centering
\vspace{1.5cm}
\includegraphics[scale=.7]{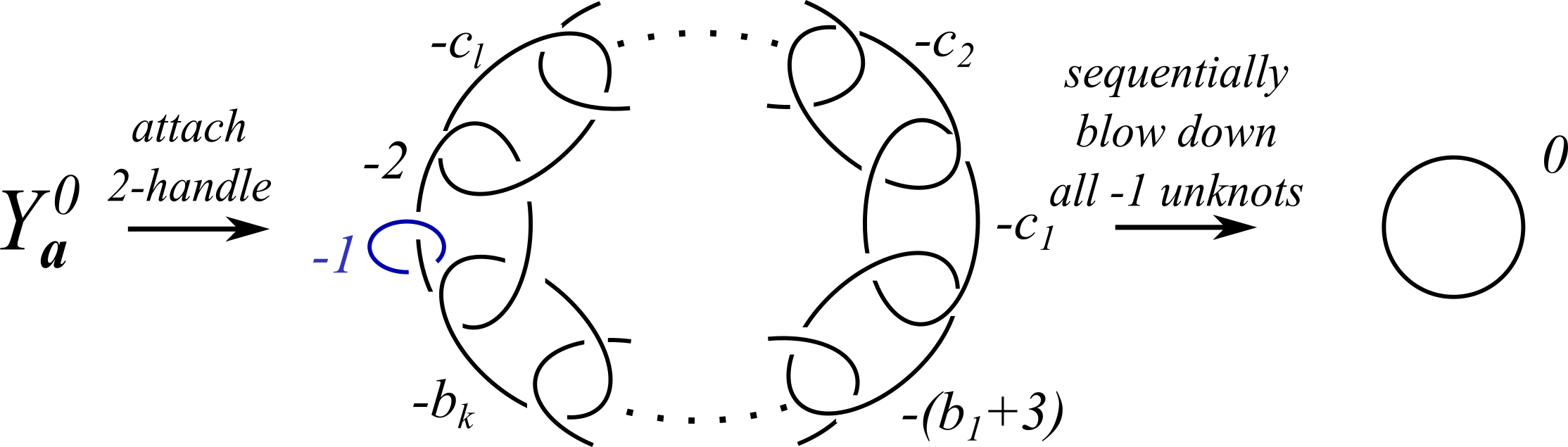}
\caption{If $\textbf{a}\in\mathcal{S}_{2a}$, then  $Y^{0}_{\textbf{a}}$ bounds a $\QQ B^4$}\label{2a}
\end{subfigure}\vspace{1cm}
\begin{subfigure}{\textwidth}
\centering
\includegraphics[scale=.7]{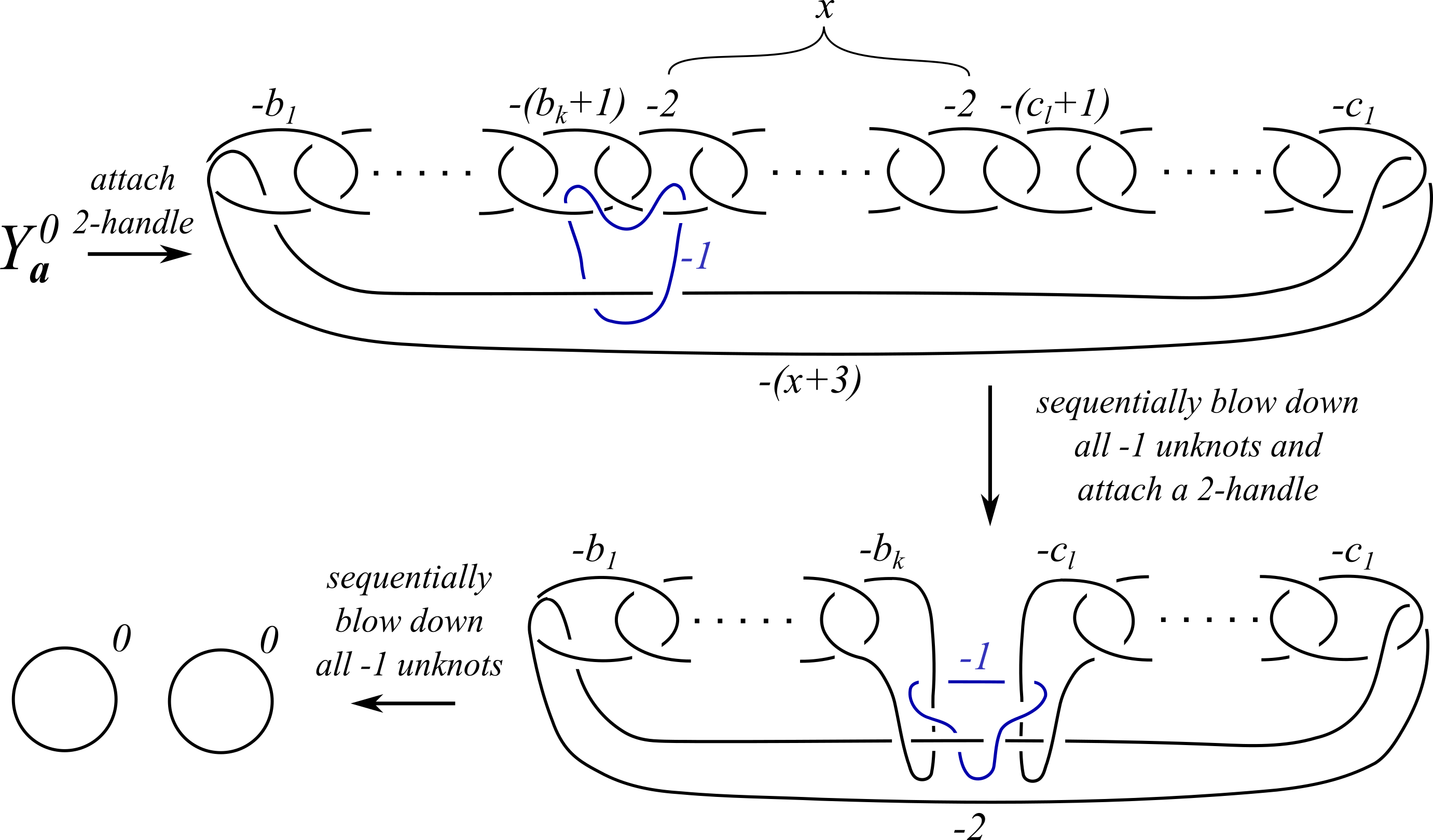}
\caption{If $\textbf{a}\in\mathcal{S}_{2b}$, then  $Y^{0}_{\textbf{a}}$ bounds a $\QQ B^4$}\label{2b}
\end{subfigure}
\caption{The 3-manifolds in Theorem \ref{thm1}(\ref{(3)}) bound rational balls (continued)}\label{positivekirby}
\end{figure}

\begin{figure}[h!]\ContinuedFloat
	\begin{subfigure}{\textwidth}
		\centering
		\includegraphics[scale=.7]{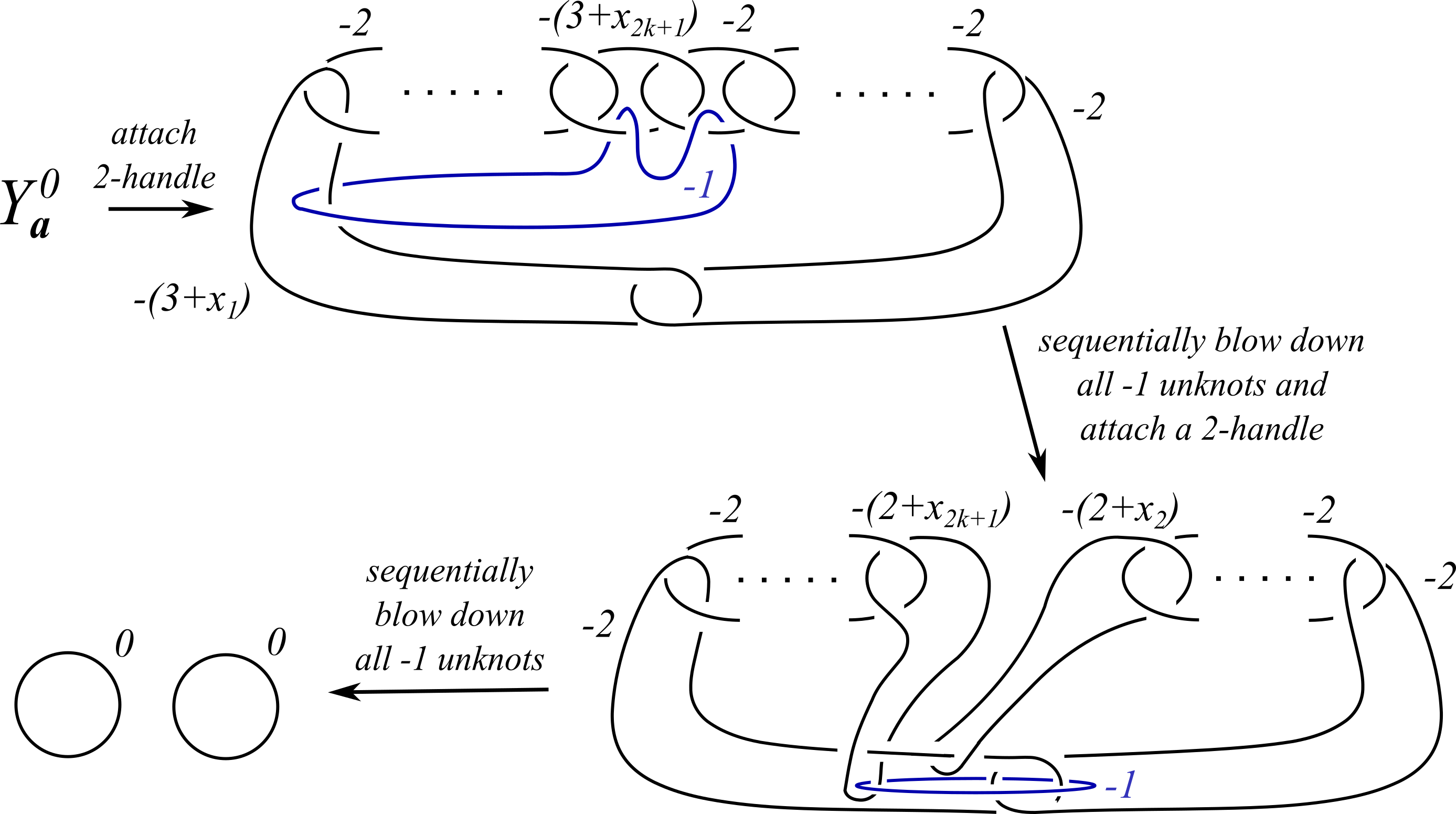}
		\caption{If $\textbf{a}\in\mathcal{S}_{2c}$, then  $Y^{0}_{\textbf{a}}$ bounds a $\QQ B^4$}\label{2c}
	\end{subfigure}\vspace{.5cm}
	\begin{subfigure}{\textwidth}
		\centering
		\includegraphics[scale=.7]{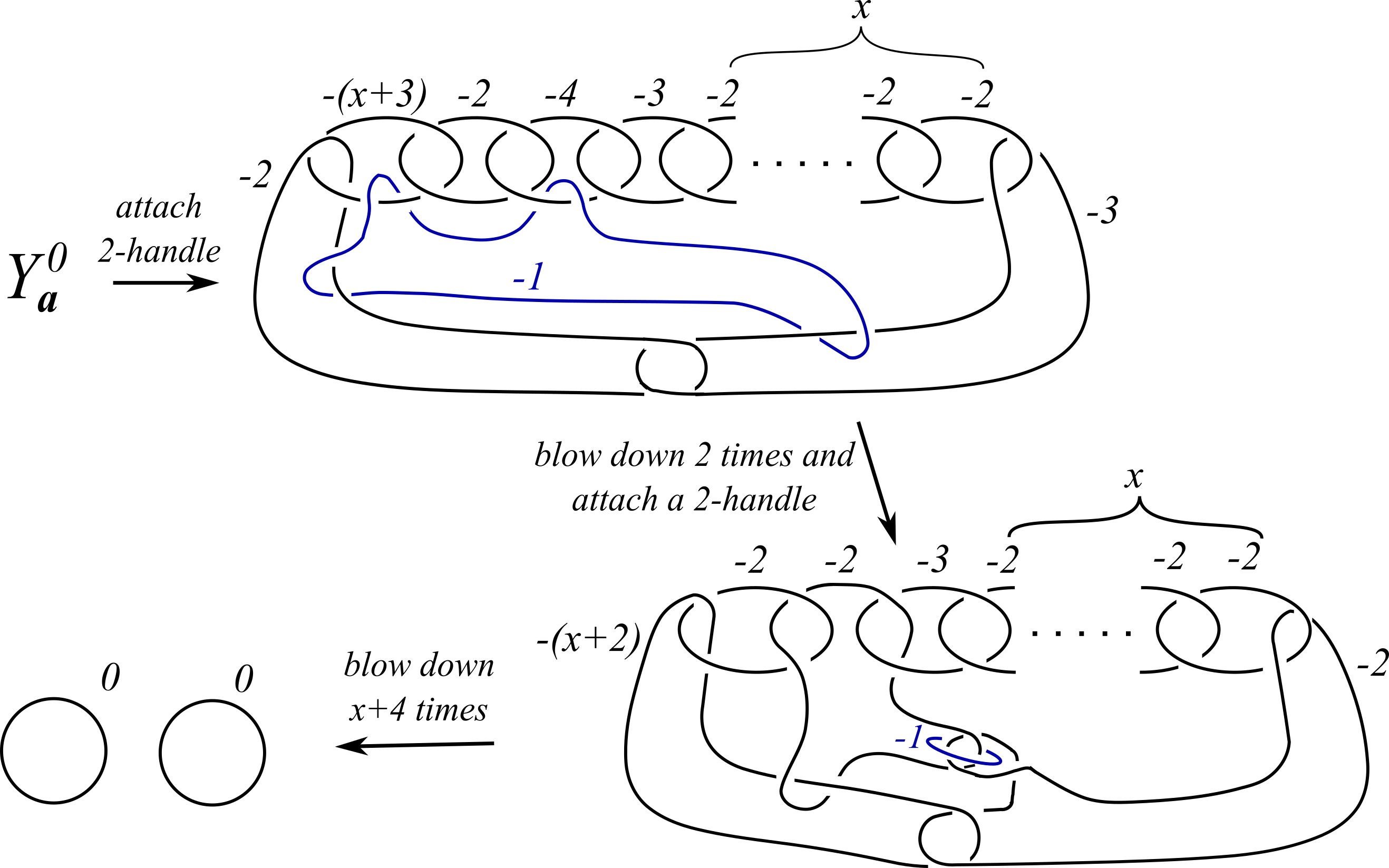}
		\caption{If $\textbf{a}\in\mathcal{S}_{2d}$, then  $Y^{0}_{\textbf{a}}$ bounds a $\QQ B^4$}\label{2d}
	\end{subfigure}
\caption{The 3-manifolds in Theorem \ref{thm1}(\ref{(3)}) bound rational balls (continued)}\label{positivekirby}
\end{figure}

\begin{figure}[h!]\ContinuedFloat
	\vspace{1in}
	\begin{subfigure}{\textwidth}
	\centering
	\includegraphics[scale=.7]{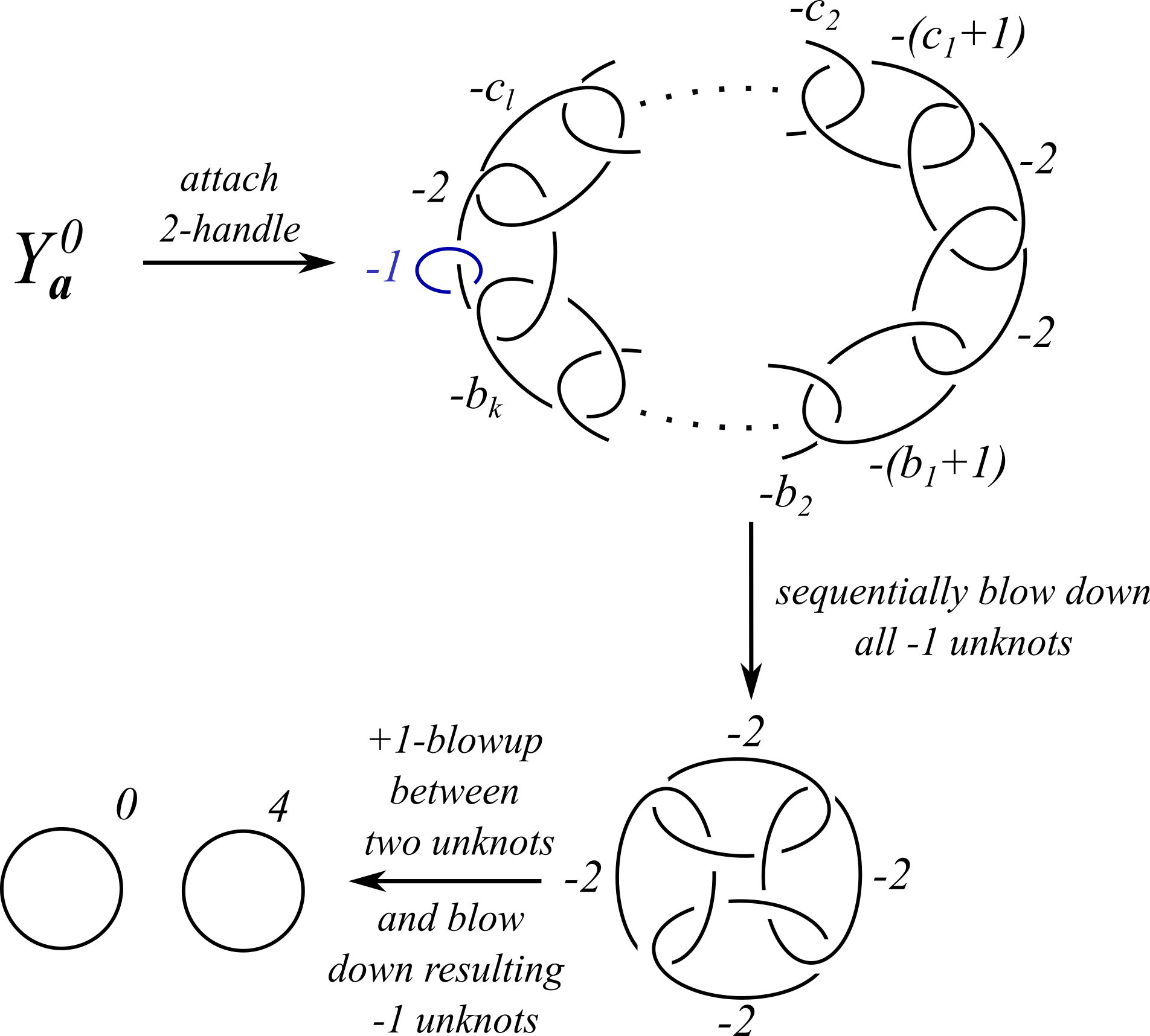}
	\caption{If $\textbf{a}\neq(3,2,2,2)\in\mathcal{S}_{2e}$, then  $Y^{0}_{\textbf{a}}$ bounds a $\QQ B^4$}\label{2e}
\end{subfigure}\vspace{1cm}
\begin{subfigure}{\textwidth}
	\centering
	\includegraphics[scale=.7]{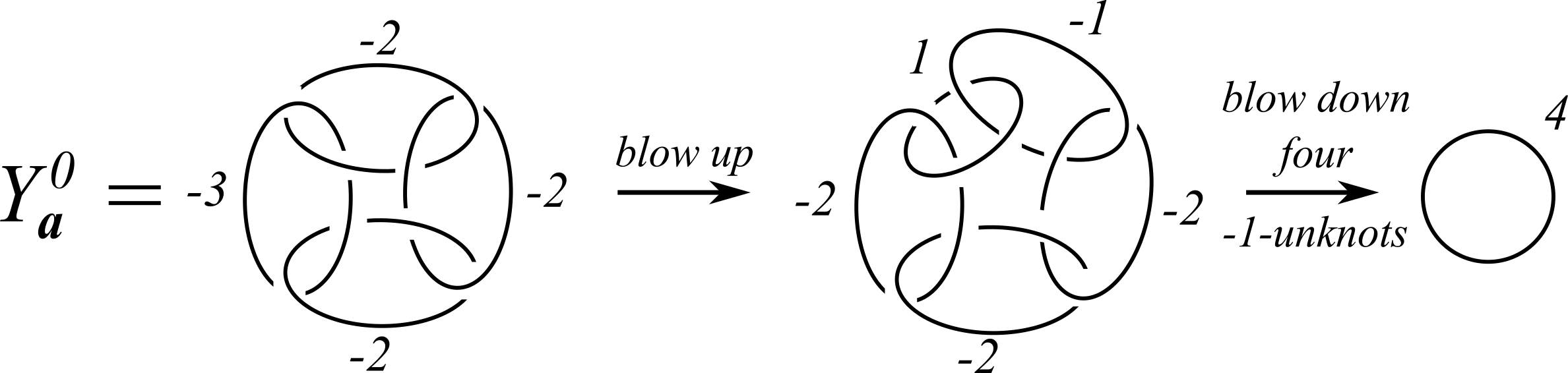}
	\caption{If $\textbf{a}=(3,2,2,2)\in\mathcal{S}_{2e}$, then  $Y^{0}_{\textbf{a}}$ bounds a $\QQ B^4$}\label{2f}
\end{subfigure}
\caption{The 3-manifolds in Theorem \ref{thm1}(\ref{(3)}) bound rational balls (continued)}\label{positivekirby}
\end{figure}

\clearpage


\section{Cyclic Subsets}\label{lattice}
The remainder of the sections are dedicated to proving the sufficient conditions of Theorem \ref{thm1}. In fact, we will prove something more general. We will show that if $t$ is odd and $Y_{\textbf{a}}^t$ bounds a $\QQ B^4$, then $\textbf{a}\in\mathcal{S}_1\cup\mathcal{O}$ or $\textbf{d}\in\mathcal{S}_1\cup\mathcal{O}$, and if $t$ is even and $Y_{\textbf{a}}^t$  bounds a $\QQ B^4$, then $\textbf{a}\in\mathcal{S}_2$ or $\textbf{d}\in\mathcal{S}_2$. For convenience, we recall the definition of these sets.\\

\noindent\textbf{Definition \ref{definition}.} Two strings are considered to be equivalent if one is a cyclic reordering and/or reverse of the other. Each string in each of the following sets is defined up to this equivalence. Moreover, in the following sets, strings of the form $(b_1,\ldots,b_k)$ and $(c_1,\ldots,c_l)$ are assumed to be linear-dual.
	\begin{itemize} 
		\item $ \mathcal{S}_{1a}=\{(b_1,\ldots,b_k,2,c_l,\ldots,c_1,2)\,|\, k+l\ge3\}$
		\item $ \mathcal{S}_{1b}=\{(b_1,\ldots,b_k,2,c_l,\ldots,c_1,5)\,|\, k+l\ge2\}$
		\item $ \mathcal{S}_{1c}=\{(b_1,\ldots,b_k,3,c_l,\ldots,c_1,3)\,|\, k+l\ge2\}$
		\item $ \mathcal{S}_{1d}=\{(2,b_1+1,b_2,\ldots,b_{k-1},b_k+1,2,2,c_l+1,c_{l-1},\ldots,c_2,c_1+1,2)\,|\, k+l\ge3\}$
		\item $ \mathcal{S}_{1e}=\{(2, 3+x, 2, 3, 3, 2^{[x-1]},3,3)\,|\, x\ge0 \textup{ and } (3,2^{[-1]},3):=(4)\}$
		\item $ \mathcal{S}_{2a}=\{(b_1+3,b_2,\ldots,b_k,2,c_l,\ldots,c_1)\}$
		\item $ \mathcal{S}_{2b}=\{(3+x,b_1,\ldots,b_{k-1},b_k+1,2^{[x]},c_l+1,c_{l-1},\ldots,c_1)\,|\, x\ge0\text{ and } k+l\ge2\}$
		\item $\mathcal{S}_{2c}=\{(b_1+1,b_2,\ldots,b_{k-1},b_k+1,c_1,\ldots,c_l)\,|\,k+l\ge2\}$
		\item $ \mathcal{S}_{2d}=\{(2,2+x,2,3,2^{[x-1]},3,4)\,|\, x\ge0 \textup{ and } (3,2^{[-1]},3):=(4)\}$
		\item $ \mathcal{S}_{2e}=\{(2,b_1+1,b_2,\ldots,b_k,2,c_l,\ldots,c_2,c_1+1,2),(2,2,2,3)\,|\, k+l\ge3\}$
		\item $\mathcal{O}=\{(6,2,2,2,6,2,2,2), (4,2,4,2,4,2,4,2), (3,3,3,3,3,3)\}$
	\end{itemize}
Moreover, $\mathcal{S}_1=\mathcal{S}_{1a}\cup\mathcal{S}_{1b}\cup\mathcal{S}_{1c}\cup\mathcal{S}_{1d}\cup\mathcal{S}_{1e}$, $\mathcal{S}_2=\mathcal{S}_{2a}\cup\mathcal{S}_{2b}\cup\mathcal{S}_{2c}\cup\mathcal{S}_{2d}\cup\mathcal{S}_{2e}$, and $\mathcal{S}=\mathcal{S}_1\cup\mathcal{S}_2$.\\

Also recall, to remove the necessity of different cases, if $\textbf{a}\in\mathcal{S}_{1d}\cup\mathcal{S}_{2c}$ and $k=1$, then the substring $(b_1+1,b_2,\ldots,b_{k-1},b_k+1)$ is understood to be the substring $(b_1+2)$.

First suppose $n=1$ and let $\textbf{a}=(a_1)$, where $a_1\ge 3$. Then $L_1^0$ and $L_1^{-1}$ are both the unknot and so $Y_{(a_1)}^{0}=L(a_1-2,1)$ and $Y_{(a_1)}^{-1}=L(a_1+2,1)$ (c.f. Figure \ref{chainlinks}). By Lisca's classification of lens spaces that bound $\QQ B^4s$ (\cite{liscalensspace}), the only such lens spaces that bound $\QQ B^4s$ are $L(1,1)=S^3$ and $L(4,1)$. Thus $Y^{-1}_{(a_1)}$ does not bound a $\QQ B^4$ for all $a_1\ge3$ and $Y_{(a_1)}^0$ bounds a $\QQ B^4s$ if and only if $a_1=3$ or $a_1=6$. In the former case, $\textbf{a}=(3)\in\mathcal{S}_{2c}$, and in the latter case, $\textbf{d}=(2,2,2,3)\in\mathcal{S}_{2e}$.

We now assume the length of $\textbf{a}$ is at least 2. Throughout, we will consider the standard negative definite intersection lattice $(\ZZ^n,-I_n)$. Let $\{e_1,\ldots,e_n\}$ be the standard basis of $\ZZ^n$. Then, with respect to the product $\cdot$ given by $-I_n$, we have $e_i\cdot e_j=-\delta_{ij}$ for all $i,j$. 
We begin by recalling definitions and results from \cite{liscalensspace} and introducing new terminology for our purposes.

We consider two subsets $S_1, S_2\subset\ZZ^n$ to be the same if $S_2$ can be obtained by applying an element of $\Aut(\ZZ^n)$ to $S_1$.
Let $S=\{v_1,\ldots,v_n\}\subset\ZZ^n$ be a subset. 
We call each element $v_i\in S$ a \textit{vector} and we call the string of integers $(a_1,\ldots,a_n)$ defined by $a_i=-v_i\cdot v_i$ the \textit{string associated} to $S$.
Two vectors $z,w\in S$ are called \textit{linked} if there exists $e\in \ZZ^n$ such that $e\cdot e=-1$ and $z\cdot e, w\cdot e\neq0$. A subset $S$ is called \textit{irreducible} if for every pair of vectors $v,w\in S$, there exists a finite sequence of vectors $v_1=v,v_2,\ldots, v_k=w\in S$ such that $v_i$ and $v_{i+1}$ are linked for all $1\le i\le k-1$.

\begin{definition} A subset $S=\{v_1,\ldots,v_n\}\in\mathbb{Z}^n$ is:
	\begin{itemize}
		\item \textit{good} if it is irreducible and $v_i\cdot v_j=
		\begin{cases}
		-a_i\le-2 & \text{if } i=j \\
		0 \text{ or } 1 & \text{if } |i-j|=1 \\
		0 & \text{otherwise}
		\end{cases}
		$
		\vspace{.2cm}
		\item \textit{standard} if $v_i\cdot v_j= 
		\begin{cases}
		-a_i\le-2 & \text{if } i=j \\
		1 & \text{if } |i-j|=1 \\
		0 & \text{otherwise}
		\end{cases}
		$
\end{itemize}
\end{definition}

Note that by definition standard subsets are good. If $S$ is a good subset, then a vertex $v\in S$ is called: \textit{isolated} if $v\cdot w=0$ for all $w\in S\setminus\{v\}$; \textit{final} if there exists exactly one vertex $w\in S\setminus\{v\}$ such that $v\cdot w=1$; and \textit{internal} otherwise. A \textit{component} of a good subset $G$ is a subset of $G$ corresponding to a connected component of the intersection graph of $G$ (which is the graph consisting of vertices $v_1,\ldots,v_n$ and an edge between two vertices $v_i$ and $v_j$ if and only if $v_i\cdot v_j=1$).
	
\begin{definition} A subset $S=\{v_1,\ldots,v_n\}\in\mathbb{Z}^n$ is:
\begin{itemize}
\item  \textit{negative cyclic} if either 
\begin{itemize}
      \item[(1)] $n=2$ and $v_i\cdot v_j=\begin{cases}
      -a_i\le-2 & \text{if } i=j \\
      0 & \text{if } i\neq j \\
      \end{cases}$ \hspace{.5cm} or \vspace{.2cm}
      \item[(2)] $n\ge3$ and there is a cyclic reordering of $S$ such that\\ $v_i\cdot v_j=\begin{cases}
      -a_i\le-2 & \text{if } i=j \\
      1 & \text{if } |i-j|=1 \\
      -1 & \text{if } i\neq j\in\{1,n\}\\
      0 & \text{otherwise}
      \end{cases}$  \vspace{.2cm}
\end{itemize}
\item  \textit{positive cyclic} if $-a_i\le -3$ for some $i$ and either 
\begin{itemize}
      \item[(1)] $n=2$ and $v_i\cdot v_j=\begin{cases}
      -a_i\le-2 & \text{if } i=j \\
      2 & \text{if } i\neq j \\
      \end{cases}$ \hspace{.5cm} or \vspace{.2cm}
      \item[(2)] $n\ge3$ and there is a cyclic reordering of $S$ such that\\ $v_i\cdot v_j=\begin{cases}
      -a_i\le-2 & \text{if } i=j \\
      1 & \text{if } |i-j|=1 \\
      1 & \text{if } i\neq j\in\{1,n\}\\
      0 & \text{otherwise}
      \end{cases}$   \vspace{.2cm}
\end{itemize}
\item \textit{cyclic} if $S$ is negative or positive cyclic.
\end{itemize}
\end{definition}

If $S$ is cyclic, then the indices of each vertex are understood to be defined modulo $n$ (e.g. $v_{n+1}=v_1$). If $v_i\cdot v_{i+1}=\pm 1$, then we say that $v_i$ and $v_j$ have a \textit{positive/negative intersection}. Moreover, if $S$ is cyclic and $S'$ is obtained from $S$ by reversal and/or cyclic reordering, then we consider $S$ and $S'$ to be the same subset. In this way, associated strings of cyclic subsets are well-defined up to reversal and cyclic-reordering. 

\begin{remark} By standard linear algebra, it is easy to see that if $S$ is good, cyclic, or the union of a good subset and a cyclic subset, then $S$ forms a linearly independent set in $\ZZ^n$ (c.f. Remark 2.1 in \cite{liscalensspace}).\label{LIremark}\end{remark}

\begin{remark} Suppose $S=\{v_1,\ldots,v_n\}$ is a cyclic subset. Then by replacing $v_k$ with $v_k'=-v_k$, we obtain a new subset $\hat{S}=\{v_1,\ldots,v_{k-1},v_k',v_{k+1},\ldots,v_n\}$ such that $v_{k-1}\cdot v_{k}'=-v_{k-1}\cdot v_{k}$ and $v_k'\cdot v_{k+1}=-v_k\cdot v_{k+1}$. Notice that $S$ and $\hat{S}$ have the same associated strings. Thus we can change the number of positive and negative intersections of $S$ without changing the associated string. Conversely, any subset of the form $S=\{v_1,\ldots,v_n\}$ where $n\ge 3$ and 
\begin{center} $v_i\cdot v_j=\begin{cases}
      -a_i\le-2 & \text{if } i=j \\
      \pm 1 & \text{if } |i-j|=1 \\
      \pm 1 & \text{if } i\neq j\in\{1,n\}\\
      0 & \text{otherwise}
      \end{cases}$\end{center}
can modified into a positive or negative cyclic subset by changing the signs of select vertices. In particular, for any negative cyclic subset, the negative intersection can be moved at will by negating select vertices. 

Similarly, any irreducible subset of the form $G=\{v_1,\ldots,v_n\}$, where 
\begin{center} $v_i\cdot v_j=\begin{cases}
	-a_i\le-2 & \text{if } i=j \\
	\pm 1 & \text{if } |i-j|=1 \\
	0 & \text{otherwise}
	\end{cases}$\end{center}
can be modified into a good subset by changing the signs of select vertices. In Section \ref{p1=0}, we will often create such subsets and assume that they are good, without specifying the need to possibly negative select vertices first.
\label{moveneg}\end{remark}

\begin{definition} Let $S=\{v_1,\ldots,v_n\}\subset\ZZ^n$ be a subset with $v_i\cdot v_i=-a_i$. We define the following notation. $$I(S):=\displaystyle\sum_{i=1}^n(a_i -3)$$
	$$E^S_i:=\{j\text{ }|\text{ } v_j\cdot e_i\neq0\}$$
	$$ V^S_{i}:=\{j\text{ }|\text{ } v_i\cdot e_j\neq 0\}$$
	$$p_i(S):=\Big|\{j\text{ }|\text{ } |E_j^S|=i\}\Big|$$

 In some cases we will drop the superscript $S$ from the above notation if the subset being considered is understood. \end{definition}

\begin{remark} In \cite{liscalensspace}, Lisca classified all standard subsets of $\ZZ ^n$ with $I(S)<0$. The results in the next three sections rely in part on his classification of standard subsets. We will review his classification in Section \ref{liscaclassification}.\end{remark}

\begin{example} The subset $S=\{e_1-e_2,e_2-e_3,\ldots,e_{n-1}-e_n,e_n+e_1\}\subset\ZZ^n$, $n\ge 2$, is a negative cyclic subset with associated string $(2^{[n]})$. Moreover, $I(S)=-n$, $p_2(S)=n$, and $p_j(S)=0$ for all $j\neq 2$. When $n=4$, there is an alternate subset with associated string $(2,2,2,2)$, namely $S'=\{e_1-e_2,e_2-e_3,-e_2-e_1,e_1+e_4\}$, which satisfies $p_1(S')=p_3(S')=2$. This latter subset will be used to construct the family strings in $\mathcal{S}_{1a}$. \label{example}\end{example}

Let $\textbf{a}=(a_1,\ldots,a_n)$. The rational sphere $Y^t_{\textbf{a}}$ is the boundary of the negative definite 2-handlebody $P^t$ whose handlebody diagram is given in Figure \ref{2handlebody}. Let $Q_{P^t}$ denote the intersection form of $P^t$. Note that $Q_{P^t}$ depends only on the parity of $t$. Further suppose $Y^t_{\textbf{a}}$ bounds a rational homology ball $B$. Then the closed 4-manifold $X^t=P^t\cup B$ is negative definite. By Donaldson's Diagonalization Theorem \cite{donaldson}, the intersection lattice $(H_2(X^t),Q_{X^t})$ is isomorphic to the standard negative definite lattice $(\ZZ^n,-I_n)$. Thus the intersection lattice $(H_2(P^t),Q_{P^t})$ must embed in $(\ZZ^n,-I_n)$. The existence of such an embedding implies the existence of a cyclic subset $S\subset \ZZ^n$ with associated string $(a_1,\ldots,a_n)$. Thus our goal is to classify all cyclic subsets of $\ZZ^n$, where $n\ge 2$.

\begin{figure}
	\centering
	\includegraphics[scale=.6]{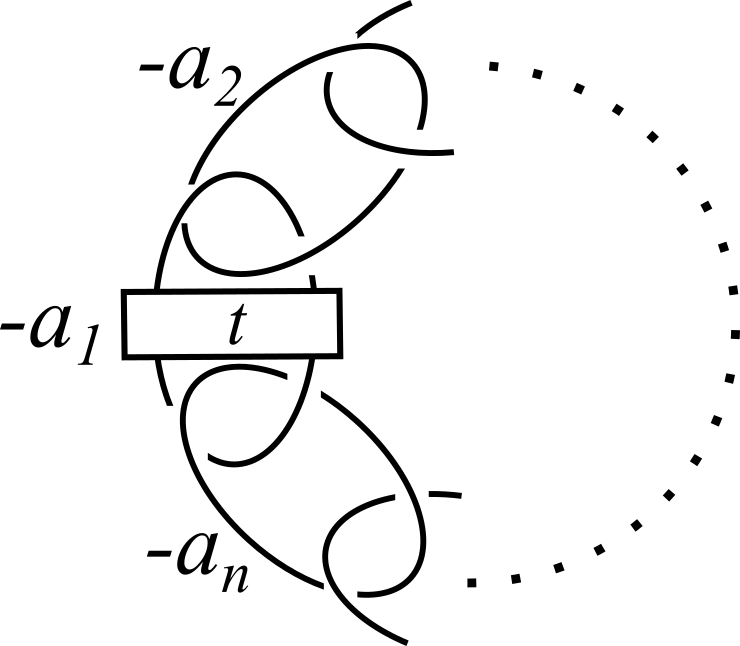}
	\caption{A 4-manifold with boundary $Y^t_{\textbf{a}}$}\label{2handlebody}
\end{figure}

Recall that by reversing the orientation of $Y^t_{\textbf{a}}$, we obtain the $\overline{Y^t_{\textbf{a}}}=Y^{-t}_{\textbf{d}}$, where $\textbf{d}=(d_1,\ldots,d_m)$ is the cyclic-dual of $(a_1,\ldots,a_n)$ (Section \ref{reverseorientation}). In particular, $(a_1,\ldots,a_n)$ is of the form $(2^{[m_1]},3+n_1,\ldots,2^{[m_k]},3+n_k)$ if and only if $(d_1,\ldots,d_m)$ is of the form $(3+m_1,2^{[n_1]},\ldots,3+m_k,2^{[n_k]})$. If $S$ and $\overline{S}$ denote the cyclic subsets associated to $(a_1,\ldots,a_n)$ and $(d_1,\ldots,d_m)$, respectively, then $I(S)+I(\overline{S})=0$. Now since $Y^t_{\textbf{a}}$ bounds a $\QQ B^4$ if and only if  $Y^{-t}_{\textbf{d}}$ bounds a $\QQ B^4$, we will focus our attention on subsets satisfying $I(S)\le 0$. 
The following theorem is the main result of our lattice analysis. 

\begin{thm} Let $S$ be a cyclic subset such that $I(S)\le 0$. Then $S$ is either negative with associated string in $\mathcal{S}_1\cup\mathcal{O}\cup\{(2^{[n]})\,|\,n\ge2\}$ or positive with associated string in $\mathcal{S}_2$.
\label{latticethm}\end{thm}

\begin{proof}
The theorem follows from Example \ref{example} and Propositions \ref{prop1}, \ref{prop2}, \ref{prop:p2>0}, which will be proven in Sections \ref{p1>0} and \ref{p1=0}.
\end{proof}

We can now prove Theorem \ref{thm1}, which we recall here for convenience.\vspace{.2cm}

\noindent\textbf{Theorem \ref{thm1}.} \textit{Let $\textbf{a}=(a_1,\ldots,a_n)$, where $n\ge1$, $a_i\ge 2$ for all $i$, and $a_j\ge 3$ for some $j$, and let $\textbf{d}$ be the cyclic-dual of $\textbf{a}$.
	\begin{enumerate} 
		\item Suppose $\textbf{d}\notin\mathcal{S}_{1a}\cup\mathcal{O}$. Then $Y_{\textbf{a}}^{-1}$ bounds a $\QQ B^4$ if and only if $\textbf{a}\in\mathcal{S}_1$ or $\textbf{d}\in\mathcal{S}_{1b}\cup\mathcal{S}_{1c}\cup\mathcal{S}_{1d}\cup\mathcal{S}_{1e}$.\label{(1)}
		\item Suppose $\textbf{a}\notin\mathcal{S}_{1a}\cup\mathcal{O}$. Then $Y_{\textbf{a}}^{1}$ bounds a $\QQ B^4$ if and only if $\textbf{d}\in\mathcal{S}_1$ or $\textbf{a}\in\mathcal{S}_{1b}\cup\mathcal{S}_{1c}\cup\mathcal{S}_{1d}\cup\mathcal{S}_{1e}$.\label{(2)}
		\item $Y_{\textbf{a}}^0$ bounds a $\QQ B^4$ if and only if $\textbf{a}\in \mathcal{S}_2$ or $\textbf{d}\in \mathcal{S}_2$.\label{(3)}
	\end{enumerate}}

\begin{proof} The sufficient conditions of Theorem \ref{thm1} follow from the calculations in Section \ref{spheres}. The necessary conditions of Theorem \ref{thm1} follow from Theorem \ref{latticethm} and the fact that $Y^t_{\textbf{a}}$ bounds a $\QQ B^4$ if and only if  $Y^{-t}_{\textbf{d}}$ bounds a $\QQ B^4$.
\end{proof}

The proof of Theorem \ref{latticethm} will span the next three sections. The proof will begin in earnest in Section \ref{p1>0}. The proof applies two strategies. The first will be to reduce certain  cyclic subsets to good subsets and standard subsets and appeal to Lisca's work in \cite{liscalensspace} and \cite{liscasumslensspaces}. The second will be to reduce certain  cyclic subsets (via \textit{contractions}) to a small list of base cases. In the upcoming subsection, we will recall Lisca's classification standard subsets. In the subsequent subsection, we will describe how to perform contractions and list the relevant base cases. In the final subsection, we will prove a few preliminary lemmas that will be useful going forward.


\subsection{Lisca's Standard and Good Subsets}\label{liscaclassification} In Section \ref{p1=0}, we will construct good subsets and standard subsets satisfying $I<0$ from cyclic subsets, thus reducing the problem of classifying certain cyclic subsets to Lisca's work in \cite{liscalensspace} and \cite{liscasumslensspaces}. In this section, we collect relevant results proved by Lisca. The first two propositions can be found in Sections 3-7 in \cite{liscalensspace}. In particular, the `moreover' statements in Proposition \ref{liscasummary2} are obtained by examining the proofs of Lemmas 7.1-7.3 in \cite{liscalensspace}. 

\begin{prop} Let $T=\{v_1,\ldots,v_n\}$ be a  standard subset with $I(T)<0$. Then:
\begin{enumerate}
\item $I(T)\in\{-1,-2,-3\}$;
\item $|v_i\cdot e_j|\le1$ for all $i,j$;
\item $p_1(T)=1$ if and only if $I(T)=-3$ and if $p_1(T)=0$, then $p_2(T)>0$;
\item If $I(T)=-3$, then $p_1(T)=p_2(T)=1$ and $p_3(T)=n-2$;
\item If $I(T)=-2$, then $p_2(T)=3$, $p_4(T)=1$, and $p_3(T)=n-4$; and
\item If $I(T)=-1$, then $p_2(T)=2$, $p_4(T)=1$ and $p_3(T)=n-3$.
\end{enumerate}
\label{liscasummary1}\end{prop}

\begin{prop} Let $T$ be standard with $I(T)<0$. Let $x,y\ge0$. 
\begin{enumerate}
\item If $I(T)=-3$, then: if $E_i=\{s\}$, then $v_s$ is internal (i.e. $1< s< n$) and $v_s\cdot v_s=-2$; if $|E_j|=2$, then $E_j=\{1,n\}$; either $v_1\cdot v_1=-2$ or $v_n\cdot v_n=-2$; and $v_1\cdot e_j=-v_n\cdot e_j$. Moreover, $T$ has associated string of the form $(b_1,\ldots,b_k,2,c_l,\ldots,c_1)$, where $(b_1,\ldots,b_k)$ and $(c_1,\ldots,c_l)$ are linear-dual strings.
\item If $I(T)=-2$, then (up to reversal) $T$ has associated string of the form: 
\begin{enumerate}[(a)]
\item $(2^{[x]}, 3, 2+y, 2+x, 3, 2^{[y]})$; 
\item $(2^{[x]}, 3+y,2,2+x,3,2^{[y]})$; or 
\item $(b_1,\ldots ,b_{k-1},b_k+1,2,2,c_l+1,c_{l-1},\ldots,c_1)$, where $(b_1,\ldots,b_k)$ and $(c_1,\ldots,c_l)$ are  linear-dual strings. 
\end{enumerate}
Moreover, up to the action of $\Aut(\ZZ^n)$, the corresponding embeddings are of the form: 
\begin{enumerate}[(a)]
\item $\{e_{x+4}-e_{x+3},e_{x+3}-e_{x+2},\ldots,e_5-e_4,e_4-e_2-e_3,e_2+e_1+\sum_{i=x+5}^{x+y+4}e_i,\\-e_2-e_4-\sum_{i=5}^{x+4}e_i,e_2-e_1-e_3,e_1-e_{x+5},e_{x+5}-e_{x+6},\ldots,e_{x+y+3}-e_{x+y+4}\}$
\item $\{e_{x+4}-e_{x+3},e_{x+3}-e_{x+2},\ldots,e_5-e_4,e_4-e_2-e_3-\sum_{i=x+5}^{x+y+4}e_i,e_2+e_1,\\-e_2-e_4-\sum_{i=5}^{x+4}e_i,e_2-e_1-e_3,e_3-e_{x+5},e_{x+5}-e_{x+6},\ldots,e_{x+y+3}-e_{x+y+4}\}$
\item $\{u_1,\ldots,u_{k-1}, u_k +e_4-e_2-e_3, e_2+e_1,-e_2-e_4,e_2-e_1-e_3+w_1,w_2,\ldots,w_l\}$, where $k+l\ge 3$, $u_k=0$ or $w_1=0$, $|E_1|=|E_4|=2$. Furthermore (up to reversal), we may assume that $u_1^2=-2$; consequently, there exist integers $i,j$ such that $|E_i|=3$, $|E_j|=2$, $u_1\cdot e_i=-u_2\cdot e_i=-w_l\cdot e_i=1$, and $|u_1\cdot e_j|=|w_l\cdot e_j|=1$. 
\end{enumerate}
\item If $I(T)=-1$, then (up to reversal) $T$ has associated string of the form:
\begin{enumerate}[(a)]
\item $(2+x,2+y,3,2^{[x]},4, 2^{[y]})$; 
\item $(2+x,2,3+y,2^{[x]},4,2^{[y]})$; or 
\item $(3+x,2,3+y,3,2^{[x]},3,2^{[y]})$.
\end{enumerate}
Moreover, up to the action of $\Aut(\ZZ^n)$, the corresponding embeddings are of the form: 
\begin{enumerate}[(a)]
\item $\{e_2+e_4+\sum_{i=5}^{x+4}e_i,e_1-e_2+\sum_{i=x+5}^{x+y+4}e_i,e_2-e_3-e_4,e_4-e_5,e_5-e_6,\ldots,e_{x+3}-e_{x+4},\\e_{x+4}-e_1-e_2-e_3,e_1-e_{x+5},e_{x+5}-e_{x+6},\ldots,e_{x+y+3}-e_{x+y+4}\}$
\item $\{e_2+e_4+\sum_{i=5}^{x+4}e_i,e_1-e_2,e_2-e_3-e_4-\sum_{i=x+5}^{x+y+4}e_i,e_4-e_5,\ldots,e_{x+3}-e_{x+4},\\e_{x+4}-e_1-e_2-e_3,e_3-e_{x+5},e_{x+5}-e_{x+6},\ldots,e_{x+y+3}-e_{x+y+4}\}$
\item $\{e_1-e_2-e_5-\sum_{i=6}^{x+5}e_i,e_2+e_3,-e_2-e_1-e_4-\sum_{i=x+6}^{x+y+5}e_i,-e_5+e_2-e_3,\\e_5-e_6,e_6-e_7,\ldots,e_{x+4}-e_{x+5},e_{x+5}+e_1-e_4,e_4-e_{x+6},e_{x+6}-e_{x+7},\ldots, \\e_{x+y+4}-e_{x+y+5}\}$
\end{enumerate}
\end{enumerate}
\label{liscasummary2}
\end{prop}

The following proposition follows from the proof of the main theorem in \cite{liscasumslensspaces} (``First Case: $S$ irreducible" on page 2160) and Lemma 6.2 in \cite{liscalensspace} (see also Lemmas 6.6 in \cite{acetogollalarsonlecuona}). See Definition 4.1 in \cite{liscalensspace} for the definition of \textit{bad component}.

\begin{prop}[\cite{liscasumslensspaces}] Let $G\subset\ZZ^n$ be a good subset with two components and $I(G)\le -2$. If $G$ has no bad components, then $I(G)=-2$ and $G$ has associated string of the form $(b_1,\ldots,b_k)\cup(c_1,\ldots,c_l)$, where  $(b_1,\ldots,b_k)$ and $(c_1,\ldots,c_l)$ are linear-dual strings.
Moreover, if $G=\{v_1,\ldots,v_k,v_{k+1},\ldots,v_{k+l}\}$, where $-v_i^2=b_i$ for $1\le i\le k$ and $-v_{k+j}^2=c_j$ for all $1\le j\le l$, then there exist integers $\alpha,\beta$ such that $E_\alpha=\{1,k+1\}$ and $E_\beta=\{k,k+l\}$.
\label{prop:goodsubsets}
\end{prop}


\subsection{Contractions, expansions, and base cases}\label{basecases} 
In this section, we discuss how to reduce the length of certain cyclic subsets via contractions.

\begin{definition} Suppose $S=\{v_1,\ldots,v_n\}$, $n\ge 3$, is a cyclic subset and suppose there exist integers $i,s$, and $t$ such that $E_i=\{s,\tilde{s},t\}$, where $\tilde{s}\in\{s\pm1\}$, $V_{\tilde{s}}\cap V_s=\{i\}$, $|v_u\cdot e_i|=1$ for all $u\in E_i$, and $a_t\ge 3$. After possibly cyclically reordering and re-indexing $S$, we may assume that $s\notin\{ 1,n\}$. Let $S'\subset \mathbb{Z}^{n-1}=\langle e_1,\ldots,e_{i-1},e_{i+1},\ldots, e_n\rangle$ be the subset defined by
\begin{center}	
$S'=(S\setminus\{v_s,v_{\tilde{s}},v_t\})\cup\{v_s+v_{\tilde{s}},\pi_{e_i}(v_t)\},$
\end{center}
	 
\noindent where $\pi_{e_i}(v_t)=v_t+(v_t\cdot e_i)e_i$. We say that $S'$ is obtained from $S$ by a \textit{contraction} and $S$ is obtained from $S'$ by an \textit{expansion}.\end{definition}

Since $s\notin\{1,n\}$ and $|v_{\tilde{s}}\cdot e_i|=|v_{s}\cdot e_i|=1$, we have that $v_{s-1}\cdot e_i=-v_{s}\cdot e_i$. Thus 
\begin{center} 
	$(v_s +v_{\tilde{s}})\cdot v_u=\begin{cases}
	1 & \text{if } \tilde{s}=s+1\text{ and } u\in\{s-1,s+2\} \\
	1 & \text{if } \tilde{s}=s-1\text{ and } u\in\{s-2,s+1\} \\
	0 & \text{otherwise}
	\end{cases}$
\end{center}

\noindent Moreover, $(\pi_{e_i}(v_t))^2=v_t^2+1\le -2$ and 
\begin{center} 
	$\pi_{e_i}(v_t)\cdot v_u=\begin{cases}
	1 & \text{if } u=t\pm1 \\
	0 & \text{otherwise}
	\end{cases}.$
\end{center}

\noindent Therefore, $S'$ is positive/negative cyclic subset if and only if $S$ is positive/negative cyclic. Moreover, $I(S')=I(S)$, $p_j(S')=p_j(S)$ for all $j\neq 3$, and $p_3(S')=p_3(S)-1$.

\begin{definition} Using the notation above, if $v_t\cdot v_s=1$ (so that $t=s\pm1$ if $\tilde{s}=s\mp1$) and $a_{\tilde{s}}=2$, then we say
\begin{itemize}
	\item $v_s$ is the \textit{center of $S$ relative to $e_i$},
	\item $S'$ is obtained by a \textit{contraction of $S$ centered at $v_s$}, and
	\item $S$ is obtained by a \textit{$-2$-expansion of $S$}.
\end{itemize}
\end{definition}

Note that a subset obtained by a contraction of $S$ centered at $v_s$ is unique. Indeed, if $E_i=\{s-1,s,s+1\}$, $a_{s-1}=2$, $a_{s+1}\ge3$, then $V_{s-1}\cap V_s=\{i\}$ and the only contraction centered at $v_s$ is $S\setminus\{v_s,v_{s-1},v_{s+1}\}\cup\{v_{s-1}+v_s,\pi_{e_i}(v_{s+1})\}$. Similarly, if $E_i=\{s-1,s,s+1\}$, $a_{s-1}=2$, $a_{s+1}\ge3$, then $V_{s-1}\cap V_s=\{i\}$ and the only contraction centered at $v_s$ is $S\setminus\{v_s,v_{s-1},v_{s+1}\}\cup\{v_{s}+v_{s+1},\pi_{e_i}(v_{s-1})\}$. Now let $S$ have associated string $(a_1,\ldots,a_n)$. Then, under the contraction centered at $v_s$, the associated string changes as follows: 
\begin{center}
 $(a_1,\ldots,a_{s-2},2,\bm{a_{s}},a_{s+1},a_{s+2},\ldots,a_n)\to(a_1,\ldots,a_{s-2},\bm{a_{s}},a_{s+1}-1,a_{s+2},\ldots, a_n)$ or \\
 $(a_1,\ldots,a_{s-2},a_{s-1},\bm{a_{s}},2,a_{s+2},\ldots,a_n)\to(a_1,\ldots,a_{s-2},a_{s-1}-1,\bm{a_{s}},a_{s+1},\ldots, a_n)$.
\end{center}

Notice that two strings $(b_1,\ldots,b_k)$ and $(c_l,\ldots,c_1)$ are reverse linear-dual if and only if $(b_1,\ldots,b_{k-1})$ and $(c_l-1,\ldots,c_1)$ or $(b_1,\ldots,b_{k}-1)$ and $(c_{l-1},\ldots,c_1)$ are reverse linear-dual. Thus the substrings on either side of $a_s$ in the associated string of $S$ are reverse linear-dual if and only if the substrings on either side of $a_s$ in the associated string of the contraction of $S$ centered at $v_s$ are reverse linear-dual.

More generally, let $S=\{v_1,\ldots,v_n\}$ and consider a sequence of contractions $S^0=S$, $S^1$, $S^2,\ldots, S^m$ such that $S^k$ is obtained from $S^{k-1}$ by performing a contraction centered at $v_{s}^{(k-1)}\in S^{k-1}$, where $v_s^{(0)}=v_s$. We call such a sequence of contractions \textit{the sequence of contractions centered at $v_s$} and call the reverse sequence of expansions \textit{a sequence of $-2$-expansions centered at $v^{(m)}_s$}. Notice that for all $1\le k\le m$, $v_s^{(k)}=v^{(k-1)}_s+v^{(k-1)}_{\tilde{s}}$, where $v^{(k-1)}_{\tilde{s}}$ is the unique vertex of $S^{k-1}$ adjacent to $v^{(k-1)}_s$ with square $-2$. We have proven the following.

\begin{lem} Let $S'$ be obtained from $S$ by a sequence of contractions centered at $v$ and let $v^2=-a$. Then $S$ has associated string of the form $(b_1,\ldots,b_k,a,c_l,\ldots,c_1)$, where $(b_1,\ldots,b_k)$ and $(c_l,\ldots,c_1)$ are reverse linear-dual, if and only if $S'$ has associated string of the form $(b'_1,\ldots,b'_{k'},a,c'_{l'},\ldots,c'_1)$, where $(b'_1,\ldots,b'_{k'})$ and $(c'_{l'},\ldots,c'_1)$ are reverse linear-dual.\label{centeredlem}\end{lem}
   
When $I(S)\le0$ and either $p_1(S)>0$ or $p_1(S)=p_2(S)=0$, we will be able to sequentially perform contractions until we arrive to certain base cases. In light of Example \ref{example}, we will restrict our attention to cyclic subsets containing at least one vector with square at most $-3$. We will now list all such cyclic subsets of length 2 and 3 with $I(S)\le 0$. It can be concretely checked case-by-case that the only such cyclic subsets are positive and (up to the action of $\Aut(\ZZ^2)$) are of the form:
\begin{itemize}
\item $\{2e_1,-e_1+e_2\}$, which has associated string $(4,2)\in\mathcal{S}_{2a}$.
\item $\{2e_1-e_3,e_3+e_2,-e_1-e_3\}$, which has associated string $(5,2,2)\in\mathcal{S}_{2a}$; and 
\item $\{e_1-e_2-e_3,e_3-e_1-e_2,e_2-e_3-e_1\}$, which has associated string $(3,3,3)\in\mathcal{S}_{2c}$.
\end{itemize}

Notice that the second and third vertices of the subset with associated string $(5,2,2)$ are both centers relative to $e_3$. If we perform a contraction centered at either vertex relative to $e_3$, we obtain the subset with associated string $(4,2)$. Note that when $n=3$, centers are not unique, but when $n\ge 4$, centers are necessarily unique. 

\begin{remark} We will usually denote cyclic subsets by $S$, standard subsets by $T$, and good subsets by $G$. Moreover, $S'$ will be reserved for contractions of $S$.\end{remark}


\subsection{Preliminary Lemmas}
The following lemmas will be important in future sections. The first follows from the proof of Lemma 2.5 in \cite{liscalensspace}

\begin{lem}[Lemma 2.5 in \cite{liscalensspace}] If $S=\{v_1,\ldots,v_n\}\subset\ZZ^n=\langle e_1,\ldots,e_n\rangle$ is any subset, then $$2p_1(S)+p_2(S)+I(S)\ge\displaystyle\sum_{j=4}^n(j-3)p_j(S),$$ with equality if and only if $|v_\alpha\cdot e_\beta|\le1$ for all $1\le\alpha,\beta\le n$.\label{keylem1}\end{lem}

\begin{lem} Let $S$ be cyclic such that $p_2(S)>0$ and $|v_\alpha\cdot e_\beta|\le1$ for all $1\le \alpha,\beta\le n$. Then $\sum_i p_{2i}(S)\equiv -I(S)\mod 4$. 
	\label{lem:p2+p4}
\end{lem}

\begin{proof}
	First notice that since $I(S)=\sum_{i=1}^n (a_i-3)$, we have that $\sum_{i=1}^na_i=3n +I(S)$. Now
	\[-\Big(\sum_{i=1}^n v_i\Big)^2=\sum_{i=1}^n a_i-\sum_{i=1}^{n-1}2v_i\cdot v_{i+1}-2v_1\cdot v_n=\begin{cases} 
	n +I(S) & \text{ if } S \text{ is positive} \\
	n+4+I(S) & \text{ if } S \text{ is negative}.
	\end{cases} 
	\]
	On the other hand, set $\displaystyle\sum_{i=1}^nv_i=\sum_{i=1}^n\lambda_ie_i$ and let $k_{\alpha}=|\{i\,|\,|\lambda_i|=2\alpha+1\}|$ and let $x_\beta=|\{i\,|\,|\lambda_i|=2\beta\}|$. Finally, let $m\in\ZZ$ be the largest integer such that $k_m\neq0$ and $k_t=0$ for all $t>m$ and let $y\in\ZZ$ be the largest integer such that $x_y\neq0$ and $x_t=0$ for all $t>y$. Since $|v_\alpha\cdot e_\beta|\le1$ for all $\alpha,\beta$, we have that $\sum_i p_{2i}(S)=x_0+\ldots+x_y$. Hence
	\begin{equation*}
	\begin{split}
	-\Big(\displaystyle\sum_{i=1}^n v_i\Big)^2&=-\displaystyle\sum_{i=1}^n\lambda_i^2=\Big(n-\Big(\sum_{\alpha=1}^mk_\alpha\Big)-\Big(\sum_{\beta=0}^yx_\beta\Big)\Big)+\sum_{\alpha=1}^m(2\alpha+1)^2k_\alpha +\sum_{\beta=0}^y(2\beta)^2x_\beta\\
	&=
	n+\sum_{\alpha=1}^m(4\alpha^2+4\alpha)k_\alpha +\sum_{\beta=0}^y(4\beta^2-1)x_\beta\\
	&=
	n+\sum_{\alpha=1}^m(4\alpha^2+4\alpha)k_\alpha +\sum_{\beta=0}^y(4\beta^2)x_\beta-\Big(\sum_i p_{2i}(S)\Big)
	\end{split}
	\end{equation*}
	
\noindent Thus we have that 
	\[
	\sum_{\alpha=1}^m(4\alpha^2+4\alpha)k_\alpha +\sum_{\beta=1}^y(4\beta^2)x_\beta=\begin{cases} 
	\sum_i p_{2i}(S)+I(S) & \text{ if } S \text{ is positive} \\
	\sum_i p_{2i}(S)+4+I(S) & \text{ if } S \text{ is negative}
	\end{cases}.
	\]
	
\noindent It follows that $\sum_i p_{2i}(S)\equiv -I(S)\mod 4$.
\end{proof}

\begin{lem} If $G=\{v_1,\ldots,v_n\}\subset\ZZ^n$ is a good subset with $I(G)=0$, $p_3(G)=n$, and $n$ components, then up to the action of $\Aut\ZZ^n$, negating vertices, and permuting vertices:
	\begin{itemize}
		\item $G=\{e_1-e_2+e_3-e_4,e_1+e_2,-e_1+e_2+e_3-e_4,e_3+e_4\}$ and has associated string $(4,2,4,2)$; or
		\item $G=\{e_1-e_2-e_3,e_1+e_2-e_4,e_2-e_3+e_4,e_1+e_3+e_4\}$ and has associated string $(3,3,3,3)$.
	\end{itemize}
	\label{lem:4comps}
\end{lem}

\begin{proof}
	First notice that by Lemma \ref{keylem1}, $|v_\alpha\cdot e_\beta|\le 1$ for all $\alpha, \beta$.	
	Let $i,s,t,u$ be integers such that $E_i=\{s,t,u\}$. Since every vertex of $G$ is isolated, up to negating vertices we may assume that $v_s\cdot e_i=v_t\cdot e_i=v_u\cdot e_i=-1$. 
	
	First suppose $a_s=2$ and let $v_s=e_i+e_j$. Then since $v_s\cdot v_t=v_s\cdot v_u=0$, we have that $v_t=e_i-e_j+a$ and $v_u=e_i-e_j+b$. Since $v_t\cdot v_u=0$, there are integers $k,l\in V_t\cap V_u$ such that $v_t=e_i-e_j+e_k-e_l+a'$ and $v_u=e_i-e_j-e_k+e_l+b'$. If $(a')^2\neq0$, then let $R=\{v_1',\ldots,v_{s-1}',v_{s+1}',\ldots,v_n'\}\subset\ZZ^{n-2}=\langle e_1,\ldots,e_n\rangle/\langle e_i,e_j\rangle$, where $v_t'=\pi_{e_j}(\pi_{e_i}(v_t))$, $v_u'=\pi_{e_j}(\pi_{e_i}(v_u))$, and $v_x':=v_x$ for all $x\not\in\{t,u\}$. Then $(v_t')^2\le -3$, $v_t'\cdot v_u'=2$, and $v_t'\cdot v_x=v_u'\cdot v_x'=0$ for all $x\not\in \{t,u\}$. Consequently, $R$ is the union of a positive cyclic subset $\{v_t',v_u'\}$ and a good subset $R\setminus\{v_t',v_u'\}$. Thus by Remark \ref{LIremark}, $R$ is a linearly independent set of $n-1$ vectors in $\ZZ^{n-2}$, which is impossible. Thus $(a')^2=0$ and similarly, $(b')^2=0$; hence $v_t=e_i-e_j+e_k-e_l$ and $v_u=e_i-e_j-e_k+e_l$. Now since $|E_k|=|E_l|=3$, there exists an integer $z$ such that $k,l\in V_{z}$ and since $v_z\cdot v_t=0$, we may assume that $v_z= e_k+e_l+c$. By a similar argument as above, we have that $c^2=0$ and so $v_z=e_k+e_l$. Since $G$ is irreducible, it follows that $n=4$ and so $G$ has associated string of the form $(4,2,4,2)$. Setting $i=3$, $j=4$, $k=1$, $l=2$, we have the subset listed in the statement of the lemma.
	
	Next suppose $a_s,a_t,a_u\ge3$. Assume $a_s>3$. Let $R=\{v_1',\ldots,v_{s-1}',v_{s+1}',\ldots,v_n'\}\subset\ZZ^{n-1}=\langle e_1,\ldots,e_n\rangle/\langle e_i\rangle$, where $v_s'=\pi_{e_i}(v_s)$, $v_t'=\pi_{e_i}(v_t)$, $v_u'=\pi_{e_i}(v_u)$, and $v_x':=v_x$ for all $x\not\in\{s,t,u\}$. Then $(v_s')^2<-2$ and $v_s'\cdot v_t'=v_s'\cdot v_u'=v_t'\cdot v_u'=1$; hence $\{v_s',v_t',v_u'\}$ is a positive cyclic subset. Moreover, $v_s'\cdot v_x'=v_t'\cdot v_x'=v_u'\cdot v_x'=0$ for all $x\not\in\{s,t,u\}$. Thus $R$ is the union of a positive cyclic subset and a good subset and so by Remark \ref{LIremark}, $R$ is a linearly independent set of $n-1$ vectors in $\ZZ^{n-2}$, which is impossible. Thus $a_s=3$; similarly, $a_t=a_u=3$. Without loss of generality, we have that $v_s=e_i-e_j-e_k$, $v_t=e_i+e_j-e_l$, and $v_u=e_i+e_k+e_l$, for some integers $j,k,l$. Since $|E_j|=3$, there exists an integer $z$ such that $j\in V_z$. Since $v_z\cdot v_s=v_z\cdot v_t=v_z\cdot v_u=0$, we have that $v_z=e_j-e_k+e_l+a$. If $a^2\neq 0$, then we can define a subset $R$ as above and arrive at a similar contradiction. Thus $v_z=e_j-e_k-e_l$. Since $G$ is irreducible, it follows that $n=4$ and so $G$ has associated string of the form $(3,3,3,3)$. Setting $i=1$, $j=2$, $k=3$, $l=4$, we have the subset listed in the statement of the lemma.
\end{proof}


\section{Lattice Analysis Case I: $p_1(S)>0$}\label{p1>0}

Throughout this section, we will assume that $S=\{v_1,\ldots,v_n\}$ is a cyclic subset with $I(S)\le0$ and $p_1(S)>0$. Thus there exist integers $i$ and $s$ such that $E_i=\{s\}$. Lemmas \ref{lem1.3}$-$\ref{lem1.2} will ensure that we can contract such subsets.

\begin{lem} Let $S$ be a cyclic subset of length 4 such that $I(S)\le0$ and $E_i=\{s\}$ for some integers $i$ and $s$. If $a_{s+1}\ge 3$ or $a_{s-1}\ge3$, then $S$ is positive and has associated string of the form $(6,2,2,2)$ or $(5,2,2,3)$. If $a_{s\pm1}=2$, then $S$ is either: negative and has associated string of the form $(2,2,2,2)$ or  $(2,2,2,5)$; or positive and has associated string of the form $(2,2,2,3)$ or $(2,2,2,6)$.
\label{lem1.3}\end{lem}

\begin{proof} If $|V_s|=1$, then since $E_i=\{s\}$, we obtain $v_s\cdot v_{s+1}=0$, which is a contradiction. Thus $|V_s|\ge2$.

Suppose $a_{s-1}\ge 3$. If $|V_s|\ge 3$, then let $R\subset\ZZ^3$ be the subset obtained by replacing $v_s$ by $v_s+(v_s\cdot e_i)e_i$.  Then $R$ is a cyclic subset and by Remark \ref{LIremark}, $R$ is made of four linearly independent vectors in $\ZZ^3$, which is not possible. Thus $|V_s|=2$. Let $V_s=\{i,j\}$. Then $E_j=\{s-1,s,s+1\}$, since otherwise, we would necessarily have that $|E_i|>1$. Moreover, since $V_{s-1}\cap V_s=V_{s+1}\cap V_s=\{j\}$, we necessarily have that $|v_{s-1}\cdot e_j|=|v_{s}\cdot e_j|=|v_{s+1}\cdot e_j|=1$. If $S$ is positive cyclic, then it is clear that $v_{s-1}\cdot e_j=v_{s+1}\cdot e_j=-v_{s}\cdot e_j$. If $S$ is negative cyclic, then by possibly moving the negative intersection (c.f. Remark \ref{moveneg}), we may assume that $v_{s-1}\cdot e_j=v_{s+1}\cdot e_j=-v_{s}\cdot e_j$. Thus we may perform a contraction of $S$ centered at $v_s$ relative to $e_j$ to obtain a length 3 cyclic subset $S'$ with $I(S')=I(S)\le0$ and $p_1(S')>0$. By considering the base cases in Section \ref{basecases}, it is clear that $S'=\{2e_1-e_3,e_3+e_2,-e_1-e_3\}$ (up to the action of $\Aut(\ZZ^3)$), which has associated string $(5,2,2)$. Thus $i=2$, $j=4$, and either $S=\{2e_1-e_3-e_4,e_2+e_4,-e_4+e_3,-e_1-e_3\}$ or $S=\{2e_1-e_3,e_3-e_4,e_4+e_2,-e_4-e_1-e_3\}$. Therefore, $S$ is positive and has associated string $(6,2,2,2)$ or $(5,2,2,3)$.
	
Now suppose $a_{s-1}=a_{s+1}=2$. Without loss of generality, assume $s=j=4$. Let $T=\{v_1,v_2,v_3\}\subset\ZZ^3=\langle e_1,e_2,e_3\rangle$ be the length 3 standard subset obtained by removing $v_s$ from $S$. Then $T$ has associated string of the form $(2,a_2,2)$. Since $I(S)\le0$, we must have $a_2\le 6$. It is easy to see that $a_2\neq6$, since, otherwise, $v_2=2e_1-e_2-e_3$ (up to the action of $\Aut(\ZZ^3)$), implying that $v_1\cdot v_2\neq \pm1$, which is a contradiction. If $a_2=5$, then $T$ is of the form $\{e_1-e_2,e_2+2e_3,-e_2-e_1\}$ and therefore, $S$ must be of the form $\{e_1-e_2,e_2+2e_3,-e_2-e_1, e_1+e_4\}$  (up to the action of $\Aut(\ZZ^3)$). Thus $S$ is negative with associated string $(2,5,2,2)$ (equivalently, $(2,2,2,5)$). If $a_2\le 4$, then $I(T)<0$. By Lemma \ref{liscasummary2}, the only such length 3 standard subset has associated string $(2,2,2)$. Moreover, $T$ is of the form $T=\{e_1-e_2,e_2-e_3,-e_2+e_1\}$ (c.f. Lemma 2.4 in \cite{liscalensspace}). Since $v_3\cdot v_4=\pm1$, either $1\in V^S_4$, $2\in V^S_4$, or both.  If $1,2\in V^S_4$, then since $v_2\cdot v_4=0$, we must have $3\in V^S_4$; thus $|V^S_4|=4$. Moreover, since $v_1\cdot v_4=\pm1$, we must have that $v_4\cdot e_1=v_4\cdot e_2\pm1$, implying that $a_4\ge 7$, which is not possible. Thus either $1\in V^S_4$ or $2\in V^S_4$, but not both. If $1\in V^S_4$, then $S$ is negative and of the form $\{e_1-e_2,e_2-e_3,-e_2-e_1,e_1+e_4\}$ or $\{e_1-e_2,e_2-e_3,-e_2-e_1,e_1+2e_4\}$, which have associated strings $(2,2,2,2)$ and $(2,2,2,5)$ (note that we found the latter subset above). If $2\in V^S_4$, then $3\in V^S_4$ and $S$ is positive and of the form $\{e_1-e_2,e_2-e_3,-e_2-e_1,e_2+e_3+e_4\}$ or $\{e_1-e_2,e_2-e_3,-e_2-e_1,e_2+e_3+2e_4\}$, which have associated strings $(2,2,2,3)$ and $(2,2,2,6)$. 
\end{proof}

\begin{lem} Let $S$ be a cyclic subset of length at least 5 such that $E_i=\{s\}$ for some $i$ and $s$. Then $|V_s|=2$. Moreover, if $V_s=\{i,j\}$, then $E_j=\{s-1,s,s+1\}$ and $v_{s-1}\cdot e_j=v_{s+1}\cdot e_j=-v_{s}\cdot e_j=\pm1$.\label{lem1.1}\end{lem}

\begin{proof} First note that if $|V_s|=1$, then since $E_i=\{s\}$, we obtain $v_s\cdot v_{s+1}=0$, which is a contradiction. Now suppose $|V_s|\ge3$. Then by replacing $v_s$ with $v_s'=v_s+(v_s\cdot e_i)e_i$ and relabeling $v_u'=v_u$ for all $u\neq s$, we obtain a subset $R=\{v'_1,...,v'_{s-1},v_{s}',v'_{s+1},...,v'_n\}\subset \mathbb{Z}^{n-1}=\langle e_1,\ldots,e_{i-1},e_{i+1},\ldots,e_n\rangle$. Let $(a_1',\ldots,a_n')$ be the string associated to $R$, where $-a_s':=v_s'\cdot v_s'\le -2$, and $a_j'=a_j$ for all $j\neq i$. If $S$ is negative cyclic, then so is $R$ and thus by Remark \ref{LIremark}, $R$ is made of $n$ linearly independent vectors in $\mathbb{Z}^{n-1}$, which is not possible. If $S$ is  positive cyclic and if either $a_s'\ge 3$ or $a_i\ge 3$ for some $i\neq s$, then $R$ is also positive cyclic, and we obtain a similar contradiction. Now suppose $S$ is positive cyclic, $a_s'=2$ and $a_t'=a_t=2$ for all $t\neq s$. Let $T$ be the subset obtained by removing $v_s$ from $S$. Then $T$ has associated string $(2^{[n-1]})$ and so $I(T)=-(n-1)\le -4$. If $|E_k^{S}|\ge2$ for all $k\in V_s^{S}$, where $k\neq i$, then $T$ is a standard subset of $\ZZ^{n-1}$ with $I(T)\le -4$, which contradicts Proposition \ref{liscasummary1}. If $|E_k^{S}|=1$ for some $k\in V_s^{S}$ such that $k\neq i$, then by Remark \ref{LIremark}, $T$ consists of $n-1$ linearly independent vectors in $\ZZ^{m}$, where $m<n-1$, which is not possible. Thus $|V_s|=2$.  Let $V^S_s=\{i,j\}$. Then, as in the proof of Lemma \ref{lem1.3}, $E_j=\{s-1,s,s+1\}$ and $v_{s-1}\cdot e_j=v_{s+1}\cdot e_j=-v_{s}\cdot e_j=\pm1$.
\end{proof}

\begin{lem} Let $S$ be a cyclic subset of length at least 5 such that $I(S)\le0$ and $E_i=\{s\}$ for some $i$ and $s$. Then either $a_{s-1}\ge 3$ or $a_{s+1}\ge3$. Moreover, if $a_{s\pm1}\ge 3$, then $S$ is positive with associated string $(2,3,2,3,2)$ or $(2,3,5,3,2)$.\label{lem1.2}\end{lem}

\begin{proof} By Lemma \ref{lem1.1}, we have $V_s=\{i,j\}$ and $E_j=\{s-1,s,s+1\}$. Assume that $a_{s-1}=a_{s+1}=2$. Then $V_{s-1}=\{j,k\}$ for some $k$, $V_{s+1}=\{j,k'\}$ for some $k'$, and $|v_{s\pm1}\cdot e_j|=|v_{s-1}\cdot e_k|=|v_{s+1}\cdot e_{k'}|=1$. Since $v_{s-1}\cdot v_{s+1}=0$, we must have $k=k'$. Since $|v_{s-2}\cdot v_{s-1}|=1$ and $j\notin V_{s-2}$, we must have $k\in V_{s-2}$. But then $v_{s-2}\cdot v_{s+1}\neq0$, which is a contradiction.

Now suppose $a_{s-1},a_{s+1}\ge 3$ and let $R$ be the subset obtained by removing $v_s$ and replacing $v_{s\pm1}$ with $v_{s\pm 1}'=v_{s\pm1}+(v_{s\pm1}\cdot e_j)\cdot e_j$. Note that $v_{s-1}'\cdot v_{s+1}'=\pm1$. As in the proof of Lemma \ref{lem1.1}, $R$ is either cyclic or $S$ is positive cyclic and $R$ has associated string of the form $(2^{[n-1]})$. In the former case, by Remark \ref{LIremark}, $R\subset\ZZ^{n-2}$ contains $n-1$ linearly independent vectors, which is not possible. In the latter case, let $T\subset\ZZ^{n-1}$ be the standard subset obtained from $S$ by only removing $v_s$. Then $T$ has associated string $(3,2,\ldots,2,3)$. By Proposition \ref{liscasummary2}, the only such standard subset is $\{e_4+e_3-e_2,e_2+e_1,-e_2-e_4,e_2+e_3-e_1\}$ (up to the action of $\Aut(\ZZ^4)$), which has associated string $(3,2,2,3)$. Thus $j=3$, $|v_s\cdot e_3|=1$. Since $I(S)\le0$, $S$ is of the form $\{-e_2-e_4,e_2+e_3-e_1,e_5-e_3,e_4+e_3-e_2,e_2+e_1 \}$ or $\{-e_2-e_4,e_2+e_3-e_1,2e_5-e_3,e_4+e_3-e_2,e_2+e_1 \}$, which are positive and have associated strings $(2,3,2,3,2)$ and $(2,3,5,3,2)$, respectively. 
\end{proof}

Let $S=\{v_1,\ldots,v_n\}$ be a cyclic subset such that $n\ge6$, $I(S)\le0$ and $E^S_i=\{s\}$ for some integers $i$ and $s$. By Lemma \ref{lem1.1}, we may assume that $V^S_s=\{i,j\}$ and $E^S_j=\{s-1,s,s+1\}$ for some integer $j$. Thus $v_s$ is the center vertex of $S$ relative to $e_j$. By Lemma \ref{lem1.2}, we may further assume that $a_{s+1}\ge3$ and $a_{s-1}=2$ and so $V^S_{s-1}=\{j,j_1\}$ for some integer $j_1$. Let $S'=\{v_1',\ldots,v_{s-2}',v'_s,v_{s+1}',\ldots,v'_n\}$ be the contraction of $S$ centered at $v_s$, where $v_x'=v_x$ for all $x\notin\{ s-1,s,s+1\}$, $v_{s}'=v_{s-1}+v_s$, and $v_{s+1}'=\pi_{e_j}(v_t)$. Since $V_{s}^{S'}=\{i,j_1\}$ and $E_{j_1}^{S'}=\{s-2,s,s+1\}$, $v_{s}'$ is the center vertex of $S'$ relative to $e_{j_1}$ and by Lemma \ref{lem1.2}, either $(v_{s-2}')^2\le-3$ or $(v_{s+1}')^2\le-3$. If $(v_{s-2}')^2\le-3$ and $(v_{s+1}')^2\le-3$, then by Lemma \ref{lem1.2}, $S'$ is positive and has associated string of the form $(2,3,2,3,2)$ or $(2,3,5,3,2)$. If $(v_{s-2}')^2=-2$ or $(v_{s+1}')^2=-2$, then we can perform the contraction centered at $v_s'$ relative to $e_{j_1}$, as above. Continuing in this way, we have a sequence of contractions centered at $v_s$, which ends in a subset $\hat{S}$ either of length $4$ or of length 5 with associated string $(2,3,2,3,2)$ or $(2,3,5,3,2)$. Let $\hat{v}_s$ denote the resulting center vertex of $\hat{S}$. Then $V_s^{\hat{S}}=\{i,k\}$ for some integer $k$ and $|E_k^{\hat{S}}|=3$. 

Suppose that $\hat{S}$ has length 4. By considering the length 4 cyclic subsets in the proof of Lemma \ref{lem1.3}, it is clear that $\hat{S}$ is either negative and of the form:
\begin{itemize}
	\item $\{e_1-e_2,e_2-e_3,-e_2-e_1,e_1+e_4\}$ with associated string $(2,\textbf{2},2,\textbf{2})$; or 
	\item $\{e_1-e_2,e_2-e_3,-e_2-e_1,e_1+2e_4\}$ with associated string $(2,\textbf{2},2,\textbf{5})$;
\end{itemize}
or positive and of the form:
\begin{itemize}
	\item $S=\{2e_1-e_3-e_4,e_2+e_4,-e_4+e_3,-e_1-e_3\}$ with associated string $(6,\textbf{2},2,2)$; or
	\item $S=\{2e_1-e_3,e_3-e_4,e_4+e_2,-e_4-e_1-e_3\}$ with associated string $(5,2,\textbf{2},3)$.
\end{itemize}

\noindent Each bold number in the above strings corresponds to a vertex $\hat{v}_m$ satisfying $E^{\hat{S}}_{\alpha}=\{m\}$ for some integers $\alpha, m$. In particular, one of the bold numbers in each of the above strings corresponds to $\hat{v}_s$. In the first two cases, notice that the substrings in between the bold numbers (i.e. $(2)$ and $(2)$) are reverse linear-dual. Thus, by Lemma \ref{centeredlem}, $S$ has  associated string of the form $(b_1,\ldots,b_k,2,c_l,\ldots,c_1,2)$ or $(b_1,\ldots,b_k,2,c_l,\ldots,c_1,5)$, where $(b_1,\ldots,b_k)$ and $(c_l,\ldots,c_1)$ are reverse linear-dual. Similarly, the third and fourth strings are of the form $(b_1+3,b_2,\ldots,b_k,2,c_l,\ldots,c_1)$, where $(b_1,\ldots,b_k)$ and $(c_l,\ldots,c_1)$ are reverse linear-dual and so $S$ has associated string of the same form. Note that the strings $(5,2,2)$ and $(4,2)$ also fall under this family (recall that the linear-dual of $(1)$ is the empty string).

Now suppose $\hat{S}$ has length 5. Then by the proof of Lemma \ref{lem1.2}, $\hat{S}$ is positive and of the form 
\begin{itemize}
	\item $\{-e_2-e_4,e_2+e_3-e_1,e_5-e_3,e_4+e_3-e_2,e_2+e_1 \}$ with associated string $(2,3,\textbf{2},3,2)$ or 
	\item $\{-e_2-e_4,e_2+e_3-e_1,2e_5-e_3,e_4+e_3-e_2,e_2+e_1 \}$ with associated string $(2,3,\textbf{5},3,2)$.
\end{itemize}

\noindent As above, the bold numbers in these two strings correspond to the vertex $\hat{v}_s$. Notice that after performing a $-2$-expansion centered at $\hat{v}_s$, the first and last entries in each string remain unchanged. Moreover, the substrings adjacent to the bold numbers are $(3)$ and $(3)$; notice $(3-1)=(2)$ and $(3-1)=(2)$ are reverse linear-dual strings. Thus, as above, $S$ has associated string of the form $(2,b_1+1,b_2,\ldots,b_k,2,c_l,\ldots,c_2,c_1+1,2)$ or $(2,b_1+1,b_2,\ldots,b_k,5,c_l, \ldots,c_2,c_1+1,2)$, where $(b_1,\ldots,b_k)$ and $(c_l,\ldots,c_1)$ are reverse linear-dual strings.

\begin{remark} Consider the length 5 subsets above. We can perform contractions to obtain the cyclic subsets of Lemma \ref{lem1.3} with associated strings $(2,2,2,3)$ and $(2,2,2,6)$. However, these do not fall under the general formulas listed above. Moreover, the string $(2,2,2,6)$ is also the associated string of a different subset, as seen in Lemma \ref{lem1.3}. This string already appeared in first set of cases we considered and so we will not count this string again. \end{remark}

\noindent Combining all of these cases, we have proven the following.

\begin{prop} Let $S$ be a cyclic subset with $I(S)\le0$ and $p_1(S)>0$. Then $S$ is either negative with associated string in $\mathcal{S}_{1a}\cup\mathcal{S}_{1b}$ or positive with associated string in $\mathcal{S}_{2a}\cup\mathcal{S}_{2b}\cup\mathcal{S}_{2e}$. 
\label{prop1}
\end{prop}


\section{Lattice Analysis Case II: $p_1(S)=0$}\label{p1=0}

In this section, we will assume that $S=\{v_1,\ldots,v_n\}$ is  cyclic with $I(S)\le0$ and  $p_1(S)=0$. By Lemma \ref{keylem1}, $p_2(S)\ge\sum_{j=4}^n(j-3)p_j(S)$. If $p_2(S)=0$, then the inequality is necessarily an equality and so $p_j(S)=0$ for all $4\le j\le n$. Thus, in this case, $I(S)=0$ and $p_3(S)=n$. Therefore, we have two cases to consider: $p_2(S)=0$ and $p_2(S)>0$.

\subsection{Case IIa}\label{p3=n}
Let $S$ be cyclic and $p_1(S)=p_2(S)=0$. Then as shown above, $I(S)=0$ and $p_3(S)=n$. The next two lemmas provide some general properties of $S$.

\begin{lem} If $S$ is cyclic and $p_1(S)=p_2(S)=0$, then $|v_\alpha\cdot e_\beta|\le 1$ for all $1\le\alpha,\beta\le n$. \label{lem2a.1} \end{lem}
\begin{proof}
Let $v_i=\sum_{j=1}^nm_{ij}e_j$ for each $i$, where $m_{ij}=v_i\cdot e_j$. Then since $I(S)=0$, we have that $3n=-\sum_{i=1}^nv_i^2=\sum_{i,j} m_{ij}^2\ge \sum_{i,j}|m_{ij}|\ge 3n$. Thus $m^2_{ij}=|m_{ij}|$ for all $i,j$ and so $|v_i\cdot e_j|=|m_{ij}|\le1$ for all $i,j$. 
\end{proof}

\begin{lem} If $S$ is cyclic and $p_1(S)=p_2(S)=0$, then $S$ is positive cyclic.\label{lem2a.2}\end{lem}

\begin{proof} Again, let $v_i=\sum_{j=1}^nm_{ij}e_j$. By Lemma \ref{lem2a.1}, $|m_{ij}|\le 1$ for all $i,j$. Let $\sum_{i=1}^nv_i=\sum_{i=1}^n\lambda_ie_i$. Then since $p_3(S)=n$,  $\lambda_i\in\{\pm1, \pm 3\}$ for all $i$. Now, if $S$ is negative, then $-3n= \sum_{i=1}^nv_i^2=(\sum_{i=1}^nv_i)^2-2\sum_{i<j}v_i\cdot v_j = (-\sum_{i=1}^n \lambda_i^2)-2(n-2)$ or $\sum_{i=1}^n\lambda_i^2=n+4$. Thus there must exist $j$ such that $\lambda_j=\pm 3$. But then $n-1\le\sum_{i\neq j}\lambda_i^2=n-5$, which is impossible. Thus $S$ must be positive.\end{proof}

If $p_3(S)=n$, then it is clear that $n\ge3$. If $n=3$, then $S$ is the subset with associated string $(3,3,3)\in\mathcal{S}_{2b}\cap\mathcal{S}_{2c}$ found in Section \ref{basecases}. From now on, we will assume that $n\ge4$.

\begin{lem} Let $S$ be cyclic with $p_1(S)=p_2(S)=I(S)=0$. Suppose there exist integers $i$ and $s$ such that $E_i=\{s-1,s,s+1\}$. Then $S$ is positive and has associated string in $\mathcal{S}_{2b}$.
\label{lem2a.6}
\end{lem}

\begin{proof}
	By Lemma \ref{lem2a.2}, we know that $S$ is necessarily positive.
	Now since $E_i=\{s-1,s,s+1\}$, we necessarily have that $a_s\ge3$; otherwise, if $a_s=2$ and $V_s=\{i,i'\}$, then $|E_{i'}|=1$, which is a contradiction.
	We further claim that $a_{s-1}\ge3$ or $a_{s+1}\ge3$. Suppose otherwise: $a_{s-1}=a_{s+1}=2$. Then $V_{s-1}=V_{s+1}=\{i,j\}$ for some integer $j$ and since $|E_i|=3$, we necessarily have that $j\in V_{s-2}\cap V_{s+2}$. Since $|E_j|=3$, we necessarily have that $n=4$. But then there exists an integer $k\in V_s$ such that either $|E_k|=1$ or $|E_k|=2$, which is a contradiction.
	Without loss of generality, assume that $a_{s-1}\ge3$. 
	
	First assume that $v_{s-1}\cdot e_i=v_s\cdot e_i$ (or similarly $v_{s+1}\cdot e_i=v_s\cdot e_i$). Let $x\ge0$ be the smallest integer such that $a_{s+x+1}\ge3$.
	Since $a_{s+1}=\cdots=a_{s+x}=2$, we have that $V_{s+\alpha}=\{i_{\alpha-1},i_\alpha\}$ for all $1\le\alpha\le x$, where $i_0:=i$ and $\{i_0,\ldots,i_{x}\}$ contains $x+1$ distinct integers. Moreover, $E_{i_\alpha}=\{s-1,s+\alpha,s+\alpha+1\}$ for all $1\le\alpha\le x$. Since $v_{s-1}\cdot e_i=v_s\cdot e_i$, by Lemmas \ref{lem2a.1} and \ref{lem2a.2}, there exist integers $m,k\in V_{s-1}\cap V_s$ such that $v_{s-1}\cdot e_m=-v_{s}\cdot e_m$ and $v_{s-1}\cdot e_k=-v_{s}\cdot e_k$. Thus $a_{s-1}\ge x+3$. Let $R=\{v_1',\ldots,v_{s-1}',v_{s+x+1}',\ldots,v_n\}\subset\ZZ^{n-x-1}=\langle e_1,\ldots,e_n\rangle/\langle e_{i_0},\ldots,e_{i_x}\rangle$, where $v_{s-1}'=\pi_{e_{i_0}}(\pi_{e_{i_1}}(\cdots(\pi_{e_{i_x}}(v_{s-1}))\cdots))$, $v_{s+x+1}'=\pi_{e_{i_x}}(v_{s+x+1})$, and $v_y'=v_y$ for all $y\not\in\{s-1,\ldots,s+x+1\}$. Then $R$ is negative cyclic with $I(R)=1-a_s\le -2$. By Proposition \ref{prop:p2>0} in Section \ref{p2>0}, $R$ must have associated string in $\mathcal{S}_{1c}\cup \mathcal{S}_{1d}\cup \mathcal{S}_{1e}\cup\mathcal{O}\cup\{(2^{[n]})\,|\,n\ge2\}$ and hence either $I(R)=-(n-x-1)$ or $I(R)=-2$. In the former case, we necessarily have that $a_{s-1}=3+x$, $a_s=n+x$, and $a_{s+x+1}=3$; hence $S$ has associated string of the form $(3+x,n+x,2^{[x]},3,2^{[n-x-3]})\in\mathcal{S}_{2b}$. In the latter case, $a_s=3$ and so $V^S_s\cap V^S_{s-1}=\{i,m,k\}$. Since $v_s^2=-3$, it follows that $V^{S}_m=V^{{S}}_k=\{s-1,s,z\}$ for some integer $z\not\in\{s-1,s\}$. It is easy to see that ${v}_{s-1}^2\le-(4+x)$ and $\tilde{v}_{z}^2\le-3$. Let $T=({S}\setminus \{{v}_{z},{v}_{s},{v}_{s-1} \})\cup \{\pi_{e_k}({v}_{s}), \pi_{e_m}(\pi_{e_k}({v}_{s-1}))\}$. Then $T$ is standard with $I(T)\le -3$ and $E_m^T=\{s\}$. By Proposition \ref{liscasummary1}, $I(T)=-3$ and so ${v}_z^2=-3$; by Proposition \ref{liscasummary2}(1), $T$ and has associated string of the form $(b_1,\ldots,b_k,2,c_l,\ldots,c_1)$, where $(b_1,\ldots,b_k)$ and $(c_1,\ldots,c_l)$ are linear-dual strings, and the middle vertex with square $-2$ is $\pi_{e_k}({v}_{s})$. Thus $S$ has associated string $(3,b_1,\ldots,b_k+2,3,c_l,\ldots,c_1)$. Since $(\beta_1,\ldots,\beta_\kappa)=(b_1,\ldots,b_k+1)$ has linear-dual $(\gamma_1,\ldots,\gamma_\lambda)=(2,c_1,\ldots,c_l)$ (c.f. Lemma \ref{lem:dualconcat}), we have  $(3,b_1,\ldots,b_k+2,3,c_l,\ldots,c_1)=(3,\beta_1,\ldots,\beta_{\kappa-1},\beta_\kappa+1,\gamma_\lambda+1,\gamma_{l-1},\ldots,\gamma_1)\in\mathcal{S}_{2b}$.
	
	Now assume that $v_{s-1}\cdot e_i=-v_s\cdot e_i=v_{s+1}\cdot e_i$.
	Suppose $a_{s+1}=2$ and set $V_{s+1}=\{i,j\}$. Note that $E_j=\{s-1,s+1,s+2\}$ and $V_s\cap V_{s+1}=\{i\}$. Thus $v_s$ is the center of $S$ relative to $e_i$. Perform the contraction of $S$ centered at $v_{s}$ to obtain the positive cyclic subset $S'=\{v_1',\ldots,v_{s}',v_{s+2}',\ldots,v_n'\}$, where $v_x'=v_x$ for all $x\notin\{ s-1,s,s+1\}$, $v_{s}'=v_{s}+v_{s+1}$, and $v_{s-1}'=\pi_{e_i}(v_{s-1})$. Then $I(S')=0$ and $p_3(S')=n-1$. Now the vertices $v_{s-1}'$, $v_{s}'$, and $v_{s+2}'$ are adjacent in $S'$, $E_j^{S'}=\{s-1,s,s+2\}$, and $(v_{s}')^2=v_{s}^2\le-3$. Thus $v_{s}'$ is the center of $S'$ relative to $e_j$. Moreover, $v_{s-2}'\cdot e_j=-v_{s}'\cdot e_j=v_{s+2}'\cdot e_j$. If $(v_{s-2}')^2=-2$ or $(v_{s+1}')^2=-2$, then we can contract $S'$ centered at $v_s'$. Continuing  in this way, we have a sequence of contractions centered at $v_{s}$ which terminates in a positive subset $\tilde{S}$ such that the resulting center vertex $\tilde{v}_{s}$ has adjacent vertices whose squares are both at most $-3$. Re-index $\tilde{S}$ chronologically and let $u=s$ under the new indexing. Then $\tilde{v}_u^2=v_{s}^2\le -3$, $\tilde{v}_{u\pm1}^2\le-3$, and there is an integer $l$ such that $E^{\tilde{S}}_l=\{u-1,u,u+1\}$ and $\tilde{v}_{u-1}\cdot e_l=-\tilde{v}_u\cdot e_l=\tilde{v}_{u+1}\cdot e_l$. Note that if $a_{s+1}\ge3$, then $\tilde{S}=S$. 
    Let $C$ be the subset obtained from $\tilde{S}$ by removing $\tilde{v}_u$, replacing $\tilde{v}_{u\pm1}$ with $\tilde{v}_{u\pm1}'=\pi_{e_l}(\tilde{v}_{u\pm1})$, and setting $\tilde{v}_x'=\tilde{v}_x$ for all $x\not\in\{u-1,u,u+1\}$. Then $I(C)\le -2$, $p_1(C)=0$, $p_2(C)>0$, and $\tilde{v}_{u-1}\cdot \tilde{v}_{u+1}=1$. 
	 If there exists a vertex of $C$ with square at most $-3$, then $C$ is a positive cyclic subset. However, by Proposition \ref{prop:p2>0} in Section \ref{p2>0}, positive cyclic subsets with $p_1=0$ and $p_2>0$ have associated strings in $\mathcal{S}_{2c}\cup \mathcal{S}_{2d}$ and thus have $I\in\{-1,0\}$. Since $I(C)\le-2$, every vertex of $C$ must have square equal to $-2$ and so $\tilde{S}$ has associated string of the form $(3+x,3,2^{[x]},3)$, where $-(\tilde{v}_u)^2=3+x$. Notice that $(3-1)=(2)$ and $(3-1)=(2)$ are reverse linear-dual strings.  Thus by Lemma \ref{centeredlem}, $S$ has associated string of the form $(3+x,b_1,\ldots, b_{k-1},b_k+1,2^{[x]},c_l+1,c_{l-1},\ldots,c_1)\in\mathcal{S}_{2b}$, where $(b_1,\ldots,b_k)$ and $(c_1,\ldots,c_l)$ are linear-dual strings. 
\end{proof}

\begin{lem} Let $S$ be a cyclic subset with $p_1(S)=p_2(S)=I(S)=0$. Suppose that for all $1\le i\le n$, $E_i\neq\{s-1,s,s+1\}$ for some integer $s$. Then $S$ is positive with associated string in $\mathcal{S}_{2c}$.
\label{lem2a.7}
\end{lem}

\begin{proof} 
Let $s$ be an integer such that $a_s\ge3$. Let $i$ be an integer such that $v_s\cdot e_i=-v_{s+1}\cdot e_i$, which exists by Lemmas \ref{lem2a.1} and \ref{lem2a.2}. Finally, let $E_i=\{s-1,s,t\}$. By assumption, $t\not\in\{s-2,s+1\}$. 
Let $x\ge0$ be the smallest integer such that $a_{s+x+1}\ge3$.
Since $a_{s+1}=\cdots=a_{s+x}=2$, we have that $V_{s+\alpha}=\{i_{\alpha-1},i_\alpha\}$ for all $1\le\alpha\le x$, where $i_0:=i$ and $\{i_0,\ldots,i_{x}\}$ contains $x+1$ distinct integers. Since $i\in V_t$ and $v_t\cdot v_{s+\alpha}=0$ for all $1\le\alpha\le x-1$, we have that $i_0,\ldots,i_{x-1}\in V_t$. If $t=s+x+1$, then it is clear that $i_x\not\in V_t$ and so $|E_{i_x}|=1$, which is a contradiction. Thus $v_t\cdot v_{s+x}=0$ and so $i_{x}\in V_t\cap V_{s+x+1}$, and $a_t\ge x+1$. Moreover, since $E_{i_x}=\{s+x,s+x+1,t\}$, by assumption, $t\neq s+x+2$. Now since $v_t\cdot v_{s-1}=v_t\cdot v_{s+x+1}=0$, there exist integers $m_1\in (V_t\setminus\{i_0,\ldots,i_x\})\cap V_{s-1}$ and $m_2\in (V_t\setminus\{i_0,\ldots,i_x\})\cap V_{s+x+1}$, implying that $a_t\ge 2+x$. If $a_t=2+x$, then $m_1=m_2$; set $m:=m_1=m_2$. But then $m\in V_{t\pm1}$, implying that $|E_m|\ge 5$, which is a contradiction. Thus $a_t\ge 3+x$. 
Let $G=\{v_1',\ldots,v_{s-1}',v_{s+x+1}',\ldots,v_{t-1}',v_{t+1}',\ldots,v_n\}\subset\ZZ^{n-x-1}=\langle e_1,\ldots,e_n\rangle/\langle e_{i_0},\ldots,e_{i_{x}}\rangle$, where $v_{s-1}'=\pi_{e_i}(v_{s-1})$, $v_{s+x+1}'=\pi_{e_{i_{x}}}(v_{s+x})$, and $v_\alpha'=v_\alpha$ for all $\alpha\not\in\{s-1,\ldots,s+x+1,t\}$. Then $G$ has two components and $p_1(G)=p_4(G)=0$ and $I(G)\le-2$.
	
We first claim that $G$ is irreducible and thus a good subset. Suppose otherwise. Then $G=G_1\cup G_2$ is the union of two standard subsets $G_1$ and $G_2$. By Proposition \ref{liscasummary1}, $I(G_1),I(G_2)\ge-3$. Since $p_1(G)=p_4(G)=0$, Proposition \ref{liscasummary1} tells us that $I(G_1),I(G_2)\ge0$. Consequently, $-2=I(G)=I(G_1)+I(G_2)\ge0$, a contradiction. Thus $G$ is a good subset. Moreover, by the hypothesis, there do not integers $l$ and $z$ such that $E_l^G=\{z-1,z,z+1\}$, implying that neither component of $G$ is bad (see Definition 4.1 in \cite{liscalensspace}). 
By Proposition \ref{prop:goodsubsets}, $I(G)=-2$ (so $a_t=3+x$) and $G_1$ and $G_2$ have associated strings of the form
$(b_1,\ldots,b_k)$ and $(c_1,\ldots,c_l)$, where  $(b_1,\ldots,b_k)$ and $(c_1,\ldots,c_l)$ are linear-dual strings. 
Thus $G$ has associated string of the form $(b_1,\ldots,b_k,c_1,\ldots,c_l)$ or $(b_1,\ldots,b_k,c_l,\ldots,c_1)$. 

To determine which string is correct, we first claim that $m_1\neq m_2$. Assume otherwise, and set $m:=m_1=m_2$. Since $a_t=3+x$, we have that $V^S_t=\{i_0,\ldots,i_x,m,z\}$ for some integer $z$. Since $E^S_m=\{s-1,s+x+1,t\}$, we necessarily have that $E^S_z=\{t-1,t,t+1\}$, contradicting the hypothesis of the lemma. Thus $m_1\neq m_2$ and $V^S_t=\{i_0,\ldots,i_x,m_1,m_2\}$. Once again by the hypothesis, we may assume that $m_1\in V^S_{t-1}$ and $m_2\in V^S_{t+1}$.  Thus $E_{m_1}^G=\{s-1,t-1\}$ and $E_{m_2}^G=\{s+x+1,t+1\}$. By Proposition \ref{prop:goodsubsets}, $G$ must have associated string $(b_1,\ldots,b_k,c_1,\ldots,c_l)$.
Consequently, $S$ has associated string of the form 
$(3+x,b_1,\ldots,b_{k-1},b_k+1,2^{[x]},c_1+1,c_2,\ldots,c_l)$.
Note that by Lemma \ref{lem:dualconcat}, $(\beta_1,\ldots,\beta_\kappa)=(2+x,b_1,\ldots,b_k)$ has linear-dual $(\gamma_1,\ldots,\gamma_{\lambda})=(2^{[x]},c_1+1,c_2,\ldots,c_1)$; hence $S$ has associated string 
$(\beta_1+1,\beta_2,\ldots,\beta_{\kappa-1},\beta_\kappa+1,\gamma_1,\ldots,\gamma_{\lambda})\in\mathcal{S}_{2c}.$
\end{proof}

\noindent Combining the above two lemmas, we have proven the following.

\begin{prop} Let $S$ be a cyclic subset with $I(S)\le0$ and $p_1(S)=p_2(S)=0$. Then $S$ is positive with associated string in $\mathcal{S}_{2b}\cup\mathcal{S}_{2c}$.
\label{prop2}
\end{prop}

\subsection{Case IIb: $p_2(S)>0$}\label{p2>0}\hfill

Throughout this section, we will consider cyclic subsets satisfying $p_1(S)=0$ and $p_2(S)>0$. In light of Example \ref{example}, we will further restrict ourselves to cyclic subsets containing at least one vertex with square at most $-3$. By the discussion in Section \ref{basecases}, there are no such cyclic subsets of length 2 or 3. Thus we assume that $n\ge4$. We start with some useful notation and some preliminary lemmas.

\begin{definition} Let $S=\{v_1,\ldots,v_n\}\subset\ZZ^n$ be any subset. We define the sets $\mathcal{I}^S$ and $\mathcal{J}^S$ as follows.
	$$\mathcal{I}^S=\{i\,|\,E_i=\{s,t\}\text{ and }a_s=2\text{ or }a_t=2\}$$
	$$\mathcal{J}^S=\{i\,|\,E_i=\{s,t\}\text{ and }a_s,a_t\ge3\}$$ 
\end{definition}

In some cases, we will drop the superscript $S$ from the notation if the subset being considered is understood.
Notice that $p_2(S)=|\mathcal{I}^S|\cup|\mathcal{J}^S|$. For each $i\in\mathcal{I}^S\cup\mathcal{J}^S$, let $E_i=\{s(i),t(i)\}$. 
For each $i\in\mathcal{I}^S$, assume $a_{s(i)}=2$. 

\begin{lem} Let $S$ be cyclic, $I(S)\le0$, $p_1(S)=0$, $p_2(S)>0$, and $n\ge4$. If $i\in\mathcal{I}$, then $a_{t(i)}\ge 3$.\label{lem2b.1}\end{lem}
\begin{proof}
Set $s:=s(i)$ and $t:=t(i)$.
Assume $a_t=2$. 
Suppose $v_s\cdot v_t=0$. Then $V_s=V_t=\{i,j\}$, for some $j$, and $E_j\supseteq\{s-1,s,s+1,t-1,t,t+1\}$. If $n\ge 5$, then either $v_{s-1}\cdot v_t=0$ or $v_{s+1}\cdot v_t=0$, and so $i\in V_{s-1}$ or $i\in V_{s+1}$, which is a contradiction. If $n=4$, then $t\pm1=s\mp1$. Since $v_{t-1}\cdot v_{t+1}=0$, there exists an integer $k$ such that $k\in V_{t\pm1}$. Moreover, there exists a fourth integer $m$ such that $m\in V_{t+1}$ or $V_{t-1}$, but not both, since $v_{t-1}\cdot v_{t+1}=0$. Thus $p_1(S)>0$, contradicting the hypothesis.

Now suppose $|v_s\cdot v_t|=1$ and, without loss of generality, let $t=s+1$. Since $a_s=a_{s+1}=2$, we have that $V_s=\{i,j\}$ and $V_{s+1}=\{i,i_1\}$, where $i_1\neq j$. Let $l\ge 2$ be the smallest integer such that $a_{s+l}\ge3$. Then it is easy to see that $V_{s+\alpha}=\{i_{\alpha-1},i_{\alpha}\}$, for all $1\le \alpha\le l-1$, where $i_0:=i$, $i_{\alpha}\notin\{ i, j\}$ for all $1\le \alpha\le l-1$ and the $i_{\alpha}'s$ are all distinct. Similarly, let $m\ge 1$ be the smallest integer such that $a_{s-m}\ge 3$. Then, as above, $V_{s-\beta}=\{j_{\beta-1},j_{\beta}\}$ for all $1\le \beta\le m-1$, where $j_0:=j$ and the set $\{j_{\beta}, i, i_{\alpha}\}$ has $m+l$ distinct elements. Now since $|v_{s+l-1}\cdot v_{s+l}|=1$, we must have that $V_{s+l-1}\cap V_{s+l}=\{i_{l-1}\}$ and $|v_{s+l}\cdot e_{i_{l-1}}|=1$. Similarly, $V_{s-m+1}\cap V_{s-m}=\{j_{m-1}\}$ and $|v_{s-m}\cdot e_{j_{m-1}}|=1$. Moreover, $E_{i_{\alpha}}=\{s+\alpha, s+\alpha +1\}$ and $E_{j_{\beta}}=\{s-\beta, s-\beta -1\}$ for all $\alpha$ and $\beta$. 

If $v_{s-m}=v_{s+l}=v_u$, then $\{i_{l-1}, j_{m-1}\}\subset V_u$. Since $|v_u\cdot e_{i_{l-1}}|=|v_u\cdot e_{j_{m-1}}|=1$ and $a_u\ge 3$, we must have that $|V_u|\ge3$. Thus there is an integer $p$ such that $E_p=\{u\}$, which contradicts $p_1(S)=0$. Now suppose that $v_{s-m}\neq v_{s+l}$. Let $T=\{v_1',\ldots,v_{s-1}',v_{s+l}',\ldots, v_n'\}\subset\ZZ^{n-(m+l)}=\langle e_1,\ldots,e_n\rangle/\langle e_{i_0},\ldots,e_{i_{l-1}},e_{j_0},\ldots,e_{j_{m-1}}\rangle$, where $v_{s-m}'=\pi_{e_{j_{m-1}}}(v_{s-m})$ and $v_{s+l}'=\pi_{e_{i_{l-1}}}(v_{s+l})$. Since $|v_{s+l}\cdot e_{i_{l-1}}|=|v_{s-m}\cdot e_{j_{m-1}}|=1$ and $a_{s-m},a_{s+l}\ge 3$, we have that $(v_{s-m}')^2, (v_{s+l}')^2\le -2$. Thus $T$ is a standard subset made of $n-(l+m-1)$ vectors. However, by Remark \ref{LIremark}, these vectors are linearly independent in $\ZZ^{n-(l+m)}$, which is not possible. 
\end{proof}

\begin{lem}  Let $S$ be cyclic, $I(S)\le 0$, $p_1(S)=0$, $p_2(S)>0$, and $n\ge4$. If $i\in\mathcal{I}$, then $v_{s(i)}\cdot v_{t(i)}=0$.\label{lem2b.2}\end{lem}

\begin{proof} Set $s:=s(i)$ and $t:=t(i)$. Let $V_s=\{i,j\}$. Then by Lemma \ref{lem2b.1}, $a_t\ge 3$. Assume $|v_s\cdot v_t|=1$ and without loss of generality, assume $t=s+1$. Then $\{s-1,s\}\subseteq E_j$. If there exists an integer $u\notin\{s-1,s,s+1\}$ such that $u\in E_j$, then we necessarily have that $i\in V_u$, implying that $|E_i|\ge 3$, which is not possible. Thus either $E_j=\{s-1,s\}$ or $E_j=\{s-1,s,s+1\}$.

If $E_j=\{s-1,s\}$, then by Lemma \ref{lem2b.1}, $a_{s-1}\ge3$. Moreover, since $|v_s\cdot e_i|=|v_s\cdot e_j|=1$, $V_s\cap V_{s-1}=\{j\}$, and $V_s\cap V_{s+1}=\{i\}$, we have $|v_{s+1}\cdot e_i|=|v_{s-1}\cdot e_j|=1$. Let $T=\{v_1',\ldots,v_{s-1}',v_{s+1}',\ldots,v_n'\}\subset\ZZ^{n-2}=\langle e_1,\ldots,e_n\rangle/\langle e_i,e_j\rangle$, where $v_{s+1}'=\pi_{e_j}(v_{s+1})$, $v_{s-1}'=\pi_{e_j}(v_{s-1})$, and $v_x'=v_x$ for all $x\not\in\{s-1,s,s+1\}$. Then $(v_{s\pm1}')^2\le -2$ and $v_{s-1}'\cdot v_{s+1}'=0$. Thus $T$ is standard with final vertices $v_{s-1}'$ and $v_{s+1}'$. By Remark \ref{LIremark}, $T\subset\ZZ^{n-2}$ contains $n-1$ linearly independent vectors, which is impossible.

If $E_j=\{s-1,s,s+1\}$, then since $v_{s-1}\cdot v_{s+1}=0$, there exists an integer $k\notin\{i,j\}$ such that $k\in V_{s-1}\cap V_{s+1}$. Moreover, $|v_{s-1}\cdot e_j|=1$ and since $V_{s+1}\cap V_s=\{i,j\}$ and $|v_{s+1}\cdot v_s|=1$, we have that $|v_{s+1}\cdot e_i|=x$ and $|v_{s+1}\cdot e_j|=x\pm1$, where $x,x\pm1\neq0$. Thus $a_{s+1}\ge x^2+(x\pm1)^2+1\ge6$. If $|v_{s+1}\cdot e_i|=x\ge2$, let $T=\{v_1',\ldots,v_{s-1}',v_{s+1}',\ldots,v_n'\}\subset\ZZ^{n-1}=\langle e_1,\ldots e_n\rangle/\langle e_i\rangle$, where $v_{s+1}'=\pi_{e_i}(v_{s+1})$ and $v_x'=v_x$ for all $x\not\in\{s,s+1\}$. Then $T$ is standard and $0\ge I(S)=I(T)+x^2+(a_s-3)=I(T)+x^2-1.$ Thus $I(T)\le 1-x^2<0$ and so by Proposition \ref{liscasummary1}, we necessarily have that $I(T)=-3$ and $p_1(T)=1$. But then $p_1(S)=p_1(T)=1$, which contradicts our assumption that $p_1(S)=0$. Now suppose $|v_{s+1}\cdot e_i|=1$ so that $|v_{s+1}\cdot e_j|=2$. Since  $|v_{s-1}\cdot e_j|=1$ and $|v_{s-1}\cdot v_{s+1}|=0$, we have that either $a_{s-1}\ge3$ or $a_{s-1}=2$ and $|v_{s+1}\cdot e_k|=2$. In the latter case, note that $E_k=\{s-2,s-1,s+1\}$ and $E_j=\{s-1,s,s+1\}$. In this case, let $T'=\{v_1',\ldots,v_{s-2}',v_{s+1}',\ldots,v_n'\}=\subset\ZZ^{n-2}=\langle e_1,\ldots,e_n\rangle/\langle e_i,e_j\rangle$, where $v_{s+1}'=\pi_{e_i}(\pi_{e_j}(v_{s+1}))$ and $v_x'=v_x$ for all $x\not\in\{s-1,s,s+1\}$. Then $T'$ is standard with $p_1(T')=0$ and $0\ge I(S)=I(T')+5+(a_{s-1}-3)+(a_s-3)=I(T')+3$, implying that $I(T')\le-3$. But by  Proposition \ref{liscasummary1}, no such subset exists. In the former case ($a_{s-1}\ge3$), let $T''=\{v_1',\ldots,v_{s-1}',v_{s+2}',\ldots,v_n'\}\subset\ZZ^{n-2}=\langle e_1,\ldots,e_n\rangle/\langle e_i,e_j\rangle$, where $v_{s-1}'=\pi_{e_j}(v_{s-1})$ and $v_x'=v_x$ for all $x\not\in\{s-1,s,s+1\}$. Then $T''$ is a standard subset such that $0\ge I(S)=I(T'')+1+(a_s-3)+(a_{s+1}-3)\ge I(T'')+3$. By Proposition \ref{liscasummary1}, we necessarily have that $I(T'')=-3$ and $p_1(T'')=1$. Thus $a_{s+1}=6$ and $V^S_{s+1}=\{i,j,k\}$. This implies that $|E_k^{T''}|=1$. But $k\in V_{s-1}^{T''}$ and $v_{s-1}$ is a final vertex of $T''$. By Proposition \ref{liscasummary2}(a), no such standard subset exists.
\end{proof}

\begin{lem} Let $S$ be cyclic, $I(S)\le0$, $p_1(S)=0$, $|\mathcal{I}|>0$, and $n\ge4$.
	\begin{enumerate}[(a)]
		\item If there exist integers $i,i'\in\mathcal{I}$ such that $|v_{s(i)}\cdot v_{s(i')}|=1$, then $S$ is negative and has associated string in $\mathcal{S}_{1d}$, $|\mathcal{J}|=0$, and $|v_{\alpha}\cdot v_{\beta}|\le1$ for all $1\le \alpha,\beta\le n$.
		\item If $v_{s(i)}\cdot v_{s(i')}=0$ for all $i,i'\in\mathcal{I}$, then $p_4(S)\ge |\mathcal{I}|$.
		\end{enumerate}
	\label{lem:noadjacent}
\end{lem}

\begin{proof} 
	Suppose $|v_{s(i)}\cdot v_{s(i')}|=1$ and without loss of generality, assume $s(i')=s(i)+1$. Then $t(i)=s(i)+2$, $t(i')=s(i)-1$, and there exists an integer $j$ such that $E_{j}=\{s(i)-1,s(i),s(i)+1,s(i)+2\}$. Set $s:=s(i)$. By Lemma \ref{lem2b.1}, $a_{s-1}, a_{s+2}\ge3$; consequently, $n\ge 5$. 
	Without loss of generality, assume $v_{s-1}\cdot v_{s}=v_{s}\cdot v_{s+1}=1$ so that $v_{s-1}\cdot e_j=-v_{s}\cdot e_j=v_{s+1}\cdot e_j \in\{\pm1\}$.
	Let $S'=\{v_1',\ldots,v_{s-1}',v_{s+1}',\ldots,v_n'\}\subset\ZZ^{n-1}=\langle e_1,\ldots,e_n\rangle/\langle e_i\rangle$, where $v_{s+2}'=\pi_{e_i}(v_{s+2})$, $v_{s-1}'=\pi_{e_{i'}}(v_{s-1})$, and $v_x':=v_x$ for all $x\not\in\{s-1,s,s+2\}$. Then $S'$ is cyclic with $I(S')=I(S)-1<0$ and $p_1(S')=1$ (since $E_{i'}^{S'}=\{s+1\}$). Moreover, $v_{s-1}'\cdot e_j=v_{s+1}'\cdot e_j$ and so $S'$ is positive if and only if $S$ is negative. By the proof of Proposition \ref{prop1}, the only cyclic subset with $p_1=1$ and $I<0$ is positive and has associated string of the form $(2,b_1+1,b_2,\ldots,b_k,2,c_l,\ldots,c_2,c_1+1,2)\in\mathcal{S}_{2e}$. Moreover, the vertex with square 2 in the middle of the string is $v_{s+1}'$. Thus $S$ is negative and has associated string of the form $(2,b_1+1,b_2,\ldots,b_k+1,2,2,c_l+1,\ldots,c_2,c_1+1,2)\in\mathcal{S}_{1d}$. Furthermore, by the proof of Proposition \ref{prop1}, it is easy to see that $|v_\alpha'\cdot v_\beta'|\le1$ for all $\alpha,\beta$ and $|\mathcal{J}^{S'}|=0$; hence $|v_{\alpha}\cdot v_{\beta}|\le1$ for all $1\le \alpha,\beta\le n$ and $|\mathcal{J}^S|=0$.
	
	By Lemma \ref{lem2b.2}, for all $i\in\mathcal{I}^S$, there exists an integer $j(i)$ such that $E_{j(i)}=\{s(i)-1,s(i),s(i)+1,t(i)\}$. If $v_{s(i)}\cdot v_{s(i')}=0$ for some $i,i'\in\mathcal{I}^S$, it follows that $j(i)\neq j(i')$; hence if $v_{s(i)}\cdot v_{s(i')}=0$ for all $i,i'\in\mathcal{I}^S$, then $p_4(S)\ge |\mathcal{I}^S|$.
\end{proof}

\begin{lem}  Let $S$ be cyclic, $I(S)\le 0$, $p_1(S)=0$, $p_2(S)>0$, and $n\ge 4$. Then $|v_\alpha\cdot e_\beta|\le 1$ for all integers $\alpha$ and $\beta$.\label{lem2b.5}\end{lem}
\begin{proof}
By Lemma \ref{lem:noadjacent}, we may assume that $v_{s(i)}\cdot v_{s(i')}=0$ for all $i,i'\in\mathcal{I}$ so that $p_4(S)\ge|\mathcal{I}|$.
First suppose that $|\mathcal{J}|\neq0$. Let $i\in \mathcal{J}$ and set $s:=s(i)$ and $t:=t(i)$. Notice that we cannot have $|V_s|=|V_t|=2$. Without loss of generality, assume that $|V_s|\ge3$. Let $T=\{v_1',...,v_{t-1}',v_{t+1}',...,v_{n}'\}\subset\ZZ^{n-1}=\langle e_1,\ldots,e_n\rangle/\langle e_i\rangle$, where $v_s'=\pi_{e_i}(v_s)$ and $v_x'=v_x$ for all $x\not\in\{s,t\}$. Then $(v_s')^2\le -2$ and $v'_{t-1}\cdot v'_{t+1}=0$, and so $T$ is standard. Let $|v_s\cdot e_i|=x\ge 1$. Then
$$0\ge I(S)=I(T)+x^2+(a_t-3)\ge I(T)+x^2\ge I(T)+1.$$ 
Thus $I(T)\le -1$ and so by Proposition \ref{liscasummary1}, $I(T)\in\{-1,-2,-3\}$. Thus $a_t\le5$ and $|v_s\cdot e_i|=x=1$. Moreover, by Proposition \ref{liscasummary1}, we have that $|v'_\alpha\cdot e_\beta|\le 1$ for all $\alpha,\beta$. Thus $|v_\alpha\cdot e_\beta|\le 1$ for all $\alpha\neq t$ and for all $\beta$. If $|v_t\cdot e_j|\ge 2$ for some $j$, then since $a_t\le 5$, necessarily have $V_t=\{i,j\}$ and $a_t=5$; consequently $I(T)=-3$ and by Proposition \ref{liscasummary1}, $p_1(T)=1$. In particular, $|E^T_j|=1$ and $E^S_j=\{s,t\}$. If 
$v_s\cdot v_t=0$, then $v_t\cdot v_{t\pm1}=0$, which is a contradiction. If $|v_s\cdot v_t|=1$ and, say, $t=s+1$, then $v_{s+1}\cdot v_{s+2}=0$, which is a contradiction.
Thus $|v_\alpha\cdot e_\beta|\le 1$ for all $\alpha$, $\beta$.

Now suppose $|\mathcal{J}|=0$. Then $p_4(S)\ge p_2(S)$ and so by Lemma \ref{keylem1}, $I(S)=0$, $p_2(S)=p_4(S)$ and $p_j(S)=0$ for all $j=5,...,n$. Thus $p_3(S)=n-2p_2(S)$. Let $m_{ij}:=v_i\cdot e_j$. Then $$3n=\sum a_i= \sum_{i,j}m_{ij}^2\ge\sum_{i,j} |m_{ij}|\ge\sum ip_i(S)=2p_2(S)+4p_2(S)+3(n-2p_2(S))=3n.$$ Thus $|v_i\cdot e_j|=|m_{ij}|\le1$ for all $i,j$.
\end{proof}

In light of Lemma \ref{lem2b.5}, it will now be a standing assumption that $|v_\alpha\cdot e_\beta|\le 1$ for all integers $\alpha$ and $\beta$.

\begin{lem} Suppose $S$ is cyclic with $n\ge4$ and $|\mathcal{J}|\neq0$. 
	If there exists $i\in\mathcal{J}$ with $a_{s(i)},a_{t(i)}\ge4$, then $S$ is positive has associated string $(4,4,2,2,2)\in\mathcal{S}_{2d}$.
	\label{lem:as=3}
\end{lem}

\begin{proof} By cyclically reordering and negating vertices, we may assume $s(i)=1$ and $t(i)=k$ for some integer $k$. Let $R=\{v_1',\ldots,v_n'\}\subset\ZZ^{n-1}=\langle e_1,\ldots,e_n\rangle/\langle e_i\rangle$, where $v_1'=\pi_{e_i}(v_1)$, $v_k'=\pi_{e_i}(v_k)$, and $v_i':=v_i$ for all $i\neq 1,k$. 
	
\textbf{\uline{Case 1: $v_{1}\cdot v_{k}=0$}}. First suppose $v_{1}\cdot v_{k}=0$ (so $k\not\in\{2,n\}$).
	By Lemma \ref{lem2b.5}, $-(v_1')^2=a_1-1$, $-(v_k')^2=a_k-1$, and $v_1'\cdot v_k'=\pm1$. Let $A$ be the intersection matrix $A=(v_i'\cdot v_j')$. Assume $a_{1},a_k\ge4$. By Lemma \ref{lem:negdef1} in the Appendix, if $S$ is negative cyclic or $S$ is positive cyclic with $v_1\cdot e_i=-v_k\cdot e_i$, then $A$ is negative-definite; in these cases $R$ is a linearly independent set of $n$ vectors in $\ZZ^{n-1}$, which is not possible. 
	Thus we may assume that $S$ is positive and $v_{1}\cdot e_i=v_{k}\cdot e_i$. Again, by Lemma \ref{lem:negdef1}, we arrive at another linear independence contradiction unless $a_{1}=a_{k}=4$ and $a_x=2$ for all $x\not\in\{1,k\}$. Thus $I(S)=-(n-4)$.
	Let $T=\{v_2',\ldots,v_n'\}\subset \ZZ^{n-1}=\langle e_1,\ldots,e_n\rangle/\langle e_i\rangle$, where $v_k'=\pi_{e_i}(v_k)$ and $v_x'=v_x$ for all $x\not\in\{1, k\}$. Then $T$ is a standard subset and $I(T)=I(S)-2=-(n-2)$. Since $I(T)\ge-3$ by Proposition \ref{liscasummary1}, it follows that $n\le 5$. If $n=5$, then $I(S)=-1$, $I(T)=-3$, and $T$ has length 4. By Proposition \ref{liscasummary2}(1), up to reversal, $T$ has associated string of the form $(3,2,2,2)$. Since $a_t=4$, this implies that $k=2$, a contradiction. If $n=4$, then $I(S)=0$, $I(T)=-2$, and $T$ has length 3. But by Proposition \ref{liscasummary2}(2), no such standard subset exists.
	
\textbf{\uline{Case 2: $|v_{1}\cdot v_{k}|=1$}}. Next suppose $|v_{1}\cdot v_k|=1$ and without loss of generality, assume $k=2$. 
	 If $v_1\cdot e_i=-v_2\cdot e_i$, then $v_1'\cdot v_2'=0$; hence $R$ is standard and so by Remark \ref{LIremark}, $R$ is a linearly independent set of $n$ vectors in $\ZZ^{n-1}$, a contradiction.	
	If $v_1\cdot e_i=v_2\cdot e_i$, then $v_1'\cdot v_2'=2$; by applying Lemma \ref{lem:negdef2} as in Case 1, we obtain a contradiction unless $S$ is positive, $a_1=a_2=4$, and $a_3=\cdots=a_n=2$. In this case, $I(S)=-(n-4)$. As in Case 1, we necessarily have that $n\le5$. If $n=4$, then $I(T)=-2$ and $T$ has length 3; by Proposition \ref{liscasummary2}(2), no such subset exists. Suppose $n=5$ so that $I(T)=-3$ and $T$ has length 4. By Proposition \ref{liscasummary2}(1), up to reversal, $T$ has associated string of the form $(3,2,2,2)$. Hence $S$ is positive and has associated string of the form $(4,4,2,2,2)\in\mathcal{S}_{2d}$.
\end{proof}

We are now ready to finish the classification of cyclic subsets with $I(S)\le0$, $p_1(S)=0$, and $p_2(S)>0$. We will consider two cases: $|\mathcal{J}|\neq0$ and $|\mathcal{J}|=0$. These cases are handled respectively in the next two propositions. 

\begin{prop}  Let $S$ be cyclic, $I(S)\le0$, $p_1(S)=0$, $p_2(S)>0$, and $n\ge 4$. If $|\mathcal{J}|\neq0$, then $S$ is positive with associated string in $\mathcal{S}_{2c}\cup\mathcal{S}_{2d}$ or negative with associated string in $\mathcal{S}_{1c}\cup\mathcal{S}_{1e}\cup\mathcal{O}$.
\label{lem2b.6}\end{prop}

\begin{proof} Let $i\in\mathcal{J}$ and set $s:=s(i)$ and $t:=t(i)$. If $a_s,a_t\ge4$, then by Lemma \ref{lem:as=3}, $S$ is positive with associated string in $\mathcal{S}_{2d}$. Without loss of generality, we may now assume that $a_s=3$.
Moreover, by Lemma \ref{lem:noadjacent} $v_{s(i_1)}\cdot v_{s(i_2)}=0$ for all $i_1,i_2\in\mathcal{I}$, implying that $p_4(S)\ge |\mathcal{I}|$.
Let $T=\{v_1',\ldots,v_{s-1}',v_{s+1}',\ldots,v_n'\}\subset \ZZ^{n-1}=\langle e_1,\ldots,e_n\rangle/\langle e_i\rangle$, where $v_t'=\pi_{e_i}(v_t)$ and $v_x'=v_x$ for all $x\not\in\{ s, t\}$. By Lemma \ref{lem2b.5}, $(v_t')^2=v_t^2+1$ and so $T$ is a standard subset and $I(T)=I(S)-1\le-1$. By Proposition \ref{liscasummary1}, $I(T)\in\{-3,-2,-1\}$.
We will work case-by-case, considering each of the standard subsets listed in Proposition \ref{liscasummary2}.\\ 

\noindent\textbf{ \uline{Case 1: $I(T)=-1$}}\\
\indent Suppose $I(T)=-1$.  By Proposition \ref{liscasummary1}, $p_1(T)=0$, $p_2(T)=2$, $p_4(T)=1$, and $p_j(T)=0$ for all $j\ge5$. Thus $p_2(S)\le 3$, $p_4(S)\le 3$, $p_5(S)\le 1$, and $p_j(S)=0$ for all $j\ge6$. Note that since $a_s=3$: if $p_5(S)=1$, then $p_4(S)=p_2(S)-2$; and if $p_5(S)=0$, then $p_2(S)=p_4(S)$. By Lemma \ref{lem:p2+p4}, $p_2(S)+p_4(S)\equiv0\mod4$, implying that either: $p_5(S)=1$, $p_2(S)=3$, and $p_4(S)=1$; or $p_5(S)=0$ and $p_2(S)=p_4(S)=2$.
By Proposition \ref{liscasummary2}(3), $T$ is of one of the following forms:
\begin{enumerate}[(a)]
	\item $\{e_2+e_4+\sum_{\alpha=5}^{x+4}e_\alpha,e_1-e_2+\sum_{\alpha=x+5}^{x+y+4}e_\alpha,e_2-e_3-e_4,e_4-e_5,e_5-e_6,\ldots,e_{x+3}-e_{x+4},\\e_{x+4}-e_1-e_2-e_3,e_1-e_{x+5},e_{x+5}-e_{x+6},\ldots,e_{x+y+3}-e_{x+y+4}\}$
	\item $\{e_2+e_4+\sum_{\alpha=5}^{x+4}e_\alpha,e_1-e_2,e_2-e_3-e_4-\sum_{\alpha=x+5}^{x+y+4}e_\alpha,e_4-e_5,\ldots,e_{x+3}-e_{x+4},\\e_{x+4}-e_1-e_2-e_3,e_3-e_{x+5},e_{x+5}-e_{x+6},\ldots,e_{x+y+3}-e_{x+y+4}\}$
	\item $\{e_1-e_2-e_5-\sum_{\alpha=6}^{x+5}e_\alpha,e_2+e_3,-e_2-e_1-e_4-\sum_{\alpha=x+6}^{x+y+5}e_\alpha,-e_5+e_2-e_3,\\e_5-e_6,e_6-e_7,\ldots,e_{x+4}-e_{x+5},e_{x+5}+e_1-e_4,e_4-e_{x+6},e_{x+6}-e_{x+7},\ldots, \\e_{x+y+4}-e_{x+y+5}\}.$
\end{enumerate}

\noindent\uline{Case 1(a)}:
 Consider the first form. Without loss of generality, we may assume that the listed vertices are $v_{s+1}',\ldots,v_n',v_1',\ldots,v_{s-1}'$. First assume $p_5(S)=1$, $p_2(S)=3$, and $p_4(S)=1$. Then $2\in V_s^S$ and $3,x+y+4\not\in V_s^S$ (where $x+y+4=1$ if $y=0$). If $y=0$, then since $v_{s+2}\cdot v_s=0$ and $1\not\in V_s^S$, we have that $i\in V_{s+2}^S$. Since $v_{s+3}\cdot v_s=0$ and $2\in V_s^S$, we have that $4\in V_s^S$ and $v_s\cdot e_2=v_s\cdot e_4$. Since $V_s^S=\{i,2,4\}$, if $x\ge1$, then $v_{s+4}\cdot v_s\neq0$, which is a contradiction, and if $x=0$, then $v_{s-1}\cdot v_s=0$, which is a contradiction. Thus we may assume $y\ge 1$.
Since $v_s\cdot v_{s+2}=v_{s}\cdot v_{s+3}=0$ and $a_s=3$, either: $i\in V_{s+2}^S$ and $4\in V_s^S$; or $i\in V_{s+3}^S$ and $|\{1,x+5,\ldots,x+y+3\}\cap V_s^S|=1$. In the former case, $V_s^S=\{i,2,4\}$ and so $|v_s\cdot v_{s+1}|\neq 1$, which is a contradiction. In the latter case,
if $1\in V_s^S$, then $V_s^S=\{i,1,2\}$ and $v_s\cdot e_1=v_s\cdot e_2$ (since $v_s\cdot v_{s+2}=0$); but then $|v_{s+x+4}\cdot v_s|=2$, which is a contradiction. On the other hand, if $|\{x+5,\ldots,x+y+3\}\cap V_s^S|=1$, then since $v_s\cdot v_{s-\alpha}=0$ for all $2\le \alpha\le y$, $\{x+5,\ldots,x+y+3\}\subset V_s^S$, implying that $y=1$ and $1\in V_s^S$, which is again a contradiction.

Now assume $p_5(S)=0$ and $p_2(S)=p_4(S)=2$. Then $2\not\in V_s^S$ and either $x+y+4\in V_s^S$ or $3\in V_s^S$, but not both (where $x+y+4=1$ if $y=0$). First
assume $x+y+4\in V^S_s$. Since $x+y+4\in V_{s+2}^S$ and $v_{s+2}\cdot v_s=0$, either $|\{1,x+5,\ldots,x+y+3\}\cap V_s^S|=1$ or $i\in V_{s+2}^S$. In the former case, $y\ge1$ and since $v_{s-\alpha}\cdot v_s=0$ for all $2\le\alpha\le y$, it follows that $\{1,x+5,\ldots,x+y+3\}\subset V_s^S$, implying that $|v_s\cdot v_{s-1}|\neq1$, which is a contradiction. In the latter case, since $|v_s\cdot v_{s+1}|=1$, we have $|\{4,5,\ldots,x+4\}\cap V^S_s|=1$. Since $v_{s+\alpha}\cdot v_s=0$ for all $4\le\alpha\le x+4$, we have that $\{4,5,\ldots,x+4\}\subset V^S_s$, which implies that $x=0$ and $V_s^S=\{i,4,x+y+4\}$; but then $|v_{s+3}\cdot v_s|=1$, which is a contradiction.

Now suppose $3\in V_s^S$. 
Since $v_s\cdot v_{s+3}=0$ and $3\in V_{s+3}^S$, we have that either $i\in V_{s+3}^S$ or $4\in V_s^S$. In the former case, since $|v_s\cdot v_{s+1}|=1$, we have $|\{4,5,\ldots,x+4\}\cap V^S_s|=1$. As in the previous case, we see that $x=0$ and $V_s^S=\{i,3,4\}$ and so $v_{s+3}\cdot v_s\neq0$, which is a contradiction. In the latter case, since $4\in V_{s+4}^S$, we have that $i\in V_{s+4}^S$ and since $|v_s\cdot v_{s-1}|=1$, we necessarily have that $y=0$.
Consequently, $S$ is of the form\\

\noindent $S=\{\textcolor{blue}{e_i-e_4+e_3}, e_2+e_4+\sum_{\alpha=5}^{x+4}e_\alpha,e_1-e_2,e_2-e_3-e_4,\textcolor{blue}{e_i+}e_4-e_5,$\\
\indent\indent\indent$e_5-e_6,\ldots,e_{x+3}-e_{x+4},e_{x+4}-e_1-e_2-e_3\},$\\

\noindent which is positive and have associated string $(3,2+x,2,3,3,2^{[x-1]},4)\in\mathcal{S}_{2c}$.\\

\noindent\uline{Case 1(b)}:
 Next consider the second form. As in the previous case, we may label the vertices  $v_{s+1}',\ldots,v_n',v_1',\ldots,v_{s-1}'$. 
Note that if $y=0$, then $S$ is also of the form in Case 1(a), which we already covered. 
Thus we may assume $y\ge1$. Consequently, $|\mathcal{I}^T|=2$. If $p_5(S)=1$, then $p_2(S)=3$ and so $|\mathcal{I}^S|=2$; but we also have that $p_4(S)=1\ge |\mathcal{I}^S|$, which is a contradiction. Thus $p_5(S)=0$ and $p_2(S)=p_4(S)=2$; hence $2\not\in V_s^S$ and either $1\in V_s^S$ or $x+y+4\in V_s^S$, but not both. Assume $x+y+4\in V_s^S$. 
Since $x+y+4\in V_{s+3}^S$, we have that either $i\in V_{s+3}^S$ or $|\{3,4,x+5,\ldots,x+y+3\}\cap V_s^S|=1$. In the former case, since $|v_s\cdot v_{s+1}|=1$, following as in Case 1(a) we see that $x=0$ and $V_s^S=\{i,x+y+4,4\}$, which implies that $|v_s\cdot v_{s+3}|=1$, which is a contradiction. In the latter case, since $v_{s-\alpha}\cdot v_s=0$ for all $2\le\alpha\le y$, it is clear that $3,x+5,\ldots,x+y+3\not\in V_s^S$ and so $4\in V_s^S$.
Since $4,x+y+4\in V_{s+3}^S$ and $4\in V_{s+4}^S$, we have  $i\in V_{s+4}^S$. 
Hence: if $x\ge1$, $S$ is of the form 
\\

\noindent $S=\{\textcolor{blue}{e_i-e_4+e_{x+y+4}}, e_2+e_4+\sum_{\alpha=5}^{x+4}e_\alpha,e_1-e_2,e_2-e_3-e_4-\sum_{\alpha=x+5}^{x+y+4}e_\alpha,$\\
\indent\indent\indent$\textcolor{blue}{e_i+}e_4-e_5,\ldots,e_{x+3}-e_{x+4},e_{x+4}-e_1-e_2-e_3,e_3-e_{x+5},e_{x+5}-e_{x+6},\ldots,e_{x+y+3}-e_{x+y+4}\},$\\

\noindent which is positive and has associated string $(3,2+x,2,3+y,3,2^{[x-1]},4,2^{[y]})\in\mathcal{S}_{2c}$; and if $x=0$, then $S$ is of the form \\

\noindent $S=\{\textcolor{blue}{e_i-e_4+e_{y+4}}, e_2+e_4,e_1-e_2,e_2-e_3-e_4-\sum_{\alpha=5}^{y+4}e_\alpha,\textcolor{blue}{e_i+}e_4-e_1-e_2-e_3,$\\
\indent\indent\indent$e_3-e_{5},e_{6}-e_{7},\ldots,e_{y+3}-e_{y+4}\},$\\

\noindent which is positive and has associated string $(3,2,2,3+y,5,2^{[y]})\in\mathcal{S}_{2c}$.

Next assume $3\in V_s^S$. Since $v_s\cdot v_{s+3}=v_s\cdot v_{s+x+4}=0$ and $3\in V_{s+3}^S\cap V_{s+x+4}^S$, either $i\in V_{s+3}^S$ or $i\in V_{s+x+4}^S$. Since $y\ge 1$ and $|v_{s-1}\cdot v_s|=1$, it follows that $x+y+3\in V_s^S$ (where $x+y+3=3$ if $y=1$). But then $v_s\cdot v_{s+1}=0$, which is a contradiction. \\

\noindent\uline{Case 1(c)}:
 Consider the third form and label the vertices  $v_{s+1}',\ldots,v_n',v_1',\ldots,v_{s-1}'$. 
Notice that $|\mathcal{I}^T|=2$ if $y\ge1$. By the same argument as in Case 1(b), if $y\ge1$, then $p_5(S)\neq0$. Suppose $y=0$, $p_5(S)=1$, and $p_2(S)=3$. Then $2\in V_s^S$ and $3,4\not\in V_s^S$. Since $2,3\in V_{s+2}^S$ and $v_s\cdot v_{s+2}=0$, we necessarily have that $i\in V_{s+2}^S$. Now since $V_{s+3}^S\cap V_{s+4}^S=\{2\}$, it follows that either $v_s\cdot v_{s+3}\neq0$ or $v_s\cdot v_{s+4}\neq0$, which is a contradiction. Thus we may assume that $p_5(S)=0$ and $p_2(S)=p_4(S)=2$. Thus $2\not\in V_s^S$ and either $3\in V_s^S$ or $x+y+5\in V_s^S$, but not both (where $x+y+5=4$ if $y=0$).
If $x+y+5\in V_{s+3}^S$, then either $i\in V_{s+3}^S$ or  $|\{1,4,x+6,\ldots,x+y+3\}\cap V_s^S|=1$. In the former case, we obtain a contradictions as in Cases 1(a) and 1(b). In the latter case, we obtain similar contradictions unless $1\in V_s^S$. In this case,
since $1,x+y+5\in V_{s+3}^S$ and $1\in V_{s+x+5}^S$, we have that $i\in V_{s+x+4}^S$.
Thus $S$ is of the form\\

\noindent $S=\{\textcolor{blue}{e_i-e_1+e_{x+y+5}},e_1-e_2-e_5-\sum_{\alpha=6}^{x+5}e_\alpha,e_2+e_3,-e_2-e_1-e_4-\sum_{\alpha=x+6}^{x+y+5}e_\alpha,-e_5+e_2-e_3,\\$
\indent\indent\indent$e_5-e_6,\ldots,e_{x+4}-e_{x+5},\textcolor{blue}{-e_i+}e_{x+5}+e_1-e_4,e_4-e_{x+6},\ldots, e_{x+y+4}-e_{x+y+5}\},$\\

\noindent which is positive and has associated string $(3,3+x,2,3+y,3,2^{[x]},4,2^{[y]})\in\mathcal{S}_{2c}$. 

Next suppose $3\in V_s^S$. Since $2\not\in V_{s+2}^S$ and $v_s\cdot v_{s+2}=0$, we necessarily have that $i\in V_{s+2}^S$. Since $v_s\cdot v_{s+4}=0$, we have that $5\in V_s^S$ and so $V_s^S=\{i,3,5\}$. Moreover, since $5\in V_{s+5}^S$, $v_s\cdot v_{s+5}=0$, and $|v_s\cdot v_{s-1}|=1$ we must have that $x=y=0$. Hence $S$ is of the form\\

\noindent$S=\{\textcolor{blue}{e_i-e_3+e_5},e_1-e_2-e_5,e_2+e_3\textcolor{blue}{+e_i},-e_2-e_1-e_4,-e_5+e_2-e_3,e_{5}+e_1-e_4\}$,\\
 
\noindent which is negative cyclic with associated string $(3,3,3,3,3,3)\in\mathcal{O}$.\\

\noindent\textbf{\uline{Case 2: $I(T)=-2$}}\\
\indent Next suppose $I(T)=-2$ (so that $I(S)=-1$). Then by Proposition \ref{liscasummary2}(2), $p_1(T)=0$, $p_2(T)=3$, $p_4(T)=1$, $p_j(T)=0$ for all $j\ge5$, and  $|\mathcal{I}^T|=2$. Then since $a_s=3$, $p_2(S)\le 4$, $p_4(S)\le 3$, and $p_5(S)\le1$. By Lemma \ref{lem:p2+p4}, $p_2(S)+p_4(S)=1\mod4$. By a similar argument as in Case 1(b), $p_5(S)=0$ and so $p_2(S)=3$ and $p_4(S)=2$.
By Proposition \ref{liscasummary2}(2), $T$ is of one of the following forms:

\begin{enumerate}[(a)]
	\item $\{e_{x+4}-e_{x+3},e_{x+3}-e_{x+2},\ldots,e_5-e_4,e_4-e_2-e_3,e_2+e_1+\sum_{i=x+5}^{x+y+4}e_i,\\-e_2-e_4-\sum_{i=5}^{x+4}e_i,e_2-e_1-e_3,e_1-e_{x+5},e_{x+5}-e_{x+6},\ldots,e_{x+y+3}-e_{x+y+4}\}$
	\item $\{e_{x+4}-e_{x+3},e_{x+3}-e_{x+2},\ldots,e_5-e_4,e_4-e_2-e_3-\sum_{i=x+5}^{x+y+4}e_i,e_2+e_1,\\-e_2-e_4-\sum_{i=5}^{x+4}e_i,e_2-e_1-e_3,e_3-e_{x+5},e_{x+5}-e_{x+6},\ldots,e_{x+y+3}-e_{x+y+4}\}$
	\item $\{u_1,\ldots,u_{k-1}, u_k +e_4-e_2-e_3, e_2+e_1,-e_2-e_4,e_2-e_1-e_3+w_1,w_2,\ldots,w_l\}$, where $k+l\ge 3$, $u_k=0$ or $w_1=0$, $|E_1|=|E_4|=2$. Furthermore (up to reversal), we may assume that $u_1^2=-2$; consequently, there exist integers $j_1,j_2$ such that $|E_{j_1}|=2$, $|E_{j_2}|=3$, $u_1\cdot e_{j_2}=-u_2\cdot e_{j_2}=-w_l\cdot e_{j_2}=1$, and $|u_1\cdot e_{j_2}|=|w_l\cdot e_{j_2}|=1$. 
\end{enumerate}

\noindent\uline{Case 2(a)}: Consider the first form and label the vertices  $v_{s+1}',\ldots,v_n',v_1',\ldots,v_{s-1}'$.
Notice if $y=0$, then $T$ is also of the form given in Case 2(b). Moreover, if $x=0$, then the reverse of $T$ is of the form given in Case 2(b). We will assume that $x,y\ge1$ and handle the cases $x=0$ and $y=0$ in Case 2(b). 
Since $p_5(S)=0$ and $p_2(S)=3$, we have that $2\not\in V_s^S$ and $|\{x+4,x+y+4,3\}\cap V_s^S|=1$. If $x+4\in V_s^S$ or $x+y+4\in V_s^S$, then arguing as in Case 1, we arrive to contradictions. Assume $3\in V_s^S$. Since $3\in V_{s+x+4}$ and $v_{s}\cdot v_{s+x+4}=0$, either $i\in V_{s+x+4}^S$ or $1\in V_s^S$, but not both. In the former case, since $|v_{s}\cdot v_{s\pm1}|=1$, we have that $x+3,x+y+3\in V_s^S$, implying that $a_s\ge4$, which is a contradiction. In the latter case, $V_s^S=\{i,1,3\}$, implying that $v_s\cdot v_{s+1}=0$, which is a contradiction.\\

\noindent\uline{Case 2(b)}: Consider the second form and label the vertices  $v_{s+1}',\ldots,v_n',v_1',\ldots,v_{s-1}'$. Notice that if $x=0$, then $T$ is of the form in Case 2(c). We will assume that $x\ge 1$ and handle the case $x=0$ in Case 2(c).
Since $p_5(S)=0$ and $p_2(S)=3$, we have $2\not\in V_s^S$ and $|\{x+4,x+y+4,1\}\cap V_s^S|=1$ (where $x+y+4=3$ if $y=0$). If $1\in V_s^S$, then since $v_{s+x+2}\cdot v_s=0$, we necessarily have that $i\in V_{s+x+2}^S$. Now, since $|v_{s+1}\cdot v_s|=1$, we have that $x+3\in V_s^S$ and so $V_s^S=\{i,1,x+3\}$; but then $|v_s\cdot v_{s+2}|=1$, which is a contradiction. If $x+4\in V_s^S$, then since $v_s\cdot v_{s+\alpha}=0$ for all $2\le \alpha\le x$, it follows that $4,\ldots,x+3\not\in V_s^S$. Since $x+4\in V_{s+x+3}^S$, we must have that $i\in V_{s+x+3}^S$; consequently, since $|v_{s-1}\cdot v_s|=1$, we necessarily have that $y\ge1$ and $x+y+3\in V_s^S$. But then $v_{s-2}\cdot v_s\neq0$, which is a contradiction. 
Thus $x+y+4\in V_s^S$. As above, it is easy to see that $3,x+5,\ldots,x+y+3\not\in V_s^S$. Since $x+y+4\in V_{s+x+1}^S$, it follows that either $i\in V_{s+x+1}^S$ or $4\in V_s^S$. In the former case, since $|v_s\cdot v_{s+1}|=1$, we have $x+3\in V_s^S$, which leads to a contradiction. In the latter case, since $4\in V_{s+x+3}^S$, we see that $i\in V_{s+x+4}^S$. Since $|v_s\cdot v_{s-1}|=1$, it follows that $x=1$. Thus $S$ is of the form\\

\noindent$S=\{\textcolor{blue}{e_i+e_4+e_{x+y+4}},e_5-e_4,e_4-e_2-e_3-\sum_{i=x+5}^{x+y+4}e_i,e_2+e_1,\textcolor{blue}{e_i}-e_2-e_4-e_5,$\\
\indent\indent\indent$e_2-e_1-e_3,e_3-e_{x+5},e_{x+5}-e_{x+6},\ldots,e_{x+y+3}-e_{x+y+4}\},$\\

\noindent which is positive cyclic with associated string $(3,2,3+y,2,4,3,2^{[y]})\in\mathcal{S}_{2d}$. \\

\noindent\uline{Case 2(c)}: Consider the third form and label the vertices  $v_{s+1}',\ldots,v_n',v_1',\ldots,v_{s-1}'$. As usual, since $p_5(S)=0$, $2\not\in V_s^S$. Notice $2\in V_{s+k+1}^S\cap V_{s+k+2}^S$. By our standing assumption that $v_{s(i)}\cdot v_{s(i')}=0$ for all $i,i'\in\mathcal{I}^S$, we necessarily have that either $1\in V_s^S$ or $4\in V_s^S$, but not both. Consequently, since $v_s\cdot v_{s+k+1}=v_s\cdot v_{s+k+2}=0$, either $i\in V_{s+k+1}^S$ or $i\in V_{s+k+2}^S$.  Moreover, since $p_2(S)=3$,  $j_1\not\in V_s^S$ and so $j_2\in V_s^S$. Now since $j_2\in V_{s+2}^S$ and $v_s\cdot v_{s+2}=0$, we necessarily have that $k=2$ and $4\in V_s^S$. Hence $V_s^S=\{4,i,{j_2}\}$, $i\in V_{s+k+2}^S$, and 
$T$ has associated string of the form $(2,3+x,2,2,3,2^{[x-1]},3)$. Moreover, $v_s\cdot e_{j_2}=\pm v_{s-1}\cdot e_{j_2}=\mp v_{s+1}\cdot e_{j_2}$. Thus $S$ is negative and has associated string of the form $(3,2,3+x,2,3,3,2^{[x-1]},3)\in\mathcal{S}_{1e}$.\\

\noindent\textbf{\uline{Case 3: $I(T)=-3$}}\\
\indent Finally assume $I(T)=-3$ (so that $I(S)=-2$). By Proposition \ref{liscasummary1}, $p_1(T)=1$, $p_2(T)=1$, and $p_j(T)=0$ for all $j\ge 4$. Thus $p_j(S)=0$ for all $j\ge5$.
Let $l$ be the unique integer such that $|E^T_l|=1$ and let $u$ be the integer such that $E_l^T=\{u\}$, where $u\neq s\pm1$. Then since $p_1(S)=0$, $l\in V_s^S$. Since $a_s=3$, we have that $p_2(S)\in\{2,3\}$ and $p_4(S)=p_2(S)-2$. By Lemma \ref{lem:p2+p4}, $p_2(S)+p_4(S)=2p_2(S)-2\equiv 2\mod4$, implying that $p_2(S)=2$ and $p_4(S)=0$. 
By Proposition \ref{liscasummary2}(1), there is an integer $k$ such that $E_k^{T}=\{s-1,s+1\}$ and $v_{s-1}\cdot e_k=-v_{s+1}\cdot e_k$. Since $p_2(S)=2$, $k\in V_s^S$, and so $V_s^S=\{i,l,k\}$. Since $k\not\in V_u^S$, we must have that $i\in V_u^S$. Thus $a_u=3$. Now, by Proposition \ref{liscasummary2}(1), $T$ has associated string $(b_1,\ldots,b_k,2,c_l,\ldots,c_1)$, where the middle entry `2' corresponds to the square of $v_u'$. Now, since $v_{s-1}\cdot e_k=-v_{s+1}\cdot e_k$, we have that $v_s\cdot e_k=\pm v_{s-1}\cdot e_k=\mp v_{s+1}\cdot e_k$ and so $S$ is negative and has associated string of the form $(3,b_1,\ldots,b_k,3,c_l,\ldots,c_1)\in\mathcal{S}_{1c}$.
\end{proof}

\begin{prop} Let $S$ be cyclic, $I(S)\le0$, $p_1(S)=0$, $p_2(S)>0$, and $n\ge4$. If $|\mathcal{J}|=0$, then $S$ is negative and has associated string in $\mathcal{S}_{1d}\cup\mathcal{O}$.
	\label{prop:allas=2}
\end{prop}

\begin{proof}
Note that $|\mathcal{I}|=p_2(S)$. By Lemma \ref{lem2b.1}, $a_{t(i)}\ge3$ for all $i\in\mathcal{I}$. If there exist $i_1,i_2\in\mathcal{I}$ such that $v_{s(i_1)}\cdot v_{s(i_2)}\neq0$, then by Lemma \ref{lem:noadjacent}, $S$ is negative with associated string in $\mathcal{S}_{1d}$. Now assume that $v_{s(i_1)}\cdot v_{s(i_2)}=0$ for all $i_1,i_2\in\mathcal{I}$. Then by Lemmas \ref{keylem1} and \ref{lem:noadjacent}, $p_4(S)=p_2(S)$, $I(S)=0$, and $p_j(S)=0$ for all $j\not\in\{2,3,4\}$.
Let $G=(S\setminus\{v_{s(i)},v_{t(i)}\,|\,i\in\mathcal{I}\})\cup \{\pi_{e_i}(v_{t(i)})\,|\,i\in\mathcal{I}\}$ and set $v_{t(i)}'=\pi_{e_i}(v_{t(i)})$ for all $i\in\mathcal{I}$, $v_x':=v_x$ for all $x\not\in\{s(i),t(i)\,|\,i\in\mathcal{I}\}$, and $a'_x=-(v_x')^2$ for all $x$.
Then $p_2(G)=p_4(S)=0$, $I(G)=0$, $p_3(G)=n-p_2(G)$, and by Lemma \ref{lem:noadjacent}, $G$ has $|\mathcal{I}|$ components. Finally, since for each $i\in\mathcal{I}$, there exists an integer $j(i)$ such that $E^S_{j(i)}=\{s(i)-1,s(i),s(i)+1,t(i)\}$, $G$ is irreducible and hence a good subset. 

Assume $C$ is a component of $G$ of length at least 2. After possibly relabeling, let $C=\{v_1',\ldots,v_m'\}$. Since $v'_1\cdot v'_2=1$, by Lemma \ref{lem2b.5}, there is an integer $k\in V^G_1\cap V^G_2$ such that $v'_1\cdot e_k=-v'_2\cdot e_k$. Since $|E^G_k|=3$, there exists an integer $z$ such that $k\in V^G_z$. Since $v_1'$ is a final vertex, $v_z'\cdot v_1'=0$ and so there exists an integer $l\in V^G_1\cap V^G_z$. Moreover, since $|E_l^G|=3$, we necessarily have that $a_1'\ge3$. 
We claim that if $a_z'=2$, then $v_z'=v_3'$. If $v_z'\neq v_3'$, then it is clear that $v_z'$ must be isolated. In this case, since $v_z'\cdot v_2'=0$, we have $l\in V_2^G$ and $v_1'\cdot e_l=-v_2'\cdot e_l$. Since $v_1'\cdot v_2'=1$, there exists another integer $m\in V_1^G\cap V_2^G$ and so $a_1',a_2'\ge3$. Let $L=(G\setminus\{v_1',v_2'\})\cup\{\pi_{e_k}(v_1'), \pi_{e_k}(v_2')\}$; then $L$ is good and $p_1(L)=1$. By Corollary 3.5 in \cite{liscalensspace}, $I(L)=-3$; but it is clear that $I(L)=I(G)-2=-2$, which is a contradiction. 

Thus if $a_z'=2$, then $v_z'=v_3'$ and we can perform a contraction yielding the subset $G'=G \setminus \{v_1',v_2',v_3'\}\cup \{\pi_{e_k}(v_1'),v_2'+v_3'\}$. Notice that $G'$ is a good subset with $I(G')=0$ and $p_j(G')=0$ for all $j\neq 3$; moreover, the component $C'=\{\pi_{e_k}(v_1'),v_2'+v_3,v_4'\ldots,v_m'\}$ has length one less than the length of $C$.
On the other hand, if $a_z'\ge3$, then we can perform a contraction yielding the subset $G''=G\setminus \{v_1',v_2',v_z'\}\cup \{v_1'+v_2',\pi_{e_k}(v_z')\}$. As above, $G''$ is a good subset with $I(G'')=0$ and $p_j(G'')=0$ for all $j\neq 3$, and the component $C''$ resulting from $C$ has length one less than the length of $C$. We may continue performing contractions in this way until the component $C$ is reduced to an isolated vertex. We can similarly perform contractions on all remaining components until they are all isolated vertices. We obtain a good subset $K$ that contains only isolated vertices. By Lemma \ref{lem:4comps}, $K$ is of the form
\begin{itemize}
\item $K=\{e_1-e_2+e_3-e_4,e_1+e_2,-e_1+e_2+e_3-e_4,e_3+e_4\}$ or
\item $K=\{e_1-e_2-e_3,e_1+e_2-e_4,e_2-e_3+e_4,e_1+e_3+e_4\}.$
\end{itemize}
It is easy to see that no expansion of either subset exists. Thus $K=G$.
Moreover, by construction, $|\mathcal{I}|=4$ and we may assume that $1=j(i_1)$, $2=j(i_2)$, $3=j(i_3)$, and $4=j(i_4)$, where $\mathcal{I}=\{i_1,i_2,i_3,i_4\}$. Thus (up to the action of $\Aut\ZZ^8$), $S$ is either of the form 
\begin{itemize}
\item $S=\{e_1-e_2+e_3-e_4\textcolor{blue}{-e_{i_2}+e_{i_3}},\textcolor{blue}{e_{i_1}-e_1},e_1+e_2,\textcolor{blue}{e_{i_2}-e_2},$\\
\indent\indent\indent\indent\indent\indent\indent\indent\indent\indent\indent\indent$-e_1+e_2+e_3-e_4\textcolor{blue}{-e_{i_1}-e_{i_4}},\textcolor{blue}{e_{i_3}-e_3},e_3+e_4,\textcolor{blue}{e_{i_4}-e_4}\}$ or
\item  $S=\{e_1-e_2-e_3\textcolor{blue}{-e_{i_2}},\textcolor{blue}{e_{i_1}-e_1},e_1+e_2-e_4\textcolor{blue}{-e_{i_4}},\textcolor{blue}{e_{i_2}-e_2},$\\
\indent\indent\indent\indent\indent\indent\indent\indent\indent\indent\indent\indent\indent$e_2+e_3+e_4\textcolor{blue}{+e_{i_3}},\textcolor{blue}{e_{i_4}-e_4},e_1+e_3+e_4+\textcolor{blue}{e_{i_1}}, \textcolor{blue}{e_{i_3}-e_3}\}.$
\end{itemize}

\noindent Thus $S$ is negative cyclic with associated string $(6,2,2,2,6,2,2,2)$ or $(4,2,4,2,4,2,4,2)$, both of which are in $\mathcal{O}$. \end{proof}	

\noindent To summarize, we have proven the following.

\begin{prop} Let $S$ be a cyclic subset with $p_1(S)=0$, $p_2(S)>0$ and $I(S)\le0$. Then $S$ is positive with associated string in $\mathcal{S}_{2c}\cup\mathcal{S}_{2d}$ or negative with associated string in $\mathcal{S}_{1c}\cup\mathcal{S}_{1d}\cup\mathcal{S}_{1e}\cup\mathcal{O}\cup\{(2^{[n]}\,|\,n\ge2)\}$.
	\label{prop:p2>0}
\end{prop}


\section{Appendix}\label{appendix}

Given a sequence of integers $(a_1,\ldots,a_n)$ the (Hirzebruch-Jung) continued fraction expansion is given by
\begin{center}
	$[a_1,\ldots,a_n]=\displaystyle a_1-\frac{1}{a_2-\displaystyle\frac{1}{\cdot\cdot\cdot-\displaystyle\frac{1}{a_n}}}$
\end{center}\bigskip
If $a_i\ge2$ for all $1\le i\le n$, then this fraction is well-defined and the numerator is greater than the denominator. In fact, for coprime $p>q>0\in\mathbb{Z}$, there exists a unique continued fraction expansion $[a_1,\ldots,a_n]=\frac{p}{q}$, where $a_i\ge2$ for all $1\le i\le n$. Moreover, by reversing the order of the continued fraction, we have $[a_n,\ldots,a_1]=\frac{p}{q'}$, where $q'$ is the unique integer such that $1\le q'<p$ and $qq'\equiv 1\text{mod}p$.

\begin{lem} Let $\frac{p}{q}=[a_1,\ldots,a_n]$, $\frac{s}{r}=[a_1,\ldots,a_{n-1}]$, and $\textbf{a}=(a_1,\ldots,a_n)$.\\ Then $|\textup{Tor}(H_1(\textbf{T}_{\pm A(\textbf{a})}))|=p-(r\pm2)$.\label{order}
\end{lem}

\begin{proof}
	Let $\textbf{a}=(a_1,\ldots,a_n)$. By Theorem 6.1 of \cite{neumann}, hyperbolic torus bundles are of the form $\textbf{T}_{\pm A(\textbf{a})}=T^2\times[0,1]/(\textbf{x},1)\sim(\pm A\textbf{x},0)$, where $$A=A(\textbf{a})=\begin{pmatrix}
	p&q\\
	-s&-r\\
	\end{pmatrix}, \frac{p}{q}=[a_1,...,a_n], \textrm{ and } \frac{s}{r}=[a_1,...,a_{n-1}].$$
	Note that, since $A\in SL_2(\ZZ)$, we have that $qs-pr=1$. Moreover, since $\textbf{T}_{\pm A(\textbf{a})}$ is hyperbolic, $\text{tr} A(\textbf{a})=p-r>2$. Now, by Lemma 10 in \cite{sakuma}, $|\textup{Tor}(H_1(\textbf{T}_{\pm A(\textbf{a})}))|=|\text{tr}(\pm A(\textbf{a}))-2|=|\pm(p-r)-2|=|\pm (p-(r\pm2))| =p-(r\pm2)$.
\end{proof}

\begin{lem} Let $(b_1,\ldots,b_k)$ and $(c_1,\ldots,c_l)$ be linear-dual strings, where $l+k\ge 2$, let $x\ge 1$ be an integer, and let $[b_1,\ldots,b_k]=\frac{p}{q}$. Then $\displaystyle[b_1,\ldots,b_k,x+1,c_l,\ldots,c_1]=\frac{xp^2}{xpq+1}$ and $\displaystyle[c_1,\ldots,c_l,x+1,b_k,\ldots,b_1]=\frac{xp^2}{xp^2-xpq+1}$.\label{contfraccalc}
\end{lem}

\begin{proof} First assume $[b_1,\ldots,b_k]=\frac{p}{q}$ and $[b_1,\ldots,b_k,x+1,c_l,\ldots,c_1]=\frac{xp^2}{xpq+1}$. Then since $(xpq+1)(xp^2-xpq+1)=xp^2(xpq-q^2+1)+1$, we have that $[c_1,\ldots,c_l,x+1,b_k,\ldots,b_1]=\frac{xp^2}{xp^2-xpq+1}$. We will now prove that $[b_1,\ldots,b_k,x+1,c_l,\ldots,c_1]=\frac{xp^2}{xpq+1}$.
	
	Let $n=k+l+1$ be the length of $(b_1,\ldots,b_k,x+1,c_l,\ldots,c_1)$. We proceed by induction on $n$. If $n=3$, then $k=1$, $l=1$, $(b_1)=\frac{2}{1}$, and $[2,x+1,2]=\frac{4x}{2x+1}=\frac{x2^2}{x(2)(1)+1}$. Suppose the lemma is true for all length $n-1$ continued fractions and consider $[b_1,\ldots,b_k,x+1,c_l,\ldots,c_1]$. By definition of linear-dual strings, either $b_1=2$ and $c_1\ge 3$ or vice versa. 
	
	Assume that $b_1=2$. Then the strings $(b_2,\ldots,b_k)$ and $(c_1-1,\ldots,c_l)$ are linear-dual and by the inductive hypothesis, $$[b_2,\ldots,b_k,x+1,c_l,\ldots,c_1-1]=\frac{xm^2}{xmn+1}\text{ and}$$ $$[c_1-1,c_2,\ldots,c_l,x+1,b_k,\ldots,b_2]=\frac{xm^2}{xm^2-xmn+1},$$
	where $[b_2,\ldots,b_k]=\frac{m}{n}.$ Thus $$[c_1,c_2,\ldots,c_l,x+1,b_k,\ldots,b_2]=1+\frac{xm^2}{xm^2-xmn+1}=\frac{2xm^2-xmn+1}{xm^2-xmn+1}.$$ 
	Since $(2xmn-xn^2+2)(xm^2-xmn+1)=(2xm^2-xmn+1)(xmn-xn^2+1)+1$, we have that $$[b_2,\ldots,b_k,x+1,c_l,\ldots,c_1]=\frac{2xm^2-xmn+1}{2xmn-xn^2+2}.$$ 
	Thus $$[b_1,\ldots,b_k,x+1,c_l,\ldots,c_1]=2-\frac{2xmn-xn^2+2}{2xm^2-xmn+1}=\frac{x(2m-n)^2}{x(2m-n)m+1}$$
	$$\text{and }[b_1,\ldots,b_k]=2-\frac{n}{m}=\frac{2m-n}{m}.$$ 
	Setting $p=2m-n$ and $q=m$ yields the result.
	
	Now suppose $c_1=2$. Then $(b_1-1,\ldots,b_k)$ and $(c_2,\ldots,c_l)$ are linear-dual and: $$[b_1-1,\ldots,b_k,x+1,c_l,\ldots,c_2]=\frac{xm^2}{xmn+1}\text{ and}$$ $$[c_2,\ldots,c_l,x+1,b_k,\ldots,b_1-1]=\frac{xm^2}{xm^2-xmn+1},$$ where $[b_1-1,\ldots,b_k]=\frac{m}{n}$. Thus $$[c_1,\ldots,c_l,x+1,b_k,\ldots,b_1-1]=2-\frac{xm^2-xmn+1}{xm^2}=\frac{xm^2+xmn-1}{xm^2}.$$ 
	Since $(xmn+xn^2+1)xm^2=(xm^2+xmn-1)(xmn+1)+1$, we have that $$[b_1-1,\ldots,b_k,x+1,c_l,\ldots,c_2,c_1]=\frac{xm^2+xmn-1}{xmn+xn^2+1}.$$ 
	Thus $$[b_1,\ldots,b_k,x+1,c_l,\ldots,c_2,c_1]=1+\frac{xm^2+xmn-1}{xmn+xn^2+1}=\frac{x(m+n)^2}{x(m+n)n+1}$$  
	$$\text{and }[b_1,\ldots,b_k]=1+\frac{m}{n}=\frac{m+n}{n}.$$
	 Setting $p=m+n$ and $q=n$ yields the result.
\end{proof}

\begin{prop} Let $[b_1,\ldots,b_k]=\frac{p}{q}$ and let $\textbf{a}=(a_1,\ldots,a_n)\in\mathcal{S}_{1a}$. Then $|\textup{Tor}(H_1(\textbf{T}_{-A(\textbf{a})}))|=p^2$.\label{1(a)orderb}
\end{prop}

\begin{proof}	
	Let $\textbf{a}=(2,b_1,\ldots,b_k,2,c_l,\ldots,c_1)$, where $(b_1,\ldots,b_k)$ and $(c_1,\ldots,c_l)$ are linear-dual (up to cyclic reordering). By Lemma \ref{contfraccalc}, $[b_1,\ldots,b_k,2,c_l,\ldots,c_1]=\frac{p^2}{pq+1}$ and so $$\displaystyle[2,b_1,\ldots,b_k,2,c_l,\ldots,c_1]=2-\frac{pq+1}{p^2}=\frac{2p^2-pq-1}{p^2}.$$ By Proposition \ref{order}, $|\textup{Tor}(H_1(\textbf{T}_{-A(\textbf{a})}))|=|2p^2-pq-1-(\alpha-2)|$, where $\alpha$ is the denominator of $[2,b_1,\ldots,b_k,2,c_l,\ldots,c_2]$. By Lemma \ref{contfraccalc}, 
	$$[c_1,\ldots,c_l,2,b_k,\ldots,b_1]=\frac{p^2}{p^2-pq+1}$$  and so 
	$$[c_2,\ldots,c_l,2,b_k,\ldots,b_1]=\frac{p^2-pq+1}{(c_1-1)p^2-c_1pq+c_1}.$$
	Thus  
	$$[b_1,\ldots,b_k,2,c_l,\ldots,c_2]=\frac{p^2-pq+1}{s}\text{ for some $s$}.$$
	Now, it is clear that $\alpha=p^2-pq+1$ and so $|\textup{Tor}(H_1(\textbf{T}_{-A(\textbf{a})}))|=|2p^2-pq-1-(p^2-pq+1-2)|=p^2$.
\end{proof}

\begin{lem}
Let $$A=(a_{ij})=\begin{bmatrix} -a_1 & 1 &&(-1)^t&&& (-1)^r \\
1 & -a_2 & &&& &  \\ 
& &\ddots &1&&&\\ 
(-1)^t & &1&-a_k& 1&& \\ 
 &  & &  1&\ddots &&\\ 
& & &&&-a_{n-1}& 1\\ 
(-1)^r & &&&&1&-a_{n}\end{bmatrix}.$$

\noindent Suppose $a_i\ge2$ for all $1\le i\le n$, $a_1\ge3$, $a_k\ge3$, and $r,t\in\{0,1\}$.
\begin{enumerate}[(a)]
	\item If $r=1$ or $t=1$, then $A$ is negative-definite.
	\item If $r=t=0$ and either $a_1\ge 4$, $a_k\ge4$, or there exists an integer $i\not\in\{1,k\}$ such that $a_i\ge 3$, then $A$ is negative-definite.
\end{enumerate}
\label{lem:negdef1}
\end{lem}

\begin{proof} Let $s_i=\sum_{j=1}^n a_{ij}$ be the $i$th row sum of $A$. Then $s_i\le 0$ for all $i$. Moreover, since either $a_1\ge 4$, $a_k\ge4$, or there exists an integer $i\not\in\{1,s\}$ such that $a_i\ge3$, there exists a row sum that is strictly less than $0$. Let $w\in\ZZ^n$. Then

\begin{equation*}
\begin{split}
w^TAw
&=\sum_{i,j}a_{ij}w_iw_j=\frac{1}{2}\sum_{i,j}a_{ij}(w_i^2+w_j^2-(w_i-w_j)^2)\\
&=\sum_{i,j}a_{ij}w_i^2-\sum_{i<j}a_{ij}(w_i-w_j)^2
=\sum_{i} s_iw_i^2-\sum_{i<j}a_{ij}(w_i-w_j)^2
\end{split}
\end{equation*}

First suppose $r=t=0$. Then every term in the above expression is at most zero and so $w^TAw\le0$. Moreover, if either: $a_1\ge 4$; $a_k\ge4$; or there exists an integer $i\not\in\{1,k\}$ such that $a_i\ge 3$, then one of the row sums $s_i$ is strictly less than 0. In this case, $w^TAw=0$ if and only if $w=0$. Thus $A$ is negative-definite. Next suppose $r=1$ and $t=0$. Then $s_1,s_n\le -2$ and so:

\begin{equation*}
\begin{split}
w^TAw
&=s_1w_1^2+s_nw_n^2+(w_1-w_n)^2+\sum_{i\neq 1,n} s_iw_i^2-\sum_{\substack{i<j\\(i,j)\neq(1,n)}}(w_i-w_j)^2\\
&\le -2w_1^2-2w_n^2+(w_1-w_n)^2+\sum_{i\neq 1,n} s_iw_i^2-\sum_{\substack{i<j\\(i,j)\neq(1,n)}}(w_i-w_j)^2\\
&=-(w_1+w_n)^2+\sum_{i\neq 1,n} s_iw_i^2-\sum_{\substack{i<j\\(i,j)\neq(1,n)}}(w_i-w_j)^2.
\end{split}
\end{equation*}

Each term in this expression is clearly negative. If $w^TAw=0$, then from the first term we have that $w_1=-w_n$. From the terms in the last summand, we have that $w_1=w_2=\cdots=w_n$. Hence $w_n=-w_n$, implying that $w_1=\cdots=w_n=0$. Therefore, $A$ is negative-definite.
We obtain a similar result if $r=0$ and $t=1$. Finally assume $r=t=1$. Then $s_1\le -4$ and $s_k,s_n\le -2$. Arguing as above, we have that

\begin{equation*}
\begin{split}
w^TAw
&=s_1w_1^2+s_kw_k^2+s_nw_n^2+(w_1-w_n)^2+(w_1-w_k)^2+\sum_{i\neq 1,k,n} s_iw_i^2-\hspace{-.5cm}\sum_{\substack{i<j\\(i,j)\neq(1,n),(1,k)}}\hspace{-.5cm}(w_i-w_j)^2\\
&\le-(w_1+w_n)^2-(w_1+w_k)^2+\sum_{i\neq 1,n} s_iw_i^2-\hspace{-.5cm}\sum_{\substack{i<j\\(i,j)\neq(1,n),(1,k)}}\hspace{-.5cm}(w_i-w_j)^2.
\end{split}
\end{equation*}

\noindent Once again, we can see that $A$ is necessarily negative-definite.
\end{proof}

\begin{lem}
Let $$A=\begin{bmatrix} -a_1 &2 & &&& (-1)^r\\
2 & -a_2 & 1 &&& \\ 
& 1 & -a_3 & && \\ 
&  & &\ddots   &&\\ 
 &  & &&-a_{n-1}& 1\\ 
(-1)^r &  & &&1&-a_{n}\end{bmatrix}$$

\noindent Suppose $a_i\ge2$ for all $1\le i\le n$, $a_1\ge3$, $a_2\ge3$, and $r\in\{0,1\}$. 
\begin{enumerate}[(a)]
	\item If $r=1$, then $A$ is negative-definite.
	\item If $r=0$ and either: $a_1\ge 4$; $a_2\ge4$; or there exists an integer $i\not\in\{1,k\}$ such that $a_i\ge 3$, then $A$ is negative-definite.
\end{enumerate}
\label{lem:negdef2}
\end{lem}

\begin{proof} The proof is similar to the proof of Lemma \ref{lem:negdef1}.
\end{proof}

\bibliographystyle{plain}
\bibliography{Bibliography}

\end{document}